\documentclass[letterpaper,11pt,onecolumn]{article}
\usepackage[margin=1in]{geometry}

\PassOptionsToPackage{dvipsnames}{xcolor}
\usepackage[T1]{fontenc}
\usepackage[latin1]{inputenc}
\usepackage[english]{babel}

\usepackage{amsmath}
\usepackage{amsbsy}
\usepackage{amsfonts}
\usepackage{amssymb}
\usepackage{amsthm}
\usepackage{array}
\usepackage{bm}
\usepackage{booktabs}
\usepackage{csquotes}
\usepackage{enumerate}
\usepackage{float}
\usepackage{graphicx}
\usepackage{mathrsfs}
\usepackage{pgf}
\usepackage{hyperref}

\usepackage[font={footnotesize}]{caption}
\usepackage{subcaption}

\hypersetup{colorlinks=false}

\def\be{\begin{equation}}
\def\ee{\end{equation}}
\def\ba{\begin{array}}
\def\ea{\end{array}}
\def\bea{\begin{eqnarray}}
\def\eea{\end{eqnarray}}
\def\beas{\begin{eqnarray*}}
\def\eeas{\end{eqnarray*}}

\newcommand{\bfHdiv}{\bfH({\rm div},\Omega)}
\newcommand{\half}{\frac{1}{2}}

\newcommand{\eq}[1]{(\ref{eq:#1})}

\newcommand{\bs}{\boldsymbol}

\newcommand{\mc}[1]{\mathcal{#1}}

\newcommand{\envE}{\boldsymbol{\mathsf E}}
\newcommand{\envH}{\boldsymbol{\mathsf H}}

\newcommand{\bi}{{\mathrm{i}}}

\newcommand{\bfEt}{{\widetilde{\bfE}}}
\newcommand{\Edreit}{{\widetilde{E}_3}}
\newcommand{\bfHt}{{\widetilde{\bfH}}}
\newcommand{\Hdreit}{{\widetilde{H}_3}}

\newcommand{\bfe}{\boldsymbol{e}}
\newcommand{\bff}{\boldsymbol{f}}
\newcommand{\bfg}{\boldsymbol{g}}

\newcommand{\bfn}{\boldsymbol{n}}

\newcommand{\bfu}{\boldsymbol{u}}
\newcommand{\bfv}{\boldsymbol{v}}

\newcommand{\bfE}{\boldsymbol{E}}
\newcommand{\bfF}{\boldsymbol{F}}
\newcommand{\bfG}{\boldsymbol{G}}
\newcommand{\bfH}{\boldsymbol{H}}

\newcommand{\ez}{\boldsymbol e_z}

\newcommand{\eps}{{\epsilon}}

\newcommand{\ome}{{\omega}}

\newcommand{\Ome}{{\Omega}}

\newcommand{\bfpsi}{\boldsymbol{\psi}}

\newcommand{\ptl}{{\partial}}

\newcommand{\bfnab}{\boldsymbol{\nabla}}

\newcommand{\doubleIN}{\mathbb{N}}
\newcommand{\doubleIR}{\mathbb{R}}
\newcommand{\doubleIC}{\mathbb{C}}

\newcommand{\ds}{\displaystyle}

\catcode `@=11
\newenvironment{numberedproof}[1]{{\bf Proof #1:}}{{}\hfill{\hbox{$\Box$}}\par\bigskip}

\newcounter{exercisenumber}

\renewcommand{\theexercisenumber}{\thesection.\arabic{exercisenumber}}

\DeclareMathOperator{\curl}{curl}

\let\div\relax
\DeclareMathOperator{\div}{div}

\newcommand{\dtn}{\operatorname{DtN}}
\newcommand{\dtnm}{\operatorname{DtN}^{\operatorname{mw}}}
\newcommand{\fraka}{\mathfrak a}

\newcommand{\IP}{I_{\mathrm{prop}}}
\newcommand{\IE}{I_{\mathrm{eva}}}
\renewcommand{\Re}{\operatorname*{Re}}
\renewcommand{\Im}{\operatorname*{Im}}
\newcommand{\Gammai}{\Gamma_{\mathrm{in}}}
\newcommand{\Gammal}{\Gamma_{\mathrm{lat}}}
\newcommand{\Gammao}{\Gamma_{\mathrm{out}}}
\newcommand{\Gammam}{\Gamma_{-}}

\newcommand{\alphahat}{\hat\alpha_i}
\newcommand{\betahat}{\hat\beta_j}

\def\tmaxwell{\operatorname*{mw}}
\def\thelm{\operatorname*{helm}}

\newcommand*\diff{\mathop{}\!\mathrm{d}}

\newcommand{\JMM}[1]{{\color{black}#1}}

\let\tilde\widetilde

\numberwithin{equation}{section}
\newtheorem{theorem}{Theorem}[section]

\newtheorem{remark}[theorem]{Remark}
\newtheorem{lemma}[theorem]{Lemma}
\newtheorem{example}[theorem]{Example}

\graphicspath{{figures/}}

\newcommand{\eremk}{\hbox{}\hfill\rule{0.8ex}{0.8ex}}

\begin{document}
\title{
Stability Analysis for Electromagnetic Waveguides. 
Part 1: Acoustic and Homogeneous Electromagnetic Waveguides}
\author{Jens M.~Melenk\thanks{Technische Universit\"at Wien, melenk@tuwien.ac.at}
	\and Leszek Demkowicz\thanks{Oden Institute, The University of Texas at Austin, $\{$leszek, stefan$\}$@oden.utexas.edu}
        \and Stefan Henneking\footnotemark[2]
        }

% Title options:
% Convergence of the DPG Method for Waveguide Problems
% 1. Convergence of Full Envelope DPG Method for Waveguide Problems
% 2. Convergence of the DPG Method for Waveguide Problems with Application to Fiber Optics
% 3. On the Stability of the DPG Method for Waveguide Problems -- Part 1: Homogeneous Waveguide
% 4. On the Stability of the DPG Method for Waveguide Problems with Application to Fiber Optics -- Part 1: Homogeneous Waveguide

\maketitle

\begin{abstract}
In a time-harmonic setting, we show for heterogeneous acoustic and homogeneous electromagnetic wavesguides stability estimates with the stability
constant depending linearly on the length $L$ of the waveguide. These stability estimates are used for the analysis of the (ideal) ultraweak (UW) variant 
of the Discontinuous Petrov--Galerkin (DPG) method. For this UW DPG, we show that the stability deterioration with $L$ can be countered by suitably
scaling the test norm of the method. We present the ``full envelope approximation'', a UW DPG method based on non-polynomial ansatz functions that allows for treating long waveguides. 
\end{abstract}

{\bf keywords:} 
waveguide, Helmholtz equation, Maxwell's equation, non-reflecting boundary conditions, stability 

%\begin{keyword}
%waveguide, Helmholtz equation, Maxwell's equation, non-reflecting boundary conditions, stability 
%\end{keyword}
%\begin{AMS}
%78A50, %Antennas, Wave guides 
%35Q60, %PDEs in connection with optics and electromagnetic theory
%35J05  %laplacian operator, reduced wave equation (Helmholtz equation)
%65N30 % FEM
%
%\end{AMS}

% Introduction
\section{Introduction}
\label{sec:introduction}
\paragraph{Motivation}
Acoustic and electromagnetic (EM) waveguide problems have many important applications and are therefore discussed widely in the literature \cite{kawano2004waveguide, snyder1983optical, agrawal, griffiths, jackson}. For many applications of interest, such as optical fibers, the propagating wave has high frequency and the waveguide length, denoted by $L$ throughout this work, is very large compared to the wavelength. It is therefore challenging to approximate the solution numerically; to obtain an accurate solution, one must sufficiently resolve the wavelength scale and additionally counter the effect of numerical pollution to overcome stability issues of the discretization \cite{BabuskaSauter97, ernst2012difficult}.

% Some FE methods such as the Discontinuous Petrov--Galerkin (DPG) Finite Element (FE) Method with Optimal Test Functions can circumvent the stability problem and deliver a robust discretization for any wavenumber \cite{DPG4, DemkowiczGopalakrishnanMugaZitelli12, Petrides_Demkowicz_21}, but they do not eliminate the pollution error in multiple dimensions.

While the stability of finite element (FE) discretizations of Helmholtz and time-harmonic Maxwell problems has been analyzed in fixed domains for increasing wave frequency $\omega$ \cite{melenk2011conv, DemkowiczGopalakrishnanMugaZitelli12, melenk2020maxwell}, there is to the best of our knowledge no such corresponding analysis for the waveguide problem where $\omega$ is fixed and the waveguide length $L$ increases. In practical applications, this is of great relevance as the available computational tools become more powerful thereby enabling numerical solution of waveguide models of realistic length scales. We discuss the present work in the context of modeling optical fiber amplifiers but emphasize that the main results of this work are relevant to FE discretization of acoustic and EM waveguide problems with the Discontinuous Petrov--Galerkin (DPG) Method \cite{DPG_Encyklopedia_18} \emph{in general.}

\paragraph{Optical amplifiers}
Optical fiber amplifiers can produce highly coherent laser outputs with great efficiency, which has enabled advances in many engineering applications \cite{jauregui2013fiber}. However, at high-power operation, these fiber laser systems are susceptible to the onset of various nonlinear effects that are adverse to the beam quality of the laser \cite{agrawal, kobyakov2010sbs, smith1972optical}. One particular challenge is mitigating the effects of heating of the silica-glass fiber. Under sufficient heat load, the fiber amplifier experiences a thermally-induced nonlinear effect called the transverse mode instability (TMI) \cite{eidam2011experimental, jauregui2020tmi}. TMI is characterized by a sudden reduction of the beam coherence above a certain power threshold. 
%Specifically, there is a transition from a stable beam to a chaotic energy transfer between the fiber's guided modes in time. 
This instability is a major limitation for the average power scaling of fiber laser systems \cite{jauregui2020tmi}. While a scientific consensus on the thermal origins of TMI has developed over the past years, finding effective mitigation strategies that do not incite other power limiting nonlinearities remains an active field of research in fiber optics.

In the context of studying TMI and other nonlinear effects in fibers, numerical simulations play an important role. Typically, a model needs to account for two fields in the fiber amplifier: 1)~the \emph{signal laser}, a highly coherent light source that is seeded into the fiber core; and 2)~the \emph{pump field}, which provides the energy for amplification of the signal and is typically injected into the fiber cladding. A variety of different models are employed in fiber amplifier simulations (e.g., \cite{saitoh2001bpm, ward2013bpm, naderi2013tmi, goswami2021fiber} and references therein). These models are usually derived from the time-harmonic Maxwell equations, and by making additional modeling assumptions they become easier to discretize and compute than a vectorial Maxwell problem.

\paragraph*{Vectorial Maxwell fiber amplifier model}
This work is part of a continued effort to build reliable, high-fidelity FE models for investigating TMI in optical amplifiers \cite{Nagaraj_Grosek_Petrides_Demkowicz_MoraPaz_19, Henneking_Demkowicz_21, Henneking_Grosek_Demkowicz_21, Henneking_phd, Henneking_Grosek_Demkowicz_22}. The model consists of a system of two nonlinear time-harmonic Maxwell equations (one for the signal field and one for the pump field) coupled with each other and with the transient heat equation to account for thermal effects \cite{Nagaraj_Grosek_Petrides_Demkowicz_MoraPaz_19, Henneking_Grosek_Demkowicz_21}. Modeling of a 1--10~m long fiber segment involves the solution with $\mc{O}(10^6)$ to $\mc{O}(10^7)$ wavelengths. Solving such a problem with a direct FE discretization is infeasible, even on state-of-the-art supercomputers. Hence, our initial efforts focused on so-called \emph{equivalent short fiber models}, which artificially scale physical parameters of the model, involving first $\mc{O}$(100) wavelengths (using OpenMP parallelization) \cite{Nagaraj_Grosek_Petrides_Demkowicz_MoraPaz_19, Henneking_Grosek_Demkowicz_21}, and then, more realistic models of (a tiny segment of) the actual fiber with up to $\mc{O}$(10,000) wavelengths \cite{Henneking_Demkowicz_21, Henneking_phd, Henneking_Grosek_Demkowicz_22} using MPI+OpenMP parallelization and up to 512 manycore compute nodes.\footnote{These fiber models have been implemented in the $hp$3D finite element library \cite{Henneking_24_hp3d} and were tested on the \emph{Frontera} supercomputer at the Texas Advanced Computing Center (TACC). Each compute node on TACC's \emph{Frontera} system is equipped with two Intel Cascade Lake processors and a total of 56 cores.}

Because the laser light in a meter-long optical fiber has millions of wavelengths, it is extremely difficult to resolve the wavelength scale of the propagating light for the full length. For this reason, even simplified models typically resolve a longer length scale. In the context of TMI studies, it is common to resolve only the length scale of the \emph{mode beat} between the fundamental mode and higher-order modes since the mode instabilities occur at that scale. In a typical weakly-guiding, large-mode-area fiber amplifier, the mode beat length is on the order of $\mc{O}$(1,000) wavelengths.

This brought forth the idea of the \emph{full envelope approximation}, the solution to an alternative formulation of our vectorial Maxwell model in which the field is less oscillatory in the (longitudinal) $z$-direction. Consider the linear time-harmonic Maxwell equations in a non-magnetic and dielectric medium, in the absence of free charges:
\begin{align}
	\bfnab \times \bfE &= -\bi \omega \mu_0 \bfH \, ,
	\label{eq:time-harmonic-1} %\\
& 
	\bfnab \times \bfH &= \bi \omega \eps \bfE \, ,
	%\label{eq:time-harmonic-2}
\end{align}
where $\bi = \sqrt{-1}$, $\omega$ is the angular frequency of the light, $\mu_0$ denotes permeability in vacuum, and $\eps = \eps(x,y,z)$ denotes permittivity.
The idea of the full envelope approximation is very simple: Instead of solving for the original EM field $(\bfE, \bfH)$ using 
%\eq{time-harmonic-1}--\eq{time-harmonic-2}, we seek the solution in the form:
\eq{time-harmonic-1} we seek the solution in the form:
\begin{align}
	\bfE(x,y,z) &= \envE(x,y,z) e^{-\bi kz} \, ,
	\label{eq:envelope-ansatz-1} %\\
& 
	\bfH(x,y,z) &= \envH(x,y,z) e^{-\bi kz} \, ,
%	\label{eq:envelope-ansatz-2}
\end{align}
where the envelope wavenumber $k \in \doubleIR$ corresponds to a dominant frequency in the waveguide $z$-direction. If the effective wavenumber of the propagating fields $(\bfE, \bfH)$ in the $z$-direction is indeed close to $k$, then the approximation of the new fields ($\envE$, $\envH$) requires orders of magnitude fewer elements along the fiber than the approximation of the original fields $(\bfE, \bfH)$. Upon substituting the ansatz into Maxwell's equations and factoring out the exponential, we obtain the modified Maxwell model
%\begin{subequations}
	%\label{eq:envelope} \\
\begin{align}
	\label{eq:envelope} 
	\bfnab \times \envE - \bi k \ez \times \envE &= -\bi \omega \mu_0 \envH \, , 
& 
%	\label{eq:envelope-1} \\
	\bfnab \times \envH - \bi k \ez \times \envH &= \bi \omega \eps \envE \, ,
%	\label{eq:envelope-2}
\end{align}
%\end{subequations}
where $\ez = (0,0,1)^T$ is the unit vector in $z$-direction. This idea has turned out to be very successful enabling modeling of TMI in fiber amplifiers with several million wavelengths \cite{Henneking_Grosek_Demkowicz_24}.

\paragraph*{Non-homogeneous acoustic waveguide}
Acoustical waveguides are of interest on their own, but in this work we treat the acoustic problem mostly as a stepping stone for the Maxwell waveguide problem. The two problems, albeit very similar, represent a surprisingly different level of difficulty. For a start, the presented analysis of the acoustic problem, based on the standard spectral theorem for self-adjoint operators, is possible for the non-homogeneous case, whereas, for the Maxwell problem, this is possible only for the case of a homogeneous waveguide (see \cite{Demkowicz_Melenk_Badger_Henneking_24} for a partial result for non-homogeneous Maxwell waveguides). We mention additionally that the common assumption in the optics community about the electric field being approximately divergence free, leads to the vector Helmholtz equation for the electric field, a problem directly related to the acoustics case.

\paragraph*{Scope of this work}
The presented paper provides theoretical foundations for the full envelope model and the convergence of the ultraweak DPG method for acoustic and EM waveguide problems with increasing waveguide length $L$. Numerical studies to this extent were carried out in \cite{Henneking_Demkowicz_21}; however, the convergence of the method was not analyzed therein.

We begin by introducing in Section~\ref{sec:model-problems} two model problems of interest: a (possibly non-homogeneous) acoustic waveguide, and a homogeneous EM waveguide. In the same section, we provide a quick overview of DPG analysis and demonstrate that the modified Maxwell model resulting from the full envelope ansatz shares the boundedness below constant with the original Maxwell operator.
%Here, we focus on the classical (homogeneous) waveguide problem involving a single Helmholtz or Maxwell equation. Section~\ref{sec:model-problems} introduces the Helmholtz and Maxwell model problems for the homogeneous waveguide with the corresponding ultraweak DPG formulations, and it discusses the stability of the full envelope approximation.
Section~\ref{sec:helmholtz} analyzes the Helmholtz problem (main result: Theorem~\ref{thm:stability-helmholtz}), and Section~\ref{sec:maxwell} analyzes the Maxwell problem (main result: Theorem~\ref{thm:stability-maxwell}). Section~\ref{sec:numerical} presents numerical results and concludes with a discussion of the non-homogeneous EM waveguide problem.

\paragraph*{Notation} We use $(\cdot,\cdot)_{L^2}$ and $\|\cdot\|_{L^2}$ to denote the standard sesquilinear form (antilinear in the second argument) and associated norm 
on the Hilbert space $L^2$. If clear from the context, we write $(\cdot,\cdot)$ and $\|\cdot\|$ respectively. For Sobolev spaces on Lipschitz domains $A \subset \doubleIR^d$, $d \in \{1,2,3\}$,
we follow the conventions of \cite{mclean00}, i.e., for $s \ge 0$ we denote by $H^s(A)$ the closure of $C^\infty(\overline{A})$ under the norm $\|\cdot\|_{H^s(A)}$, 
$\widetilde{H}^s(A)$ is the closure of $C^\infty_0(A)$ under the norm $\|\cdot\|_{H^s(\doubleIR^d)}$, $H^{-s}(A)$ is the dual of $\widetilde{H}^s(A)$ and 
$\widetilde{H}^{-s}(A)$ the dual of $H^s(A)$. For 3D vector fields $\bfE$ on $A \subset \doubleIR^3$, we define on $\partial A$ with outer normal 
$\bfn$ the tangential trace $\gamma_t \bfE = \bfE \times \bfn$ and the tangential component $\pi_t \bfE = \bfn \times (\bfE \times \bfn)$.
$\bfH(\curl,A)$ denotes the space of $L^2(A)$ functions whose $\curl$ is in $L^2(A)$; 
we set $\bfH_0(\curl,A) = \{\bfE  \in \bfH(\curl,A)\,|\, \gamma_t \bfE = 0 \mbox{ on $\partial A$}\}$. For 
$A \subset \doubleIR^d$ we denote similarly by  
$\bfH(\div,A)$ the space of $L^2(A)$ functions whose $\div$ is in $L^2(A)$; 
we set $\bfH_0(\div,A) = \{\bfE  \in \bfH(\div,A)\,|\, \bfn \cdot \bfE = 0 \mbox{ on $\partial A$}\}$, with $\bfn$ being the outer normal vector. 
The constant $C>0$ may be different in each occurance but does not depend on critical parameters such as the length $L$
of the waveguide. The expression $A \lesssim B$ indicates the existence of $ C > 0$ such that $A \leq C B$ with the implied constant $C$ independent of critical parameters.

% Model problems and DPG
%----------------------------------
\section{Model problems and DPG formulation}
\label{sec:model-problems}
%----------------------------------
We present two model problems whose stability we analyze in the following sections.
Throughout this work, $D \subset \doubleIR^d$, $d \in \doubleIN$, is a bounded Lipschitz domain with a piecewise smooth boundary, 
$L > 0$, and we set $\Omega = D \times (0,L)$. Throughout, we assume $\ome > 0$ given but exclude certain
degenerate cases when $\omega$ coincides with an eigenvalue of a suitable transverse eigenvalue problem, 
viz., (\ref{eq:non-degeneracy}) for the acoustic case and (\ref{eq:non-degeneracy-maxwell})  for the EM problem.
%----------------------------------
\subsection{Helmholtz problem}
%----------------------------------
On $D$ let $a:D\rightarrow \operatorname{GL}({\mathbb R}^{d \times d})$ be a pointwise symmetric positive definite matrix
with $0 < \lambda_{min} \leq a \leq \lambda_{max} < \infty$ uniformly in $x \in D$. Set
\begin{equation*}
\fraka:= \left(\begin{array}{cc} a(x) & 0 \\ 0 & 1 \end{array}\right).
\end{equation*}
On $\Ome = D \times (0,L) \subset \doubleIR^{d+1}$ we consider%\footnote{do we want to consider $i\ome {\mathfrak b} p + \div \bfu$?}
\begin{subequations}
\label{eq:primal-helmholtz}
\begin{align}
\label{eq:primal-helmholtz-a}
\bi \ome \bfu + \fraka  \nabla p & = \bfg \quad \mbox {in $\Ome$}, \\
\label{eq:primal-helmholtz-b}
\bi \ome p  + \div \bfu & = f  \quad \mbox {in $\Ome$}, \\
\label{eq:primal-helmholtz-c}
p & = 0 \quad \mbox {on $\Gammai:= D \times \{0\}$}, \\
\label{eq:primal-helmholtz-d}
\bfu \cdot \bfn  & = 0 \quad \mbox {on $\Gammal:= \partial D \times \{0,L\}$}, \\
\bi\ome \bfu \cdot \bfn  + \dtn p & = 0 \quad \mbox{on $\Gammao:= D \times \{L\}$}.
\end{align}
\end{subequations}
Here, $\bfn$ denotes the outer normal vector on $\partial\Omega$. We describe the operator $\dtn$ in (\ref{eq:dtn}) in Section~\ref{sec:dtn}.
This operator ensures that waves are ``outgoing'' on $\Gammao$. That is, if one considers instead an infinite waveguide $D \times (0,\infty)$ such that the right-hand sides $f$, $\bfg$ vanish on $D \times (L,\infty)$, one requires that waves be going to the right. %To that end, we represent the operator $\dtn$ in terms of eigenfunctions of an operator.

We define the operator
\begin{equation}
\label{eq:Ahelm}
A^{\thelm} \left(\begin{array}{c} p \\
\bfu\end{array}\right) :=
\left(\begin{array}{c}
	\bi \ome \bfu + \fraka \nabla p \\
	\bi \ome p + \div \bfu \end{array}\right)
\end{equation}
with domain $D(A^{\thelm})$ that includes the boundary conditions and is given by
\begin{equation}
\begin{split}
D(A^{\thelm}) = \{(p,\bfu) \in H^1(\Omega) \times  & \bfHdiv\,|\,
 p|_{\Gammai} = 0, \ \bfu \cdot \bfn = 0 \mbox{ in $H^{-1/2}(\Gammal)$},\\
& \ \bi\ome \bfu \cdot \bfn + \dtn p = 0 \mbox{ in $\widetilde{H}^{-1/2}(\Gammao)$}\}. 
\end{split}
\label{eq:D(Ahelm)}
\end{equation}

%----------------------------------
\subsection{Maxwell problem}
%----------------------------------
Let $d = 2$. In a homogeneous medium, we consider the Maxwell system for the electric field $\bfE$ and the magnetic field $\bfH$ that satisfy
\begin{subequations}
\label{eq:primal-maxwell}
\begin{alignat}{2}
\label{eq:primal-maxwell-a}
\bfnab \times \bfE - \bi \omega \bfH & = \bff \
&& \mbox{in $\Omega  = D \times (0,L) \subset \doubleIR^3$}, \\
\label{eq:primal-maxwell-b}
\bfnab \times \bfH+ \bi \omega \bfE & = \bfg \
&& \mbox{in $\Omega \subset \doubleIR^3$}, \\
\label{eq:primal-maxwell-c}
\gamma_t \bfE & = 0 \ && \mbox{on $\Gammai$}, \\
\label{eq:primal-maxwell-d}
\gamma_t \bfE & = 0 \ && \mbox{on $\Gammal$}, \\
\label{eq:primal-maxwell-e}
\pi_t \bfH - \dtnm \bfE & = 0 && \mbox{ on $\Gammao$}. 
%\bfE, \bfH & \text{ outgoing} \ && \mbox{on $\Gammao$.}
\end{alignat}
\end{subequations}
%Here, $\gamma_t \bfE = \bfE \times \bfn$ is the tangential trace on $\partial\Omega$, $\pi_t \bfH = \bfn \times (\bfH \times \bfn)$ the tangential component, 
Here, the operator $\dtnm$ is described in Section~\ref{sec:maxwell} in (\ref{eq:dtnm}) and realizes that waves are ``outgoing'' at $z = L$. 
%We will discuss the condition to be ``outgoing'' in more detail in Section~\ref{sec:maxwell}.
Problem (\ref{eq:primal-maxwell}) can be written in operator form
\begin{equation}
\label{eq:Amaxwell}
A^{\tmaxwell} \left(\begin{array}{c} \bfE \\ \bfH\end{array}\right) =
\left(\begin{array}{c} \bff \\ \bfg \end{array}\right) \, ,
\qquad
A^{\tmaxwell} \left(\begin{array}{c} \bfE \\ \bfH\end{array}\right) :=
\left(
\begin{array}{c}
\nabla \times \bfE - i \ome \bf H \\
\nabla \times \bfH + i \ome \bfE
\end{array}
\right),  
\end{equation}
where the domain $D(A^{\tmaxwell})$ includes the boundary conditions and is given by 
\begin{align*}
D(A^{\tmaxwell}) & = \{(\bfE, \bfH) \in \bfH_{0,\Gammam}(\curl,\Ome) \times \bfH(\curl,\Ome)\,|\, \pi_t \bfH = \dtnm \bfE \mbox{ on $\Gammao$}\}, 
\end{align*}
with $\Gammam:= \partial \Ome \setminus \overline{\Gammao}$ and 
\begin{equation}
\label{eq:H0Oammaminus}
\bfH_{0,\Gammam} = \{\bfE \in \bfH(\curl,\Ome)\,|\, \gamma_t \bfE = 0 \quad \mbox{ on $\Gammam$}\}. 
\end{equation}

\subsection{Ultraweak DPG formulation and DPG essentials}
\label{sec:DPG}

%We shall denote the $L^2(\Omega)$-norm with $\Vert \cdot \Vert$ without any subindices.

Suppose we are given an injective closed operator representing a system of first-order linear Partial Differential Equations (PDEs),
\[
A\,  \colon\,  L^2(\Omega) \supset D(A) \to L^2(\Omega) \, ,
\]
where $D(A)$ is the domain of the operator incorporating (homogeneous) boundary conditions. We want to solve the problem,
$$
\mbox{ find $u \in D(A)$} \quad \mbox{ such that } \quad Au = f . 
$$
%\[
%\left\{
%\begin{array}{lll}
%u \in D(A) \\[5pt]
%Au = f \, .
%\end{array}
%\right.
%\]
The problem is trivially equivalent to the so-called {\em strong variational formulation:}
$$
\mbox{ find $u \in H_A(\Omega)$} \quad \mbox{ such that } \quad 
(Au,v) = (f,v) \quad  \forall \, v \in L^2(\Omega) \, ,
$$
%\[
%\left\{
%\begin{array}{lll}
%u \in H_A(\Omega)  \\[5pt]
%(Au,v) = (f,v) \quad  \forall \, v \in L^2(\Omega) \, ,
%\end{array}
%\right.
%\]
where $H_A(\Omega) := D(A)$ is a Hilbert space equipped with the graph norm: 
\[
\Vert u \Vert^2_{H_A} : = \Vert u \Vert^2 + \Vert A u \Vert^2 \, .
\]
The ultraweak (UW) variational formulation is obtained by integrating by parts and passing {\em all} derivatives to the test function. It reads as follows:
\begin{equation}
\label{eq:UWformulation}
\mbox{ find $u \in L^2(\Omega)$} \quad \mbox{ such that } \quad 
(u,A^\ast v) = (f,v) \quad  \forall \, v \in H_{A^\ast}(\Omega) \, ,
\end{equation}
%\be
%\left\{
%\begin{array}{lll}
%u \in L^2(\Omega) \\[5pt]
%(u,A^\ast v) = (f,v) \quad  \forall \, v \in H_{A^\ast}(\Omega) \, ,
%\end{array}
%\right.
%\label{eq:UWformulation}
%\ee
where
\[
A^\ast \, : \, L^2(\Omega) \supset D(A^\ast) \to L^2(\Omega)
\]
is the $L^2$-adjoint of the original operator in the sense of closed operator theory, assumed to be injective as well, and the test space $H_{A^\ast}(\Omega) := D(A^\ast)$ has been equipped with the scaled adjoint graph norm:
\[
\Vert v \Vert^2_{H_{A^\ast}} : = \beta^2 \Vert v \Vert^2 + \Vert A^\ast v \Vert^2
\]
with $\beta \geq 0$.
As the test functions are now more regular (and the solution less regular), the right-hand side of (\ref{eq:UWformulation}) may be upgraded to an arbitrary continuous and antilinear functional 
$l(\cdot)$ on $H_{A^\ast}(\Omega)$.

\begin{lemma}
\label{lem:uw-inf-sup}
Assume that that the operator $A$ is surjective and bounded below with a (boundedness below) constant $\alpha >0$:
\[
\alpha \Vert u \Vert \leq \Vert A u \Vert \qquad \forall u \in D(A) \, .
\]
The UW formulation~(\ref{eq:UWformulation}) is then well-posed with the inf-sup constant
\begin{equation}
\label{eq:lem:uw-inf-sup}
\inf_{0 \ne u \in H_A} \sup_{0 \ne v \in H_{A^\ast}} \frac{(Au,v)}{\|u\|\,  \|v\|_{H_{A^\ast}}} =:
\gamma \geq \left[ 1 + \left( \frac{\beta}{\alpha} \right)^2 \right]^{-\half} \, .
\end{equation}
\end{lemma}

\begin{proof}
The Closed Range Theorem for closed operators implies that adjoint $A^\ast$ is bounded below with the same constant $\alpha$
(see, e.g., \cite[Rem.~{5.19.1}]{oden-demkowicz-applied-FA}).
Take an arbitrary $u \in L^2(\Omega)$  and consider the corresponding solution $v_u$ of the adjoint problem,
$$
\mbox{ find $v_u \in D(A^\ast)$} \quad \mbox{ such that } \quad A^\ast v_u = u. 
$$
%\[
%\left\{
%\begin{array}{ll}
%v_u \in D(A^\ast) \\[5pt]
%A^\ast v_u = u \, .
%\end{array}
%\right.
%\]
We have:
\[
\Vert v_u \Vert \leq \alpha^{-1} \Vert u \Vert
\quad \Rightarrow \quad
\Vert v_u \Vert_{H_{A^\ast}}^2 =  \Vert A^\ast v_{\JMM{u}} \Vert^2 + \beta^2 \Vert v_{\JMM{u}} \Vert^2
\leq \left[ 1 + \left(\frac{\beta}{\alpha}\right)^2 \right] \Vert u \Vert^2 \, ,
\]
and, in turn,
\[
\sup_{v \in D(A^\ast)} \frac{\vert (u, A^\ast v) \vert}{\Vert v \Vert_{H_{A^\ast}}}
\geq \frac{\vert (u, u) \vert}{\Vert v_u \Vert_{H_{A^\ast}}} \geq \left[ 1 + \left(\frac{\beta}{\alpha}\right)^2 \right]^{-\half} \frac{\vert (u, u) \vert}{\Vert u \Vert}
= \left[ 1 + \left(\frac{\beta}{\alpha}\right)^2 \right]^{-\half} \Vert u \Vert \, .
\]

%Finally, the conjugate operator coincides with the strong operator $A$ and its injectivity implies
The sesquilinear form $b(u,v) := (u, A^\ast v)$ generates two bounded operators: the operator $B: L^2(\Omega) \to D(A^\ast)^\prime$ representing the operator form of the UW formulation, and its transpose $B^\prime: D(A^\ast) \to L^2(\Omega)$ that coincides with the adjoint $A^\ast$ of the closed operator $A$. The injectivity of $A^\ast$ implies the existence of the solution of (\ref{eq:UWformulation}) for any right-hand side $l \in H_{A^\ast}(\Omega)^\prime$.
\end{proof}

Note that, for $\beta = 0$, the inf-sup constant (\ref{eq:lem:uw-inf-sup}) is one.
As the continuity constant of $A^\ast:H_{A^\ast}\JMM{(\Omega)} \rightarrow L^2(\Omega)$ is one as well, we are dealing with a \emph{duality pairing}; the ideal Petrov--Galerkin Method with Optimal Test functions delivers simply the $L^2$-projection of the exact solution.
The DPG Method with Optimal Test Functions is based on extending the UW formulation to a larger class
of {\em broken} test functions $H_{A^\ast}(\Omega_h)$ \cite{Carstensen_Demkowicz_Gopalakrishnan_16}. Unfortunately, for $\beta = 0$, the form $b(u,v) = (u,A^\ast v)$ is not {\em localizable}, i.e., the adjoint test norm cannot be extended to the broken test space and {\em we must} use a positive scaling constant $\beta > 0$. In principle though, by employing a sufficiently small constant $\beta$, we can make the inf-sup constant as close to one as we wish. It is not intuitive at all that the UW formulation allows for {\em improving} the stability of the original operator $A$ in the strong setting.

Finally, the stability of the UW formulation {\em is inherited} by the {\em broken formulation} \cite{Carstensen_Demkowicz_Gopalakrishnan_16}. In the {\em ideal DPG} method, the optimal test functions are assumed to be computed exactly; the discrete inf-sup constant is bounded from below by the continuous inf-sup constant, which indicates the relevance of understanding the continuous variational problem. The ideal DPG method is merely a logical construct on the way to analyze the {\em practical DPG} method where the optimal test functions are computed approximately. The additional error resulting from the approximation of the test functions is analyzed with the help of appropriate Fortin operators \cite{DPG_Encyklopedia_18}.

\subsection{The full envelope approximation (\ref{eq:envelope})}

The analysis of the stability of the full envelope approximation method (\ref{eq:envelope}) turns out to be surprisingly simple.

\begin{lemma}
Let $\tilde{A}$ be the operator corresponding to the full envelope ansatz, i.e.,
\[
\tilde{A} \tilde{u} :=  e^{\bi kz} A(e^{-\bi kz} \tilde{u}) \, ,
\]
where $A \in \{A^{\thelm}, A^{\tmaxwell}\}$ denotes the operator corresponding to the acoustic or EM waveguide problem 
 and $ k \in \doubleIR$.
Then, the operator $\tilde{A}$ is bounded below if and only if the operator $A$ is bounded below, and the corresponding boundedness below constants are identical:
\[
 \Vert A u \Vert \ge \alpha \Vert u \Vert \quad \Leftrightarrow \quad \Vert \tilde{A} \tilde{u}\Vert  \geq \alpha \Vert \tilde{u}\Vert \, .
\]
\end{lemma}

\begin{proof} We observe 
\[
\Vert \tilde{A} \tilde{u}\Vert  = \Vert e^{\bi kz} A( e^{-\bi kz} \tilde{u})  \Vert = \Vert A( e^{-\bi kz} \tilde{u})  \Vert  \geq \alpha \Vert e^{-\bi kz} \tilde{u}  \Vert = \alpha \Vert \tilde{u}  \Vert \, .
\]
\end{proof}

% Helmholtz stability analysis
%--------------------------------------------
\section{Stability in the scalar acoustic case: the Helmholtz equation}
\label{sec:helmholtz}
%--------------------------------------------
Section~\ref{sec:DPG} shows that at the heart of the analysis of the UW DPG method
is the understanding of the stability properties of the operator $A$. In this section,
we provide the stability of operator $A^{\thelm}$ of the Helmholtz problem
(\ref{eq:primal-helmholtz}) making the dependence on the length $L$ of the waveguide explicit. 
In view of Lemma~\ref{lem:uw-inf-sup}, this analysis then provides a guideline for selecting 
the parameter $\beta$ in the DPG method. 

%--------------------------------------------
\subsection{Analysis of the 1D Helmholtz equation}
\label{sec:1d-helmholtz}
%--------------------------------------------
Since the domain $\Omega = D \times (0,L)$ has product structure and the coefficients of the Helmholtz equation
are independent of the longitudinal variable, a decoupling into transversal modes is possible and leads to
a stability analysis analysis of 1D equations. In the present section, we present the necessary 1D stability results. 
The core of the stability analysis is a parameter-explicit analysis of the 1D Helmholtz equation with 
Robin boundary conditions at the right endpoint and either Dirichlet or Neumann boundary conditions at the left endpoint. 
We flag at this point that this tool will be utilized in the analysis of the EM case as well.

The analysis of the two 1D Helmholtz problems is achieved by first studying them on the reference interval $\hat{I} = (0,1)$
in Lemma~\ref{lemma:1D-inf-sup} and then on the interval $I = (0,L)$ in Lemma~\ref{lemma:ihl} by scaling. In the interest of brevity, 
we formulate the stability simultaneously for Dirichlet (corresponding to the choice $V = H^1_{(0}(\hat I)$) and  
and Neumann conditions (corresponding to $V = H^1(\hat I)$) at $x = 0$ 
in the following Lemma~\ref{lemma:1D-inf-sup}:  
\begin{lemma}[\protect{\cite[Thm.~{4.3}]{MST20}}]
\label{lemma:1D-inf-sup}
Let $\hat I:= (0,1)$.
Let $\hat \kappa \in \doubleIC_{\ge  0}:= \{z \in \doubleIC\colon \Re \JMM{z} \ge 0\}$ satisfy $|\hat \kappa| \ge \kappa_0 > 0$.
Introduce
$H^1_{(0}(\hat I):= \{u \in H^1(\hat I)\colon u(0) = 0\}$ and the sesquilinear form
\begin{equation*}
a^{1D}_{\hat \kappa,\hat I}(u,v):= (u^\prime,v^\prime)_{L^2(0,1)} + \hat \kappa^2 (u,v)_{L^2(0,1)} + \hat \kappa u(1) \overline{v}(1).
\end{equation*}
Introduce on $H^1(\hat I)$ the norm $\|u\|^2_{1,|\hat \kappa|} := \|u^\prime\|^2_{L^2(\hat I)} + |\hat \kappa|^2 \|u\|^2_{L^2(\hat I)}$.
Then there is $c > 0$ depending only on $\kappa_0>0$ 
%such that for both choices $V = H^1(\hat I)$ and $V = H^1_{(0}(\hat I)$ there holds
such that for both choices $V \in \{H^1(\hat I), H^1_{(0}(\hat I)\}$ there holds
\begin{equation}
\label{lemma:1D-inf-sup-2}
\hat \gamma_{\hat \kappa} :=\inf_{0 \ne u \in V} \sup_{0 \ne V} \frac{\Re a^{1D}_{\hat \kappa,\hat I}(u,v)}{\|u\|_{1,|\hat \kappa|} \|v\|_{1,|\hat \kappa|}} \ge \frac{1}{1+c \frac{|\Im \hat \kappa|}{1+\Re \hat \kappa}}
\end{equation}
\end{lemma}
\begin{proof}
The proof follows essentially from \cite[Thm.~{4.3}]{MST20}, where a similar sesquilinear form in 2D and 3D is considered. Details are given
in Appendix~\ref{appendix:proof-of-lemma-1D-inf-sup}.
\end{proof}
By scaling, we can infer from Lemma~\ref{lemma:1D-inf-sup} stability estimates for 1D Helmholtz problems on domains of length $L > 0$. To do 
so in Lemma~\ref{lemma:ihl}, we 
introduce  for $\kappa \in \doubleIC$ and an interval $I = (0,L)$ the sesquilinear form 
$a^{1D}_{\kappa}$ and the norm $\|\cdot\||_{1,|\kappa|}$ by 
\begin{align}
a^{1D}_{\kappa}(u,v) &:= (u^\prime,v^\prime)_{L^2(I)} + \kappa^2 (u,v)_{L^2(I)} + \kappa u(L) \overline{v}(L), \\
\|q\|^2_{1,|\kappa|}& := \|q^\prime\|^2_{L^2(I)} + |\kappa|^2 \|q\|^2_{L^2(I)} \, .
\end{align}
Armed with this notation, we study 1D Helmholtz problems with either Dirichlet (corresponding to $V = H^1_{(0}(I)$) or Neumann (corresponding to $V = H^1(I)$) boundary conditions at the left endpoint $x = 0$:
\begin{lemma}
\label{lemma:ihl}
Let $I = (0,L)$ and $V = H^1(I)$ or $V = H^1_{(0}(I):= \{w \in H^1(I)\colon w(0) = 0\}$. Let $\kappa \in \doubleIC_{\ge 0}$.  
Fix $c_0 > 0$. 
Consider the following two problems:
Find $q_1$, $q_2 \in V$ such that
\[
\begin{alignedat}{3}
%(q^\prime_1,w^\prime)_{L^2(I)} + \kappa^2 (q_1,w)_{L^2(I)} + \kappa q_1(L) \overline{w(L)} & = (f,w)_{L^2(I)} \quad
a^{1D}_{\kappa}(q_1,w) &= (f,w)_{L^2(I)}  &\quad \forall w \in V \, , \\
%(q^\prime_2,w^\prime)_{L^2(I)} + \kappa^2 (q_2,w)_{L^2(I)} + \kappa q_2(L) \overline{w(L)} & = (f,w^\prime)_{L^2(I)} \quad
a^{1D}_{\kappa}(q_1,w) &= (f,w^\prime)_{L^2(I)}  &\quad \forall w \in V \, .
%\forall w \in V.
\end{alignedat}
\]
Then the following holds: 
\begin{enumerate}[(i)]
\item
\label{item:ihl-0} There are $C$, $c > 0$ depending only on $c_0$ such that if $L |\kappa| \ge c_0$ then with
\begin{equation}
\label{eq:lemma:ihl-inf-sup}
\gamma:= \frac{1}{1 + c \frac{|\Im (L \kappa)|}{1 + \Re (L \kappa)}}
\end{equation}
there holds
\begin{align*}
\|q_1\|_{1,|\kappa|} &\leq C \frac{1}{\gamma |\kappa|}  \|f\|_{L^2(I)} \, ,
&
\|q_2\|_{1,|\kappa|} &\leq C \gamma^{-1}  \|f\|_{L^2(I)} \, .
\end{align*}
\item
\label{item:ihl-i}
If $\kappa \in \bi \doubleIR$ and $|\kappa| L \ge c_0$, then for a constant $C > 0$ depending only on $c_0$
\begin{align*}
\|q_1\|_{1,|\kappa|} &\leq C L  \|f\|_{L^2(I)} \, ,
&
\|q_2\|_{1,|\kappa|} &\leq C L |\kappa| \|f\|_{L^2(I)} \, .
\end{align*}
\item
\label{item:ihl-ii}
If $\kappa >0$  and $\kappa L \ge c_0$, then
for a constant $C > 0$ depending only on $c_0$
\begin{align*}
\|q_1\|_{1,\kappa} &\leq C \kappa ^{-1}  \|f\|_{L^2(I)} \, ,
&
\|q_2\|_{1,\kappa }&\leq C \|f\|_{L^2(I)} \, .
\end{align*}
\end{enumerate}
\end{lemma}
\begin{proof}
Rescaling the equation posed on $(0,L)$ to an equation posed on $(0,1)=:\hat I$, we get by denoting the rescaled functions with a $\hat{\cdot}$
and $\hat \kappa := L \kappa $ and spaces $\hat V = H^1(\hat I)$ if $V = H^1(I)$ and $\hat V = H^1_{(0}(\hat I)$ if $V = H^1_{(0}(I)$: 
\begin{align*}
(\hat{q}^\prime_1,w^\prime)_{L^2(\hat I)} + \hat{\kappa}^2 (\hat{q}_1,w)_{L^2(\hat I)} + \hat{\kappa} \hat{q}_1(1) \overline{w(1)} & = L^2 (\hat{f},w)_{L^2(I)} \quad
\forall w \in \hat V, \\
(\hat{q}^\prime_2,w^\prime)_{L^2(\hat I)} + \hat{\kappa }^2 (\hat{q}_2,w)_{L^2(\hat I)} + \hat{\kappa}  \hat{q}_2(1) \overline{w(1)} & = L (\hat{f},w^\prime)_{L^2(I)} \quad
\forall w \in \hat V.
\end{align*}
In terms of the inf-sup constant $\hat \gamma_{\hat \kappa}$ of (\ref{lemma:1D-inf-sup-2}), i.e.,
\begin{equation*}
\gamma:= \inf_{q \in \hat V} \sup_{w \in \hat V} \frac{|a^{1D}_{\hat \kappa, \hat I}(q,w)|}{\|q\|_{1,|\hat{\kappa}|}\|w\|_{1,|\hat{\kappa}|}}
\end{equation*}
we have
\begin{align*}
\|\hat{q}_1\|_{1,|\hat{\kappa}|} &\leq 
 \frac{1}{\gamma} \sup_{v \in \hat{V}} \frac{|(L^2 \hat f, v)_{L^2(\hat I)}|}{\|v\|_{1,|\JMM{\hat \kappa}|}} 
\leq 
 \frac{1}{\gamma} \sup_{v \in \hat{V}} \frac{L^2 \|\hat f\|_{L^2(\hat I)} \|v\|_{L^2(\hat I)} }{\|v\|_{1,|\JMM{\hat \kappa}|}} 
\leq 
\frac{C} {\gamma  |\hat{\kappa}|} \|L^2 \hat f\|_{L^2(\hat I)}, \\
\|\hat{q}_2\|_{1,|\hat{\kappa}|} &\leq C \gamma^{-1}  \|L \hat f\|_{L^2(\hat I)}
\end{align*}
with $C > 0$ depending only on a lower bound for $|\hat{\kappa}|$. Scaling back to $(0,L)$ yields
\begin{align*}
\|q_1\|_{1,|\kappa|}
&\sim L^{-1/2} \|\hat{q}_1\|_{1,|\hat{\kappa}|}
\leq C L^2 L^{-1/2}  \gamma^{-1} |\hat{\kappa}|^{-1} \|\hat{f}\|_{L^2(\hat I)} \leq C L \gamma^{-1} |\hat{\kappa}|^{-1} \|f\|_{L^2(I)}, \\
\|q_2\|_{1,|\kappa|} & \sim L^{-1/2} \|\hat{q}_2\|_{1,|\hat{\kappa}|}
\leq C L^{1/2} \gamma^{-1}  \|\hat{f}\|_{L^2(\hat I)} \leq C \gamma^{-1}\|f\|_{L^2(I)} .
\end{align*}

\emph{Proof of (\ref{item:ihl-0}):}
By Lemma~\ref{lemma:1D-inf-sup}, $\gamma$ has the form (\ref{eq:lemma:ihl-inf-sup}). The result follows.

\emph{Proof of (\ref{item:ihl-i}):}
For $\kappa \in \bi \doubleIR$, the inf-sup constant $\gamma$ of (\ref{eq:lemma:ihl-inf-sup}) satisfies $\gamma \sim |\hat{\kappa }|^{-1}$
%by \cite[Thm.~{4.2}]{Ihlenburg}
with an implied constants depending on a lower bound for $|\hat{\kappa }|$. The result follows.

\emph{Proof of (\ref{item:ihl-ii}):}
For $\kappa >0$, the inf-sup constant $\gamma$ of (\ref{eq:lemma:ihl-inf-sup}) satisfies $\gamma = 1$.  
The result follows.
\end{proof}

%%--------------------------------------------
%\subsection{Helmholtz equation}
%%--------------------------------------------
%Consider a bounded Lipschitz domain $D \subset {\mathbb R}^d$, $d \in \doubleIN$, and $L > 0$. Set $\Ome:= D \times (0,L)$.
%On $D$ let $a:D\rightarrow \operatorname{GL}({\mathbb R}^{d \times d})$ be a pointwise SPD matrix
%with $0 < \lambda_{min} \leq a \leq \lambda_{max} < \infty$ uniformly in $x \in D$. Set
%\begin{equation*}
%\fraka:= \left(\begin{array}{cc} a(x) & 0 \\ 0 & 1 \end{array}\right).
%\end{equation*}
%On $\Ome$ we consider%\footnote{do we want to consider $i\ome {\mathfrak b} p + \div \bfu$?}
%\begin{subequations}
%\label{eq:primal-helmholtz}
%\begin{align}
%i \ome \bfu + \fraka  \nabla p & = f \quad \mbox {in $\Ome$} \\
%i \ome p  + \div \bfu & = \bfg  \quad \mbox {in $\Ome$} \\
%p & = 0 \quad \mbox {on $\Gamma_i:= D \times \{0\}$} \\
%\bfu \cdot \bfn  & = 0 \quad \mbox {on $\Gamma_l:= \partial D \times \{0,L\}$} \\
%i\ome \bfu \cdot \bfn  - \dtn p & = 0 \quad \mbox{ on $\Gamma_o:= D \times \{L\}$}.
%\end{align}
%\end{subequations}
%Here, $\bfn$ denote the outer normal vector. We describe the DtN-operator next. This operator ensure that
%waves are ``outgoing'' on $\Gamma_o$. That is, if one considers instead an infinite waveguide
%$D \times (0,\infty)$ such that the right-hand sides $f$, $\bfg$ vanish on $D \times (L,\infty)$ one requires that
%waves be going to the right. To that end, we reprepresent the operator $\dtn$ in terms of eigenfunctions of an operator.
%
%--------------------------------------------
\subsection{The \texorpdfstring{$\dtn$}{DtN} operator}
\label{sec:dtn}
%--------------------------------------------
We describe the $\dtn$-operator at the outflow boundary $\Gammao$. This is achieved in terms of a modal decomposition
obtained by a suitable eigenvalue problem in transverse direction.

Let $(\lambda_n,\varphi_n)_{n \in \doubleIN} \subset \doubleIR^+ \times H^1(D)\setminus\{0\}$ 
be the eigenpairs of the operator $u \mapsto -\div (a \nabla u)$, i.e.,
\begin{subequations}
\label{eq:helmholtz-eigenpairs}
\begin{alignat}{2}
-\div (a \nabla \varphi_n) & = \lambda_n \varphi_n \quad
&& \mbox{in $D$}, \\
\bfn \cdot a \nabla \varphi_n & = 0 \quad
&& \mbox{on $\partial D$}
\end{alignat}
\end{subequations}
with the normalization $\|\varphi_n\|_{L^2(D)} = 1$ for all $n \in \doubleIN$. Here, the normal 
vector $\bfn$ is the outer normal on $\partial D$. 
We have the orthogonalities
\begin{equation}
\label{eq:orthogonalities}
(\varphi_n,\varphi_m)_{L^2(D)}  = \delta_{nm},
\qquad
(a \nabla \varphi_n,\nabla \varphi_m)_{L^2(D)}  = \delta_{nm} \lambda_n.
\end{equation}
The positive values $\lambda_n$ are assumed sorted in ascending order and listed according to their multiplicity. We collect
them in the spectrum $\sigma = \{\lambda_n\colon n \in \doubleIN\} $.  It is convenient to derive the 
operator $\dtn$ in (\ref{eq:primal-helmholtz-d})
using the second-order formulation for $(f,\bfg) = (0,0)$, i.e., to consider 
\begin{align*}
-\div (\fraka \nabla p) - \ome^2 p & = 0 \quad \mbox{ on $D \times (L,\infty).$}
\end{align*}
Making the ansatz $p(x,z) = \sum_n p_n(z) \varphi_n(x)$ yields the 1D equations
\begin{align*}
-p_n^{\prime\prime}(z) + (\lambda_n - \ome^2) p_n & = 0 \quad\mbox{ on $(L,\infty)$}
\end{align*}
with the fundamental solutions $e^{\pm \kappa_n z}$, where
\begin{equation}
\label{eq:helmholtz-kappa_n}
\kappa_n:= \sqrt{\lambda_n -\ome^2}
\end{equation}
and the square root is taken to be the principal branch.
\JMM{Upon w}riting $p|_{\Gammao} = \sum_{n} p_n(L) \varphi_n(x)$, the ``outgoing'' solution is defined
on $(L,\infty)$ to be
$$
p(x,z) = \sum_{n} p_n(L) \varphi_n(x) e^{-\kappa_n (z-L)}.
$$
The operator $\dtn$ is given by $p|_{\Gammao} \mapsto (\partial_z p(x,z))|_{z = L}$, viz.,
\begin{align}
\label{eq:dtn}
\dtn p & := - \sum_{n} p_n \kappa_n \varphi_n(x),
\qquad p_n := (p,\varphi_n)_{L^2(D)}.
\end{align}
The indices $n \in \doubleIN$ with $\kappa_n > 0$ are called \emph{evanescent modes}, those with
$\kappa_n \in \bi \doubleIR\setminus\{0\}$ are the \emph{propagating modes}. We assume throughout that
\begin{equation}
\label{eq:non-degeneracy}
\kappa_n \ne 0 \qquad \forall n \in \doubleIN
\end{equation}
so that we can write $\doubleIN = \IP \dot\cup \IE$ with the finite set $\IP$ of propagating modes and the infinite set $\IE$ of evanescent modes:
\begin{equation}
\IP:= \{n \in \doubleIN\,|\, \kappa_n \in \bi \doubleIR\},
\qquad
\IE:= \{n \in \doubleIN\,|\, \kappa_n > 0\}.
\end{equation}
Concerning the mapping properties of $\dtn$, we have
\begin{align*}
\dtn: H^{1/2}(\Gammao) \rightarrow \widetilde{H}^{-1/2}(\Gammao):= (H^{1/2}(\Gammao))^\prime
\end{align*}
is bounded, linear, which follows from the representation of $\dtn$:
Since $H^{1/2}(\Gammao)$ is the interpolation space between $L^2(\Gammao)$ and $H^1(\Gammao)$ (see, e.g., \cite{mclean00}), 
we have the norm equivalence $\|p\|^2_{H^{1/2}(\Gammao)} \sim \sum_{n} |p_n|^2 \sqrt{\lambda_n}$, where 
$p = \sum_n p_n \varphi_n$ with $p_n = (p,\varphi_n)_{L^2(D)}$. 
Hence, writing $p$, $q \in H^{1/2}(\Gammao)$ in the form $p = \sum_{n} p_n \varphi_n$ and $q = \sum_{n} q_n \varphi_n$, we estimate 
%(convergence in $H^{1/2}(\Gammao)$, which means $\sum_{n} |p_n|^2 \sqrt{\lambda_n}  < \infty$, $\sum_{n} |q_n|^2 \sqrt{\lambda_n} < \infty$)
\begin{align*}
\left| \langle \dtn p, q\rangle\right|  & = \left| - \sum_{n} \kappa_n  p_n \overline{q}_n \right|
\lesssim \|p\|_{H^{1/2}(\Gammao)} \|q\|_{H^{1/2}(\Gammao)} \, .
\end{align*}
%%------------------------------------------
%\subsection{An auxiliary problem}
%%------------------------------------------
%We have to analyze
%\begin{align}
%A \left(\begin{array}{c} q \\ \bfv \end{array}\right)  = \left(\begin{array}{c} \bff \\ g \end{array}\right)
%\end{align}
%for given $\bff$, $g \in L^2(\Ome)$.
%
%-------------------------------------
\subsection{Stability estimates for \texorpdfstring{$A^{\thelm}$}{A-helm}}
%-------------------------------------
Introduce $H^1_{\Gammai}:= \{v \in H^1(\Omega)\,|\, v|_{\Gammai} = 0\}$. In the ensuing analysis, we will need the following observation
about the sequences $\{\lambda_n\}_n$  and $\{\kappa_n\}_n$:
Noting that there are only finitely many propagating modes, that $\lambda_n \rightarrow \infty$ for $n \rightarrow \infty$  
and that we assumed (\ref{eq:non-degeneracy}) we have, 
for a constant that depends on $\ome$, 
\begin{subequations}
\label{eq:lambda-tildelambda}
\begin{align}
\label{eq:lambda-tildelambda-a}
\max_{n \in \IP} \left( |\kappa_n|^{-1} + |\kappa_n|\right) & \leq C  ,\\
\label{eq:lambda-tildelambda-b}
\max_{n \in \IP} |\sqrt{\lambda_n}| & \leq C  ,\\
\label{eq:lambda-tildelambda-c}
\max_{n \in \doubleIN} \frac{|\sqrt{\lambda_n}|}{|\kappa_n|} & \leq C
\end{align}
\end{subequations}

%---------------------------
\begin{lemma}
\label{lemma:L2-estimate} 
Let $c_0 > 0$ and $L \ge c_0$, assume (\ref{eq:non-degeneracy}). 
There is $C > 0$ independent of $L$ (but possibly dependent on $\omega$) such that for 
$f \in L^2(\Omega)$ the solution $p \in H^1_{\Gammai}$ of
\begin{align}
\label{eq:lemma:L2-estimate-weak-form}
(\fraka \nabla p, \nabla v)_{L^2(\Ome)} - \ome^2 (p,v)_{L^2(\Ome)} -  \langle  \dtn p,v\rangle = (f,v)_{L^2(\Ome)}
\quad \forall v\in H^1_{\Gammai}
\end{align}
satisfies%, for a $C > 0$ independent of $L$ (but possibly dependent on $\omega$),  
\begin{align*}
\|p\|_{H^1(\Ome)} \leq C L \|f\|_{L^2(\Ome)}.
\end{align*}
\end{lemma}
\begin{proof}
Make the ansatz
$
%\displaystyle
p(x,z) = \sum_n p_n(z) \varphi_n(x)
$
and set
$\displaystyle
f_n(z):= (f(\cdot,z),\varphi_n)_{L^2(D)}.
$
By orthogonality properties of the functions $\{\varphi_n\}_n$, we have
\begin{align*}
\|p\|^2_{L^2(\Ome)} & = \sum_{n} \|p_n\|^2_{L^2(0,L)}, 
& 
\|\sqrt{a} \nabla_x p\|^2_{L^2(\Ome)} & = \sum_{n} \JMM{\lambda_n} \|p_n\|^2_{L^2(0,L)}, \\
\| \partial_z p\|^2_{L^2(\Ome)} & = \sum_{n}  \|p^\prime_n\|^2_{L^2(0,L)}, 
& 
\|f\|^2_{L^2(\Ome)} & = \sum_{n} \|f_n\|^2_{L^2(0,L)}.
\end{align*}
Testing (\ref{eq:lemma:L2-estimate-weak-form}) with $v(x,z) = v_n(z) \varphi_n(x)$ with arbitrary $v_n \in H^1_{(0}(0,L)$ 
(cf.\ Lemma~\ref{lemma:ihl}) 
gives due to the orthogonalities satisfied by the functions $\{\varphi_n\}_n$
$$
(p^\prime_n,v^\prime_n)_{L^2(0,L)} + \kappa_n^2 (p_n,v_n)_{L^2(0,L)} + \kappa_n p_n(L) \overline{v}_n(L) = (f_n,v_n)_{L^2(0,L)}. 
$$
From Lemma~\ref{lemma:ihl}, we conclude
\begin{align}
\label{eq:lemma:L2-estimate-10}
\|p_n\|_{1,|\kappa_n|} \leq C
\begin{cases}
L \|f_n\|_{L^2(0,L)} & \mbox{ if $n \in \IP$}, \\
\kappa_n^{-1} \|f_n\|_{L^2(0,L)} & \mbox{ if $n \in \IE$}.
\end{cases}
\end{align}
We arrive at
\begin{align*}
& \|\sqrt{a}\nabla_x p\|^2_{L^2(\Ome)}  = \sum_{n} \JMM{\lambda_n} \|p_n\|^2_{L^2(0,L)} \\
& \qquad \stackrel{(\ref{eq:lemma:L2-estimate-10})}{\lesssim} L^2 \sum_{n\in\IP} \frac{\JMM{\lambda}_n}{|\kappa_n|^2} \|f_n\|^2_{L^2(0,L)}
+ \sum_{n \in \IE} \frac{\JMM{\lambda}_n}{|\kappa_n|^4} \|f_n\|^2_{L^2(0,L)} 
%\\ & 
\stackrel{(\ref{eq:lambda-tildelambda})}{\lesssim} L^2 \|f\|^2_{L^2(\Ome)}, \\
& \|\partial_z p\|^2_{L^2(\Ome)} = \sum_{n} \|p^\prime_n\|^2_{L^2(0,L)}
\stackrel{(\ref{eq:lemma:L2-estimate-10})}{\lesssim}  L^2 \sum_{n\in\IP}  \|f_n\|^2_{L^2(0,L)}
+ \sum_{n \in \IE} |\kappa_n|^{-2} \|f_n\|^2_{L^2(0,L)} \\
& \qquad \lesssim L^2 \|f\|^2_{L^2(\Ome)}, \\
& \|p\|^2_{L^2(\Ome)}  = \sum_{n} \|p_n\|^2_{L^2(0,L)}
\stackrel{(\ref{eq:lemma:L2-estimate-10})}{\lesssim}  \! L^2 \! \sum_{n\in\IP} |\kappa_n|^{-2} \|f_n\|^2_{L^2(0,L)} 
+ \!\! \sum_{n \in \JMM{\IE}} |\kappa_n|^{-4} \|f_n\|^2_{L^2(0,L)} \\ 
& \qquad \lesssim L^2\|f\|^2_{L^2(\Ome)}
+ \sum_{n\in\IE} |\kappa_n|^{-4} \|f_n\|^2_{L^2(0,L)}
\lesssim L^2\|f\|^2_{L^2(\Ome)} \, .
\end{align*}
\end{proof}

\begin{lemma}
\label{lemma:L2div-estimate}
Let $c_0 > 0$ and $L \ge c_0$, assume (\ref{eq:non-degeneracy}).
There is $C > 0$ independent of $L$ (but possibly dependent on $\omega$) such that 
for $\bff \in L^2(\Ome)$ The solution $q \in H^1_{\Gammai}$ of
\begin{align}
\label{eq:lemma:L2div-estimate-weak-form}
(\fraka\nabla p, \nabla v)_{L^2(\Ome)} - \ome^2 (p,v)_{L^2(\Ome)} - \langle  \dtn p,v\rangle = (\bff,\nabla v)_{L^2(\Ome)}
\quad \forall v\in H^1_{\Gammai}
\end{align}
satisfies 
\begin{align*}
\|p\|_{H^1(\Ome)} \leq C L \|\bff\|_{L^2(\Ome)}.
\end{align*}
\end{lemma}

\begin{proof}
We write the vector $\bff$ as $\bff = (\bff_x,f_z)^\top$ with a vector-valued function $\bff_x$ and a scalar function $f_z$. By linearity of the problem,
we may consider the cases $(\bff_x,0)^\top$ and $(0,f_z)^\top$ as right-hand sides separately.
For $\bff_x = 0$, we proceed as in Lemma~\ref{lemma:L2-estimate} by writing
$f_z = \sum_{n} f_n(z) \varphi_n(x)$ and get with Lemma~\ref{lemma:ihl} for the corresponding functions $p_n$
\begin{align*}
\|p_n\|_{1,|\kappa_n|} & \leq C
\begin{cases}
L |\kappa_n| \|f_n\|_{L^2(0,L)}  & \mbox{ if $n \in \IP$}, \\
 \|f_n\|_{L^2(0,L)}  & \mbox{ if $n \in \IE$}
\end{cases} \\
&
\leq C L \|f_n\|_{L^2(0,L)} 
\end{align*}
since $\max_{n \in \IP} |\kappa_n| \leq C$.
We may repeat the calculations performed in Lemma~\ref{lemma:L2-estimate} to establish
\begin{align*}
\|\sqrt{a} \nabla_x p\|^2_{L^2(\Ome)} &=  \sum_{n} \lambda_n \|p_n\|^2_{L^2(0,L)}
\leq C L^2 \sum_{n} \frac{\lambda_n}{|\kappa_n|^2}  \|f_n\|^2_{L^2(0,L)}  \leq C L^2 \|f_z\|^2_{L^2(\Ome)}, \\
\|\partial_z p\|^2_{L^2(\Ome)} &=  \sum_{n} \|p^\prime_n\|^2_{L^2(0,L)}
\leq C L^2 \sum_{n}  \|f_n\|^2_{L^2(0,L)}  \leq C L^2 \|f_z\|^2_{L^2(\Ome)}, \\
\|p\|^2_{L^2(\Ome)} &=  \sum_{n} \|p_n\|^2_{L^2(0,L)}
\leq C L^2 \sum_{n}|\kappa_n|^{-2}   \|f_n\|^2_{L^2(0,L)}  \leq C L^2 \|f_z\|^2_{L^2(\Ome)}.
\end{align*}
For the case of the right-hand side $\bff = (\bff_x,0)^\top$, we define
\[
f_n(z):= (\bff_x(\cdot,z),\nabla \varphi_n)_{L^2(D)}  = (a^{-1} \bff_x(\cdot,z),a\nabla \varphi_n)_{L^2(D)}
\]
and note by the fact that the functions $\{\|\sqrt{a} \nabla \varphi_n\|_{L^2(D)}^{-1} \nabla \varphi_n\}_n$ are an orthonormal (with respect to
$(a \cdot, \cdot)_{L^2(D)}$) basis of its span that
$$
\sum_{n} \lambda^{-1}_n \|f_n\|^2_{L^2(0,L)} =
%\int_{0}^L \sum_{n} \frac{1}{\|\sqrt{a} \nabla \varphi_n\|^2_{L^2(D)}} |(a^{-1} \bff_x(\cdot,z),a\nabla \varphi_n)_{L^2(D)}|^2 \JMM{\diff z}
\int_{0}^L \sum_{n} \frac{ |(a^{-1} \bff_x(\cdot,z),a\nabla \varphi_n)_{L^2(D)}|^2 }
{\|\sqrt{a} \nabla \varphi_n\|^2_{L^2(D)}} 
\diff z
\leq \|a^{-1} \bff_x\|^2_{L^2(\Ome)}.
$$
We expand the solution $p$ as $p = \sum_{n} p_n(z) \varphi_n(x)$. Testing the equation with functions of the form $v_n(z) \varphi_n(x)$ yields again
an equation for the coefficients $q_n$:
\begin{equation*}
\kappa^2_n (p_n,v_n)_{L^2(0,L)} + (p^\prime_n, v^\prime_n)_{L^2(0,L)} + {\kappa}_n p_n(L) \overline{v}_n(L)  = (f_n,v_n)_{L^2(0,L)} \qquad \forall v_n \in H^1_{(0}(0,L).
\end{equation*}
By Lemma~\ref{lemma:ihl} we get
\begin{align*}
\|p_n\|_{1,|{\kappa}_n|} & \leq C
\begin{cases}
L \|f_n\|_{L^2(0,L)} & \mbox{ if $n \in \IP$ }, \\
|\kappa_n|^{-1} \|f_n\|_{L^2(0,L)} & \mbox{ if $n \in \IE$ }.
\end{cases}
\end{align*}
Hence,
\begin{align*}
& \|\sqrt{a} \nabla_x p\|^2_{L^2(\Ome)}  = \sum_{n} \lambda_n \|p_n\|^2_{L^2(0,L)} \\
& \qquad \leq C L^2 \sum_{n \in \IP} \frac{\lambda_n}{\JMM{|\kappa_n|^2}} \|f_n\|^2_{L^2(0,L)} + C \sum_{n \in \IE} \frac{\lambda_n}{|\kappa_n|^4} \|f_n\|^2_{L^2(0,L)}
\leq C L^2 \|\bff_x\|^2_{L^2(\Ome)}, \\ 
& \|\partial_z p\|^2_{L^2(\Ome)}  = \sum_{n} \|p^\prime_n\|^2_{L^2(0,L)} \\
& \qquad \leq C L^2 \sum_{n \in \IP} \|f_n\|^2_{L^2(0,L)} + C \sum_{n \in \IE} |\kappa_n|^{-2} \|f_n\|^2_{L^2(0,L)}
\leq C \|\bff_x\|^2_{L^2(\Ome)}  \\
& \|p\|^2_{L^2(\Ome)}  = \sum_{n} \|p_n\|^2_{L^2(0,L)} \\
& \qquad \leq C L^2 \sum_{n \in \IP} |\kappa_n|^{-2} \|f_n\|^2_{L^2(0,L)}
+ \sum_{n \in \IE} |\kappa_n|^{-4} \|f_n\|^2_{L^2(0,L)} \leq C \|\bff_x\|^2_{L^2(\Ome)}.
\end{align*}
Putting together the above results proves the claim.
\end{proof}

\begin{theorem}
\label{thm:stability-helmholtz}
Assume (\ref{eq:non-degeneracy}).
Let $c_0 > 0$.
There is a constant $C > 0$ (depending on $\fraka$ and $\ome$ but independent of $L \ge c_0$) such that for all $(f,\bfg) \in L^2(\Ome)$ the
problem (\ref{eq:primal-helmholtz}) has a unique solution $(p,\bfu) \in D(A^{\thelm}) $ with
\begin{equation}
\label{eq:thm:stability-helmholtz-10}
\|\bfu \|_{\bfH(\div,\Ome)} + \|p\|_{H^1(\Ome)} \leq C L \left[ \|f \|_{L^2(\Ome)} + \|\bfg\|_{L^2(\Ome)}\right].
\end{equation}
In particular, for all $(\bfu,p) \in D(A^{\thelm})$
\begin{equation}
\label{eq:thm:stability-helmholtz-20}
\|A^{\thelm} (p,\bfu)^\top\|_{L^2(\Ome)} \ge C L^{-1} \|(\bfu,p)\|_{L^2(\Ome)}.
\end{equation}
\end{theorem}

\begin{proof}
\emph{Proof of (\ref{eq:thm:stability-helmholtz-20}):}
%First, we note that $A^{\thelm}:D(A^{\thelm}) \rightarrow L^2(\Ome)$ is injective. Indeed, $A^{\thelm}(p,\bfu)^\top = 0$ implies
%$\bfu = -i\ome \nabla p$ and therefore $p \in H^1(\Ome)$ satisfies a homogeneous second-order equation. Together with the
%boundary condition, one checks that $p = 0$ so that also $\bfu = 0$.
%
Abbreviate for the two components of $A^{\thelm} (p,\bfu)^\top$
$$
\bfg := \bi \ome \bfu + \fraka \nabla p \in L^2(\Ome),
\qquad f := \div \bfu +\bi \ome p \in L^2(\Ome).
$$
Hence, $(p,\bfu)$ satisfy for smooth $(\tilde p,\tilde \bfu)$
\begin{align*}
(\bi \ome \bfu, \tilde \bfu)_{L^2(\Ome)} + (\fraka \nabla p, \tilde \bfu)_{L^2(\Ome)} & = (\bfg, \tilde \bfu)_{L^2(\Ome)} \, , \\
( \div \bfu,\tilde p)_{L^2(\Ome)} + (\bi \ome p, \tilde p)_{L^2(\Ome)} & = (f,\tilde p)_{L^2(\Ome) } \, .
\end{align*}
Considering $\tilde p$ with $\tilde p|_{\Gammai} = 0$ and using the boundary conditions satisfied by $(p,\bfu)$ (i.e., $(p,\bfu) \in D(A^{\thelm})$) yields
after an integration by parts
\begin{align*}
(\bi \ome \bfu, \tilde \bfu)_{L^2(\Ome)} + (\fraka \nabla p, \tilde \bfu)_{L^2(\Ome)} & = (\bfg, \tilde \bfu)_{L^2(\Ome)} \, , \\
-(\bfu, \nabla \tilde p)_{L^2(\Ome)} + (\bi\ome p, \tilde p)_{L^2(\Ome)} - \frac{1}{\bi\ome}\langle\dtn p, \tilde p\rangle_{\Gammao}  & = (f,\tilde p)_{L^2(\Ome) } \, .
\end{align*}
Selecting $\tilde \bfu = -\frac{1}{\bi\ome} \nabla \tilde p$ yields
\begin{align*}
(\bfu, \nabla \tilde p)_{L^2(\Ome)} + \frac{1}{\bi\ome} (\fraka \nabla p, \nabla \tilde p)_{L^2(\Ome)} & = \frac{1}{\bi\ome} (\bfg, \nabla \tilde p)_{L^2(\Ome)} \, , \\
-(\bfu, \nabla \tilde p)_{L^2(\Ome)} + (\bi\ome p, \tilde p)_{L^2(\Ome)} - \frac{1}{\bi\ome}\langle\dtn p, \tilde p\rangle_{\Gammao}  & = (f,\tilde p)_{L^2(\Ome) } \, ,
\end{align*}
so that, by adding these two equations and multiplying by $\bi\ome$, we arrive at
\begin{align}
\label{eq:thm:stability-helmholtz-30}
\begin{aligned}
(\fraka \nabla p, \nabla \tilde p)_{L^2(\Ome)} & -\ome^2  (p,\tilde p)_{L^2(\Ome)} - \langle \dtn p,\tilde p\rangle_{\Gammao} \\
 & =  (\bfg,\nabla \tilde p)_{L^2(\Ome)} + \bi \ome (f,\tilde p)_{L^2(\Ome)}
\qquad \qquad \forall \tilde p \in H^1_{\Gammai} \, .
\end{aligned}
\end{align}
From Lemmas~\ref{lemma:L2-estimate}, \ref{lemma:L2div-estimate} we infer
\begin{align*}
\|p\|_{H^1(\Ome)} &\leq C L \left[ \|\bfg\|_{L^2(\Ome)} + \|f\|_{L^2(\Ome)}\right] ,
\end{align*}
which in turn yields
$$
\|(p,\bfu)\|^2_{L^2(\Ome)} \leq \|p\|^2_{L^2(\Ome)} + \|\bfu\|^2_{L^2(\Ome)}
\leq \|p\|^2_{L^2(\Ome)} + 2 \ome^{-2} \|\nabla p\|^2_{L^2(\Ome)} + 2 \ome^{-2} \|\bfg \|^2_{L^2(\Ome)} \, .
$$
In total, we arrive at $\|(p,\bfu)\|_{L^2(\Ome)} \leq C L \left[\|\bfg \|_{L^2(\Ome)} + \|f\|_{L^2(\Ome)} \right]$, i.e.,
$\|(p,\bfu)\|_{L^2(\Ome)} \leq C L \|A^{\thelm}(p,\bfu)^\top\|_{L^2(\Ome)}$, which is (\ref{eq:thm:stability-helmholtz-20}).

\emph{Proof of (\ref{eq:thm:stability-helmholtz-10}):}
To see solvability for any $(f,\bfg) \in L^2(\Omega)$, we reverse the above arguments. Lemmas~\ref{lemma:L2-estimate}, \ref{lemma:L2div-estimate}
imply solvability of (\ref{eq:thm:stability-helmholtz-30}) for $p \in H^1_{\Gammai}$. Next, setting
$\bfu:= (\bi \ome)^{-1}( \bfg - \fraka \nabla p)$, one infers from (\ref{eq:thm:stability-helmholtz-30}) that
$\bfu \in \bfH(\div,\Ome)$ with $\div\bfu = f - \bi \ome p \in L^2(\Ome)$. This implies the estimate (\ref{eq:thm:stability-helmholtz-10}).
To see that $(p,\bfu) \in D(A^{\thelm})$, we infer, using that $\div \bfu = f - \bi \ome p$, that
$\bfu \cdot \bfn = 0$ in $H^{-1/2}(\Gammal)$ and $\bi \ome \bfu \cdot \bfn  + \dtn p = 0$ in $H^{-1/2}(\Gammao)$.
To see that in fact $\bfu \cdot \bfn = 0$ in $\widetilde{H}^{-1/2}(\Gammal)$, we note that
$\bfu \cdot \bfn \in H^{-1/2}(\partial\Ome\setminus \overline{\Gammai})$
with $\operatorname{supp} \bfu \cdot \bfn \subset \overline{\Gammao}$. Hence, by \cite[Thm.~{3.29}]{mclean00},
we have $\bfu \cdot \bfn \in \widetilde{H}^{-1/2}(\Gammao)$. Since by construction
$\dtn p \in \widetilde{H}^{-1/2} (\Gammao)$, the difference $z := \bi \ome \bfu \cdot \bfn + \dtn p \in \widetilde{H}^{-1/2}(\Gammao)$.  By the density of $C^\infty_0(\JMM{\Gammao})$ in $H^{1/2}(\JMM{\Gammao})$ (cf.\ \cite[Thm.~{11.1}]{lions-magenes72}) and $z = 0$ in $H^{-1/2}(\JMM{\Gammao})$, 
we conclude that actually $z = 0$ in $\widetilde{H}^{-1/2}(\JMM{\Gammao})$. 
%\hrule 
%Since $z|_{\Gammao} = 0$, we conclude that the zero extension $\widetilde{z}$ to $H^{-1/2}(\partial\Ome\setminus \overline{\Gammai})$
%is a distribution on $\partial\Ome\setminus \overline{\Gammai}$ with
%$\operatorname{supp} z \subset \partial\Gammao$. However, since $H^{1/2}(\partial\Ome\setminus\overline{\Gammai})$-functions
%do not admit traces on lower-dimensional manifolds, the support of the distribution $z \in H^{-1/2}(\partial\Ome\setminus\overline{\Gammai})$
%%must be empty. Hence, $z = 0$ in $H^{-1/2}(\partial\Ome\setminus\overline{\Gammai})$, i.e., $z = 0$ in $\widetilde{H}^{-1/2}(\Gammao)$.
This shows that $(p,\bfu) \in D(A^{\thelm})$.
\end{proof}
%---------------------------------------------

% Maxwell stability analysis
%------------------------------------------------
\section{Stability for Maxwell's equations}
\label{sec:maxwell}
%------------------------------------------------
The stability analysis of the Maxwell system (\ref{eq:primal-maxwell}) proceeds similarly
to the case of the Helmholtz equation in that the use of a suitable system of functions in the transverse
direction reduces the stability analysis to one for a decoupled ordinary differential equation (ODE) system.

%------------------------------------------------
\subsection{Motivation: the waveguide eigenproblems}
\label{sec:motivation-maxwell}
%------------------------------------------------
We motivate our choice of transversal expansion system by deriving appropriate transversal eigenvalue problems
for the electric field $\bfE$ and \JMM{the} magnetic field $\bfH$. We write $\bfE = \bfE(x,z)$, $\bfH = \bfH(x,z)$
with $x \in D$, $z \in (0,L)$ as
\begin{align*}
\bfE &= \left(\begin{array}{c} \bfE_t \\ E_3 \end{array}\right), 
& 
\bfH & = \left(\begin{array}{c} \bfH_t \\ H_3 \end{array}\right)
\end{align*}
with the transversal components $\bfE_t$, $\bfH_t:D \times (0,L) \rightarrow \doubleIR^2$ and the longitudinal
components $E_3$, $H_3: D \times (0,L) \rightarrow \doubleIR$. The 2D curl operators $\bfnab \times $
and $\curl$ are defined by
$\bfnab \times \psi = (\partial_2 \psi, -\partial_1 \psi)^\top$ for scalar functions $\psi$
and
$\curl {\boldsymbol \psi} = \partial_1 {\boldsymbol \psi}_2 -\partial_2 {\boldsymbol \psi}_1$ for vector-valued functions ${\boldsymbol \psi}$.
Correspondingly, we define the space $\bfH(\curl,D)  = \{{\boldsymbol \psi} \in (L^2(D))^2\,|\, \curl {\boldsymbol \psi} \in L^2(D)\}$. 
We set $\bfH_0(\curl,D) = \{{\boldsymbol \psi} \in \bfH(\curl,D)\,|\, \gamma^{2D}_t {\boldsymbol \psi} = 0\}$ with 
the tangential trace $\gamma^{2D}_t {\boldsymbol \psi} = {\boldsymbol \psi} - ({\boldsymbol \psi} \cdot \bfn) \bfn$ and the outer normal vector $\bfn$ 
of $D$. We use the notation 
${\bfH}_0(\div,D) = \{ {\boldsymbol \psi} \in {\bfH}(\div,D)\,|\, {\boldsymbol \psi} \cdot \bfn = 0 \ \mbox{ on $\partial D$}\}$.
The vector $\ez$ is the unit vector in $z$-direction, and the notation $\ez \times {\boldsymbol \psi}$ for a 2D-vector field is
understood by viewing ${\boldsymbol \psi}$ as a 2D-vector field with vanishing third component.
We will use the following 2D identities (where the the differential operators $\bfnab$, $\bfnab \times$, $\div$, and $\curl$ act on the 2D variable $x \in D$ only):
\begin{subequations}
\label{eq:identities}
\begin{align}
\ez \times (\ez \times \bfE_t) & = - \bfE_t, && \\
\ez \times (\bfnab \times E_3) & = \bfnab E_3, 
&
\ez \times \bfnab E_3& = - \bfnab \times E_3, \\
\curl (\ez \times \bfE_t) & = \div \bfE_t, 
&
\div (\ez \times \bfE_t) & = - \curl  \bfE_t\, .
\end{align}
\end{subequations}
The original system of equations (\ref{eq:primal-maxwell}) translates into:
\be
\label{eq:1st_order_transient_system}
\left\{
\begin{array}{rlll}
\bfnab \times E_3 + \ez \times \frac{\ptl}{\ptl z} \bfE_t - \bi \omega \bfH_t & = \bff_t, \\[5pt]
\curl \bfE_t - \bi \omega H_3 & = f_ 3,\\[5pt]
\bfnab \times H_3 + \ez \times \frac{\ptl}{\ptl z} \bfH_t + \bi \omega  \bfE_t & = \bfg_t, \\[5pt]
\curl \bfH_t + \bi \omega  E_3 & = g_3 \, .
\end{array}
\right.
\ee
Multiplying (\ref{eq:1st_order_transient_system})$_1$ and (\ref{eq:1st_order_transient_system})$_3$ by $\bi \omega \, \ez \times$, we obtain:
\be
\left\{
\begin{array}{rlll}
\bfnab \bi \omega E_3 - \frac{\ptl}{\ptl z} \bi \omega \bfE_t + \omega^2 \, \ez \times \bfH_t & = \bi \omega \, \ez \times \bff_t, \\[5pt]
\curl \bfE_t - \bi \omega H_3 & = f_ 3,\\[5pt]
\bfnab \bi \omega H_3 - \frac{\ptl}{\ptl z} \bi \omega \bfH_t - \omega^2 \, \ez \times  \bfE_t & = \bi \omega \, \ez \times \bfg_t, \\[5pt]
\curl \bfH_t + \bi \omega  E_3 & = g_3 \, .
\end{array}
\right.
\label{eq:1st_order_transient_system_mod}
\ee
The eigensystem corresponding to the first order system operator and $e^{\bi \beta z}$ ansatz in the variable $z$ is 
to seek solutions $(\bfE,\bfH)$ of the homogeneous equation in the form 
$\bfE(x,z) = (\bfEt_t(x),\Edreit(x)) e^{\bi \beta z}$ and  
$\bfH(x,z) = (\bfHt_t(x),\Hdreit(x)) e^{\bi \beta z}$, which leads to 
\be
\left\{
\begin{array}{rllll}
\bfEt_t \in \bfH_0(\curl, D),\, \Edreit \in H^1_0(D),\\[8pt]
\bfHt_t \in \bfH(\curl,D),\, \Hdreit \in H^1(D), \\[8pt]
\bi \omega \bfnab \Edreit + \omega^2 \ez \times \bfHt_t & = - \omega \beta \bfEt_t, \\[8pt]
\curl \bfEt_t - \bi \omega \Hdreit & = 0, \\[8pt]
\bi \omega \bfnab \Hdreit - \omega^2 \ez \times  \bfEt_t & = - \omega \beta \bfHt_t, \\[8pt]
\curl \bfHt_t + \bi \omega  \Edreit & = 0 \, .
\end{array}
\right.
\label{eq:1st_order_eigensystem}
\ee
Eliminating $\Edreit$ and $\Hdreit$ from the system~(\ref{eq:1st_order_eigensystem}), we obtain a simplified but second order system for $\bfEt_t$, $\bfHt_t$ only:
\be
\left\{
\begin{array}{rlll}
\bfEt_t \in \bfH_0(\curl, D),\, \curl \bfEt_t \in H^1(D), \\[8pt]
\bfHt_t \in \bfH(\curl,D),\, \curl \bfHt_t \in H^1_0(D),\\[8pt]
- \bfnab (\curl \bfHt_t) + \omega^2 \ez \times \bfHt_t & = - \omega \beta \bfEt_t, \\[8pt]
\bfnab (\curl \bfEt_t) - \omega^2 \ez \times  \bfEt_t & = - \omega \beta \bfHt_t \, .
\end{array}
\right.
\label{eq:2nd_order_eigensystem}
\ee
%\todo{\bf wo gehen die BC hin?}
%-----------------
\subsection{Reduction to single variable eigensystems}
%-----------------
Assume $\beta \ne 0$.
Solving~(\ref{eq:2nd_order_eigensystem})$_2$ for $\bfHt_t$,
\be
\begin{array}{rll}
\bfHt_t & \ds = - \frac{1}{\omega \beta} [ \bfnab \curl \bfEt_t - \omega^2 \ez \times  \bfEt_t ], \\[8pt]
\curl \bfHt_t & \ds = \frac{\omega}{\beta} \curl (\ez \times  \bfEt_t) = \frac{\omega}{\beta} \div \bfEt_t, 
\end{array} 
\label{eq:H_t}
\ee
and substituting it into~(\ref{eq:2nd_order_eigensystem})$_1$, we obtain an eigenvalue problem for $\bfEt_t$ alone:
\be
\left\{
\begin{array}{ll}
\bfEt_t \in \bfH_0(\curl, D),\, \curl \bfEt_t \in H^1(D),\,  \div \bfEt_t \in H^1_0(D),  \\[8pt]
\bfnab \times \curl \bfEt_t - \omega^2  \bfEt_t - \bfnab \div \bfEt_t  = - \beta^2 \bfEt_t \, .
\end{array}
\right.
\label{eq:Eeigenproblem}
\ee
Similarly, solving~(\ref{eq:2nd_order_eigensystem})$_1$ for $\bfEt_t$,
\be
\begin{array}{rlll}
\bfEt_t & \ds = - \frac{1}{\omega \beta} [ - \bfnab \curl \bfHt_t  + \omega^2 \ez \times \bfHt_t ], \\[8pt]
\curl \bfEt_t & \ds = - \frac{\omega}{\beta} \curl (\ez \times \bfHt_t) = - \frac{\omega}{\beta} \div \bfHt_t, 
\end{array}
\label{eq:E_t}
\ee
and substituting it into~(\ref{eq:2nd_order_eigensystem})$_2$, we obtain an eigenvalue problem for $\bfHt_t$ alone:
\be
\left\{
\begin{array}{lll}
\bfHt_t \in \bfH(\curl, D) \cap \bfH_0(\div,D),\,  \curl \bfHt_t \in H^1_0(D),\, \div \bfHt_t \in H^1(D),  \\[8pt]
 \bfnab \times \curl \bfHt_t - \omega^2  \bfHt_t - \bfnab (\div \bfHt_t)  = - \beta^2 \bfHt_t \, .
\end{array}
\right.
\label{eq:Heigenproblem}
\ee
Note that the boundary condition $\gamma^{2D}_t \bfEt_t = 0$  on $\partial D$ implies by (\ref{eq:1st_order_eigensystem})$_1$
the boundary condition $\bfn \cdot \bfHt_t = 0$  on $\partial D$.

%-------------------------------------------------------
\subsection{Structure of the single variable eigenproblems}
%-------------------------------------------------------

\begin{lemma}[Helmholtz decompositions]

Let $D \subset \doubleIR^2$ be a simply connected domain.
For every $\bfE \in L^2(D)^2$ there exist a unique $\phi \in H^1_0(D)$ and a unique $\psi \in H^1(D)$ with $\int_D \psi = 0$
such that
\be
\bfE = \bfnab \phi + \bfnab \times \psi \, .
\label{eq:Helmholtz_decomposition}
\ee
Similarly, for every $\bfH \in L^2(D)^2$ there exist a unique $\phi \in H^1_0(D)$ and a unique $\psi \in H^1(D)$ with $\int_D \psi = 0$
such that
\be
\bfH = \bfnab \times \phi + \bfnab \psi \, .
\label{eq:Helmholtz_decomposition2}
\ee

\end{lemma}
\begin{proof}
The decomposition of $\bfH$ can be found in \cite[Thm.~{3.3}/Rem.~{3.3}]{girault-raviart86}. The assertion about $\bfE$ follows by similar arguments. 
\end{proof}

Consider now the eigenvalue problem~(\ref{eq:Eeigenproblem}) and the Helmholtz decomposition of $\bfE$. The boundary condition $\gamma^{2D}_t \bfE_t = 0$
implies that $\frac{\ptl \psi}{\ptl n} = 0$ on $\ptl D$.
Substituting~(\ref{eq:Helmholtz_decomposition}) into (\ref{eq:Eeigenproblem}), we obtain
\be
\bfnab \times (\underbrace{- \Delta \psi + (\beta^2 - \omega^2) \psi}_{=: \Psi}) + \bfnab (\underbrace{- \Delta \phi + (\beta^2 - \omega^2) \phi}_{=:\Phi}) = 0 \, .
\label{eq:derived_Helmholtz}
\ee
The equation above represents the Helmholtz decomposition of the zero function. Indeed, $\Delta \phi =  \div \bfE $ is zero on $\partial D$ and, therefore, $\Phi$ is zero on $\partial D$ as well. Multiplying~(\ref{eq:derived_Helmholtz}) by $\bfnab \times \Psi$ and
integrating the second term by parts, we learn that
$$
\Psi = - \Delta \psi + (\beta^2 - \omega^2) \psi = \mathrm{const} \, .
$$
Consequently,
$$
(\Psi, 1) = (\bfnab \psi , \underbrace{\bfnab 1}_{=0}) - \langle \underbrace{\frac{\ptl \psi}{\ptl n}}_{=0}, 1 \rangle + (\beta^2 - \omega^2) \underbrace{( \psi , 1)}_{=0}  = 0 \, .
$$
Uniqueness of $\phi$ and $\psi$ in the Helmholtz decomposition implies that $\Phi = \Psi = 0$.
Let $(\lambda_j,\phi_j)$, $j \in \doubleIN$, and $(\mu_i,\psi_i)$, $i \in \doubleIN_0$, be the Dirichlet and Neumann eigenpairs of the Laplacian in 
the domain $D$ with $\mu_0 = 0$. Vanishing of $\Phi$ and $\Psi$ implies that 

\begin{align*}
(\phi,\omega^2 - \beta^2) 
\mbox{ is a Dirichlet eigenpair} 
\quad 
\mbox{ and } 
\quad 
(\psi,\omega^2 - \beta^2) 
\mbox{ is a Neumann eigenpair.} 
\end{align*}

%there exist $i,j$ such that
%\begin{align*}
%\phi & = \phi_j,&  \omega^2 - \beta^2 & = \lambda_j && \quad \text{and} & \quad \psi &= \psi_i, &  \omega^2 - \beta^2 & = \mu_i \, .
%\end{align*}
If the Dirichlet and Neumann eigenvalues are distinct, the eigenvector $\bfE$ must reduce to either a gradient or a curl.
This is the case, e.g., for a circular domain $D$. In the case of a common Dirichlet and Neumann eigenvalue, $\lambda_i = \mu_j$,
we obtain a multiple eigenvalue $\beta^2 = \omega^2 - \lambda_i = \omega^2 - \mu_j$, with the eigenspace consisting of linear combinations of the vectors
$$
\bfE = A \bfnab \times \psi_i + B \bfnab \phi_j \, , \quad A,B \in \doubleIC \, .
$$
\begin{lemma}
\label{lemma:eigenvalues}
Let $(\lambda_j,\phi_j)$, $j \in \doubleIN$, and $(\mu_i,\psi_i)$, $i \in \doubleIN_0$, 
denote the Dirichlet and Neumann eigenpairs of the Laplacian in the domain $D$ with $\mu_0 = 0$.
The eigenvalues $\beta^2_n$ for~(\ref{eq:Eeigenproblem}) are classified into the following three families:
\begin{enumerate}[(a)]
\item $\beta^2_{n} = \omega^2 - \mu_i$ with $\mu_i$ distinct from all $\lambda_j$. The corresponding eigenvectors are curls:
$$
\bfE = \bfnab \times \psi_i \, ,
$$
with (geometric) multiplicity of $\beta^2_{n}$ equal to the (geometric) multiplicity of $\mu_i$.
\item $\beta^2_{n} = \omega^2 - \lambda_j$ with $\lambda_j$ distinct from all $\mu_i$. The corresponding eigenvectors are gradients:
$$
\bfE = \bfnab \phi_j\, ,
$$
with (geometric) multiplicity of $\beta^2_{n}$ equal to the (geometric) multiplicity of $\lambda_j$.
\item $\beta^2_{n} =  \omega^2 - \mu_i = \omega^2 - \lambda_j$ for $\mu_i = \lambda_j$. 
The corresponding eigenvectors are linear combinations of curls and gradients:
$$
\bfE = A \bfnab \times \psi_i + B \bfnab \phi_j\, , \quad A,B \in \doubleIC\, ,
$$
with (geometric) multiplicity of $\beta^2_{n}$ equal to the sum of (geometric) multiplicities of $\mu_i$ and $\lambda_j$.
\end{enumerate}
\end{lemma}
In the same way, we prove the analogous result for eigenproblem~(\ref{eq:Heigenproblem}).
\begin{lemma}
\label{lemma:eigenvalues2}
Let $(\lambda_j,\phi_j)$, $j \in \doubleIN$, and $(\mu_i,\psi_i)$, $i \in \doubleIN_0$, 
denote the Dirichlet and Neumann eigenpairs of the Laplacian in the domain $D$  with $\mu_0 = 0$.
The eigenvalues $\beta_{n}^2$ for~(\ref{eq:Heigenproblem})  are classified into the following three families.
\begin{enumerate}[(a)]
\item $\beta^2_{n} = \omega^2 - \mu_i$ with $\mu_i$ distinct from all $\lambda_j$. The corresponding eigenvectors are gradients:
$$
\bfH = \bfnab \psi_i \, ,
$$
with (geometric) multiplicity of $\beta^2_{n}$ equal to the (geometric) multiplicity of $\mu_i$.
\item $\beta^2_{n} = \omega^2 - \lambda_j$ with $\lambda_j$ distinct from all $\mu_i$. The corresponding eigenvectors are curls:
$$
\bfH = \bfnab \times \phi_j\, ,
$$
with (geometric) multiplicity of $\beta^2_{n}$ equal to the (geometric) multiplicity of $\lambda_j$.
\item $\beta^2_{n} =  \omega^2 - \mu_i = \omega^2 - \lambda_j$ for $\mu_i = \lambda_j$. The corresponding eigenvectors are linear combinations of gradients and curls:
$$
\bfH = A \bfnab \psi_i + B \bfnab \times \phi_j\, , \quad A,B \in \doubleIC\, ,
$$
with (geometric) multiplicity of $\beta^2_{n}$ equal to the sum of (geometric) multiplicities of $\mu_i$ and $\lambda_j$.
\end{enumerate}
\end{lemma}

The following example shows that multiple eigenvalues can arise. 
\begin{example}[Cylindrical waveguide] 
\label{ex:cylinder}
Consider the Dirichlet or Neumann Laplace eigenvalue problem  in a unit disk $B_1(0)$ 
$$
- \Delta u = \lambda u, \quad \lambda = \nu^2\, .
$$
For the Dirichlet problem the operator is positive definite, so $\nu > 0 $, for the Neumann problem, $u = \mathrm{const}$ corresponds to the zero eigenvalue,
all other eigenvalues are positive as well. By introducing polar coordinates $(r,\theta)$ and separating variables, one arrives at 
eigenpairs $(u,\JMM{\nu^2})$ of the form 
\begin{align*}
u_{0,m}(r,\theta) &= J_0(\nu_{0,m} r), & m&=1,2,\ldots, && \\
u_{k,m} (r,\theta) & = J_k(\nu_{k,m} r) \cos (k\theta),  & u_{k,m}(r,\theta) &= J_k(\nu_{k,m} r) \sin (k\theta),  
&
k,m=1,2,\ldots, 
\end{align*}
where the values $\nu_{k,m}$ are the zeros of the Bessel functions or its derivative: 
\begin{align*}
\mbox{ Dirichlet b.c.:} && J_k(\nu_{k,m})=  0 \qquad k=0,1,\ldots, \quad m=1,2,\ldots, \\ 
\mbox{ Neumann b.c.:} && J^\prime_k(\nu_{k,m})=  0 \qquad k=0,1,\ldots, \quad m=1,2,\ldots. 
\end{align*}
We notice that for $k > 0$, the eigenvalue $\JMM{\nu_{k,m}^2}$ is a double eigenvalue. 
\eremk
\end{example}
\subsection{Stability analysis}
%-----------------------------------------
As in the case of the Helmholtz equation in (\ref{eq:non-degeneracy}), we require the frequency $\omega$ to 
be such that the eigenvalues $\beta_n \ne 0$, i.e., our analysis will be performed under the assumption 
\begin{align}
\label{eq:non-degeneracy-maxwell} 
\lambda^2_j - \omega &\ne 0 \quad \forall j \in \doubleIN 
&\mbox{ and } \qquad 
\mu^2_i - \omega &\ne 0 \quad \forall i \in \doubleIN_0, 
\end{align}
where $\lambda_j$, $ j\in \doubleIN$, and $\mu_i$, $i \in \doubleIN_0$, are the Dirichlet and Neumann eigenvalues of the 
Laplacian on $D$. 
In this section, we show
\begin{theorem}
\label{thm:stability-maxwell} 
Let $c_0 > 0$ and assume (\ref{eq:non-degeneracy-maxwell}).
There is $C > 0$ independent of $L \ge c_0$ (but possibly dependent on $\omega$) such that
the solution $(\bfE,\bfH)$ of (\ref{eq:primal-maxwell}) satisfies
\begin{equation}
\label{eq:thm:stability-maxwell-1} 
\|(\bfE,\bfH)\|_{L^2(\Omega)} \leq C L \|(\bff,\bfg)\|_{L^2(\Omega)}.
\end{equation}
Additionally, for every $(\bfE,\bfH) \in D(A^{\tmaxwell})$ we have 
\begin{equation}
\label{eq:thm:stability-maxwell-2} 
\|(\bfE,\bfH)\|_{L^2(\Omega)} \leq C L \|A^{\tmaxwell}(\bfE,\bfH)\|_{L^2(\Omega)}.
\end{equation}
%That is, the associated operator satisfies $\|A(\bfE,\bfH)\|_{L^2(\Omega)} \ge C L^{-1} \|(\bfE,\bfH)\|_{L^2(\Omega)}$ for all
%$(\bfE,\bfH) \in D(A)$. 
\end{theorem}

%-----------------------------------------
\subsubsection{Reduction of the waveguide problem to an ODE system}
%-----------------------------------------
Let $\{(\mu_i,\psi_i)\}_{i \ge 0}$ and $\{(\lambda_j, \phi_j)\}_{j \ge 1}$ denote the Neumann and Dirichlet eigenpairs for the Laplace operator in $D$.
We normalize the eigenvectors for $i$, $j \ge 1$ (i.e., excluding the case $\mu_0 = 0$) by 
$$
\Vert \bfnab \psi_i \Vert^2 = \Vert \bfnab \times \psi_i \Vert^2 =1,\quad \Vert \bfnab \phi_j  \Vert^2 = \Vert \bfnab \times \phi_j  \Vert^2= 1
$$
so that 
$$
\Vert \psi_i \Vert^2 = \mu_i^{-1},\quad \Vert  \phi_j \Vert^2 =\lambda_j^{-1}\, .
$$
Consistently with Lemmas~\ref{lemma:eigenvalues} and \ref{lemma:eigenvalues2},
we make the following ansatz for $\bfE$ and $\bfH$:
\be
\begin{array}{rl}
\bfE & = \sum_{i\ge 1} \left(\begin{array}{c} \bfnab \times \psi_i \\ 0 \end{array}\right) \alpha_i(z) + \sum_{j\ge 1} \left(\begin{array}{c} \bfnab \phi_j \\ 0 \end{array}\right)\beta_j(z) + \sum_{j \ge 1} \ez \phi_j \gamma_j(z), \\[5pt]
\bfH & = \sum_{i\ge 1} \left(\begin{array}{c} \bfnab  \psi_i \\ 0 \end{array}\right) \delta_i(z) + \sum_{ j  \ge 1}\left(\begin{array}{c} \bfnab \times \phi_j \\ 0 \end{array}\right) \eta_j(z) + \sum_{i \ge 1} \ez \psi_i \zeta_i(z) \, .
\end{array}
\label{eq:ansatz}
\ee
Substituting into the system~(\ref{eq:1st_order_transient_system})
and testing (in the sense of $L^2$-product) the first equation with $(\bfnab \psi_i, 0)^\top, (\bfnab \times \phi_j, 0)^\top, \ez \psi_i$, and
the second equation with $(\bfnab \times \psi_i,0)^\top, (\bfnab \phi_j, 0)^\top, \ez \phi_j$, we obtain a system of six ODEs
(see Lemma~\ref{lemma:maxwell-ODE-details} for a more detailed derivation): 
\begin{subequations}
\label{eq:maxwell-ODE}
\begin{alignat}{2}
\label{eq:maxwell-1}
\alpha_i' - \bi \omega \delta_i & = (\bff,(\bfnab \psi_i,0)^\top)_{L^2(D)}\ && =: f_{i,1},  \\
\label{eq:maxwell-2}
- \beta_j' + \gamma_j - \bi \omega \eta_j & = (\bff, (\bfnab \times \phi_j,0)^\top)_{L^2(D)}\ && =: f_{j,2}, \\
\label{eq:maxwell-3}
\alpha_i - \bi \omega \mu_i^{-1} \zeta_i & = (\bff, \ez \psi_i)_{L^2(D)}\ && =: f_{i,3}, \\
\label{eq:maxwell-4}
- \delta_i' + \zeta_i + \bi \omega \alpha_i & = (\bfg,(\bfnab \times \psi_i,0)^\top)_{L^2(D)}\ && =: g_{i,1},  \\
\label{eq:maxwell-5}
\eta_j' + \bi \omega \beta_j & = (\bfg, (\bfnab \phi_j,0)^\top)_{L^2(D)}\ && =: g_{j,2}, \\
\label{eq:maxwell-6}
\eta_j + \bi \omega \lambda_j^{-1}\gamma_j & = (\bfg, \ez \phi_j)_{L^2(D)}\ && =: g_{j,3}.
\end{alignat}
\end{subequations}
Note that the system decouples into  two subsystems consisting of two first-order ODEs and one algebraic equation. 
The ODE system is complemented with boundary conditions. (By Lemma~\ref{lemma:maxwell-ODE-details}, the functions 
$\alpha_i$, $\beta_i$, $\delta_i$, $\eta_i \in H^1(0,L)$ so that by Sobolev's embedding theorem one may impose boundary
conditions at $0$ and $L$.) The condition (\ref{eq:primal-maxwell-b}) requires the conditions
\begin{subequations}
\label{eq:initial-conditions}
\begin{align}
\label{eq:initial-condition-alpha}
\alpha_i(0) & = 0 \quad \forall i \in \doubleIN , \\
\label{eq:initial-condition-beta}
\beta_j(0) & = 0 \quad \forall j \in \doubleIN .
\end{align}
\end{subequations}
The condition (\ref{eq:primal-maxwell-d}) that waves be outgoing takes the following form.
Assuming $\bff = 0$ and $\bfg = 0$ on $(L,\infty)$, one obtains for $\alpha_i$ and $\beta_j$
the second-order ODEs (see the proofs of Lemmas~\ref{lemma:alpha-delta-zeta}, \ref{lemma:beta-eta-gamma} for some more details)
\begin{alignat}{2}
\label{eq:widetilde-mu_i}
- \alpha_i^{\prime\prime} - \widetilde{\mu}_i^2 \alpha_i &= 0,
\qquad
\widetilde{\mu}_i^2 &&:= \mu_i - \omega^2, \\
\label{eq:widetilde-lambda_j}
- \beta_j^{\prime\prime} - \widetilde{\lambda}_j^2 \beta_j &= 0,
\qquad
\widetilde{\lambda}_j^2 &&:= \lambda_j - \omega^2, 
\end{alignat}
with fundamental solutions $e^{\pm \widetilde{\mu}_i z}$ and $e^{\pm \widetilde{\lambda}_j z}$.
For solutions to be outgoing, we select the minus sign and realize this with the boundary
conditions at $z = L$ given by
\begin{subequations}
\label{eq:impedance-bc}
\begin{align}
\label{eq:imdepance-bc-alpha}
\alpha_i^\prime(L) + \widetilde{\mu}_i \alpha_i(L) = 0, \\
\label{eq:imdepance-bc-beta}
\beta_j^\prime(L) + \widetilde{\lambda}_j \beta_j(L) = 0.
\end{align}
\end{subequations}
%\begin{remark}
%\label{rem:alternative-to-impedance-bc}
In the following, however, we will not use conditions (\ref{eq:impedance-bc}) directly. Instead, we obtain 
from (\ref{eq:maxwell-ODE}) with $\bff = 0 = \bfg$ and (\ref{eq:impedance-bc}) the alternative condition
\begin{subequations}
\label{eq:dtn-condition-maxwell}
\begin{align} 
\label{eq:dtn-condition-maxwell-a}
\bi \ome \delta_i(L)  &= -\widetilde{\mu}_i \alpha_i(L), \\
\label{eq:dtn-condition-maxwell-b}
\widetilde{\lambda}_j \eta_j(L) & = \bi \ome \beta_j(L), 
\end{align}
\end{subequations}
that is, a relation between the electric field $\bfE$ and the magnetic 
field $\bfH$ on $\Gammao$. 
%\end{remark}
\begin{remark}
In the recent, independent work \cite{kim17}, the expansions (\ref{eq:ansatz}) based on the Laplace Dirichlet and Neumann eigenvalue problems have 
also been used to formulate the DtN-operator for Maxwell's equations for waveguides and show well-posedness of the truncated waveguide problem. 
The present analysis goes beyond that by revealing the dependence on the length of the waveguide. 
\eremk
\end{remark}
%-----------------------------------------
\subsubsection{The operator \texorpdfstring{$\dtnm$}{DtNmw}}
%-----------------------------------------
We observe that for sufficiently smooth $\bfE$, $\bfH$, the values $\alpha_i(L)$, $\beta_j(L)$, $\delta_i(L)$, $\eta_j(L)$ are 
the coefficients when expanding the tangential components $(\pi_t \bfE)|_{\Gammao}$, $(\pi_t \bfH)|_{\Gammao}$ in terms of the 
orthogonal systems $(\bfnab \times \psi_i)_i$, $(\bfnab \phi_j)_j$ on the one hand and $(\bfnab \psi_i)_i$, $(\bfnab \times \phi_j)_j$ on the other hand. 
That is, $(\pi_t \bfE)|_{\Gammao}$ can be expanded as 
\begin{align*}
(\pi_t \bfE)|_{\Gammao} & = \sum_{i\ge 1} \alphahat(\bfE) \bfnab \times \psi_i + \sum_{j\ge 1} \betahat(\bfE) \bfnab \phi_j
\end{align*}
where $\alphahat$, $\betahat$ are linear functionals on $\bfH_{0,\Gammam}(\curl,\Ome)$ and satisfy, for $\bfE$ sufficiently smooth, 
$\alphahat(\bfE) = (\pi_t \bfE, \bfnab \times \psi_i)_{L^2(\Gammao)} = \alpha_i(L)$ and   
$\betahat(\bfE) = (\pi_t \bfE, \bfnab \phi_j)_{L^2(\Gammao)} = \beta_j(L)$. (See Theorem~\ref{thm:alphai-betaj} for the precise statement.) 
%For sufficiently smooth functions $\alpha_i$, $\beta_j$ one then has $\alpha_i(L) = \alphahat(\bfE)$, $\beta_j(L) = \betahat(\bfE)$. 
An analogous expansion holds for $\pi_t \bfH$. 

Motivated by the condition (\ref{eq:dtn-condition-maxwell}) on $\Gammao$, we define the operator 
$\dtnm:\bfH_{0,\Gammam}(\curl,\Ome) \rightarrow \bfH^{-1/2}(\Gammao)$ by 
\begin{align}
\label{eq:dtnm}
\dtnm \bfE &:= \sum_{ i \ge 1} -\frac{\widetilde{\mu}_i}{\bi \ome} \alphahat(\bfE) \bfnab \psi_i 
+ \sum_{ j\ge 1} \frac{\bi \ome}{\widetilde{\lambda}_j} \betahat(\bfE) \bfnab \times \phi_j. 
\end{align}
We refer to Section~\ref{sec:mapping-properties-dtnm} for the mapping properties of this 
thus formally defined operator.

%-----------------------------------------
\subsubsection{Analysis of the system (\ref{eq:maxwell-ODE}) and proof of Theorem~\ref{thm:stability-maxwell}}
%-----------------------------------------

%
%
Representing $\bfE$ in the form (\ref{eq:ansatz})$_1$, we obtain:
$$
\int_D \vert \bfE(x,y,z) \vert^2 \, \diff x \diff y  = \sum_{i\ge 1} \left[ \vert \alpha_i(z) \vert^2 + \vert \beta_i(z) \vert^2 + \lambda_i^{-1} \vert \gamma_i(z) \vert^2 \right] \, 
$$
with an analogous formula for $\|\bfH\|^2_{L^2(D)}$. Hence, 
\begin{align*}
\Vert \bfE \Vert^2 & = \int_0^L \int_D \vert E(x,y,z) \vert^2 \, \diff x \diff y \diff z = \sum_{i\ge 1} \Vert (\alpha_i,\beta_i,\lambda_i^{-\half} \gamma_i )\Vert_{L^2(0,L)}^2 \, , \\
\Vert \bfH \Vert^2 &= \sum_{i \ge 1} \Vert (\delta_i,\eta_i,\mu_i^{-\half} \zeta_i ) \Vert_{L^2(0,L)}^2 \, .
\end{align*}
Using the mutual orthogonality of $(\bfnab \psi_i,0)^\top$, $(\bfnab \times \phi_i,0)^\top$, $\ez \psi_i$ in $L^2(D)$, the function $\bff$ in (\ref{eq:primal-maxwell})
can be represented in the form
\begin{align*}
\bff &= \sum_{i \ge 1} \Bigl[
(\bff,(\bfnab \psi_i,0)^\top)_{L^2(D)} (\bfnab \psi_i,0)^\top + (\bff,(\bfnab \times \phi_i,0)^\top)_{L^2(D)} (\bfnab \times \phi_i,0)^\top 
\\ 
& \qquad \qquad \mbox{} + (\bff,\ez \mu_i^{\half} \psi_i)_{L^2(D)} \ez \mu_i^{\half} \psi_i \Bigr]. 
\end{align*}
Consequently,
\begin{align*}
\int_D \vert \bff(x,y,z) \vert^2 \, \diff x \diff y & = \sum_{i \ge 1} \Bigl[ \vert (\bff,(\bfnab \psi_i,0)^\top)_{L^2(D)} (z) \vert^2 + \vert (\bff, (\bfnab \times \phi_i,0)^\top)_{L^2(D)} (z) \vert^2 \\
& \qquad \mbox{} +
\vert (\bff,\ez \psi_i)_{L^2(D)} (z) \vert^2 \mu_i \Bigr]
\end{align*}
and thus
\begin{align*}
\Vert \bff \Vert^2 & = \int_0^L \int_D \vert \bff(x,y,z) \vert^2 \, \diff x \diff y \diff z \\
& = \sum_{i \ge 1} \Vert ( (\bff,(\bfnab \psi_i,0)^\top), (\bff,(\bfnab \times \phi_i,0)^\top), \mu_i^{\half} (\bff, \ez \psi_i)) \Vert^2_{L^2(0,L)} \, .
\end{align*}
Similarly, for the right-hand side $\bfg$ in the second equation in (\ref{eq:primal-maxwell}) we have
$$
\Vert \bfg \Vert^2 = \sum_{i \ge 1} \Vert ((\bfg,(\bfnab \times \psi_i,0)^\top), (\bfg,(\bfnab \phi_i,0)^\top), \lambda_i^{\half} (\bfg, \ez \phi_i)) \Vert^2_{L^2(0,L)} \, .
$$
The formulas for $\Vert \bfE \Vert^2, \Vert \bfH \Vert^2$ and $\Vert \bff \Vert^2, \Vert \bfg \Vert^2$, and the system of ODEs~(\ref{eq:maxwell-ODE}) imply
that sufficient (and necessary) for the boundedness below in Theorem~\ref{thm:stability-maxwell} are the $L^2$-estimates for the subsystems of ODEs:
\begin{subequations}
\label{eq:maxwell-requirement}
\begin{align}
\Vert (\alpha_i,\delta_i,\mu_i^{-\half} \zeta_i ) \Vert_{L^2(0,L)}^2
& \leq C \Vert ((\bff,(\bfnab \psi_i,0)^\top),(\bfg,(\bfnab \times \psi_i,0)^\top), (\bff, \ez \mu_i^{\half} \psi_i) ) \Vert^2_{L^2(0,L)},  \\
\Vert (\beta_i,\eta_i,\lambda_i^{-\half} \gamma_i ) \Vert_{L^2(0,L)}^2
& \leq C \Vert ((\bff,(\bfnab \times \psi_i,0)^\top),(\bfg,(\bfnab  \psi_i,0)^\top), (\bfg, \ez \lambda_i^{\half} \phi_i) ) \Vert^2_{L^2(0,L)}
\end{align}
\end{subequations}
with some constant $C$ independent of $i$.
We have from Lemma~\ref{lemma:ihl}:
\begin{lemma}
\label{lemma:alpha-delta-zeta}
Let $(\alpha_i,\delta_i,\zeta_i)$ solve (\ref{eq:maxwell-1}), (\ref{eq:maxwell-3}),
(\ref{eq:maxwell-4}) together with the boundary conditions
(\ref{eq:initial-condition-alpha}), 
%(\ref{eq:imdepance-bc-alpha}). 
(\ref{eq:dtn-condition-maxwell-a}). 
Assume that
$\inf_{i \ge 1} |\widetilde{\mu}_i| > 0$. Then the following holds:
\begin{enumerate}[(i)]
\item (evanescent modes)
\label{item:lemma:alpha-delta-zeta-i}
There is $C > 0$ independent of $i$ such that for all $i \in \doubleIN$ with $\widetilde{\mu}_i >0$
\begin{align}
\label{eq:lemma:alpha-delta-zeta-1}
\|\alpha^\prime_i\| + \sqrt{\mu_i}\|\alpha_i\| &\leq C \left[ \|f_{i,1}\| + \sqrt{\mu_i} \|f_{i,3}\| + \JMM{\mu_i^{-1/2}}\|g_{i,1}\|\right], \\
\label{eq:lemma:alpha-delta-zeta-2}
\|\delta_i\| & \leq C \left[  \|f_{i,1}\| + \sqrt{\mu_i} \|f_{i,3}\| + \JMM{\mu_i^{-1/2}}\|g_{i,1}\|\right], \\
\label{eq:lemma:alpha-delta-zeta-3}
\mu^{-1/2}_i \|\zeta_i\| &\leq C \left[ \|f_{i,1}\| + \sqrt{\mu_i} \|f_{i,3}\| + \JMM{\mu_i^{-1/2}} \|g_{i,1}\|\right].
\end{align}
\item
(propagating modes)
\label{item:lemma:alpha-delta-zeta-ii}
For all $i \in \doubleIN$ with $\widetilde{\mu}_i \in \bi \doubleIR$ the estimates
(\ref{eq:lemma:alpha-delta-zeta-1})--(\ref{eq:lemma:alpha-delta-zeta-3}) hold
with $\mu_i$ replaced by $1$ and an additional factor $L$
on the right-hand side.
\end{enumerate}
\end{lemma}

\begin{proof} The component $\alpha_i$ satisfies the following weak form:
For all $v \in H^1_{(0}(0,L)$ 
\begin{align}
\begin{aligned}
(\alpha^\prime_i, v^\prime)_{L^2(0,L)} + & \widetilde{\mu}^2_i (\alpha_i,v)_{L^2(0,L)}
+\widetilde{\mu}_i \alpha_i(L) \overline{v}(L) \\
& = (f_{i,1},v^\prime)_{L^2(0,L)} +
  i\ome (g_{i,1},v)_{L^2(0,L)} + \mu_i (f_{i,3},v)_{L^2(0,L)}.
\end{aligned}
\end{align}
This is obtained by the following steps
(see Appendix~\ref{appendix:alpha-delta-zeta}) for details):
first one eliminates the variable $\zeta_i$; second, one multiplies
(\ref{eq:maxwell-1}) (in the form obtained after removing $\zeta_i$) by
$v^\prime$ and integrates over $(0,L)$; third, one multiplies
(\ref{eq:maxwell-4}) by $v$ and integrates over $(0,L)$; fourth, in the thus obtained
equation the term $(\delta^\prime,v)_{L^2(0,L)}$ is integrated by parts and
the condition (\ref{eq:dtn-condition-maxwell-a}) is employed. 
%$i \ome \delta_i(L) = -\widetilde{\mu}_i \alpha_i (L)$ from

%Remark~\ref{rem:alternative-to-impedance-bc} is used.

\emph{Proof of (\ref{item:lemma:alpha-delta-zeta-i}):} We note
$\widetilde{\mu}_i \sim \sqrt{\mu_i}$ with implied constant independent of $i$.
Lemma~\ref{lemma:ihl} together with the condition $\alpha_i(0) = 0$ from
(\ref{eq:initial-condition-alpha}) then readily implies (\ref{eq:lemma:alpha-delta-zeta-1}).
Combining (\ref{eq:lemma:alpha-delta-zeta-1}) with (\ref{eq:maxwell-1}) provides
(\ref{eq:lemma:alpha-delta-zeta-2}). Finally, (\ref{eq:lemma:alpha-delta-zeta-1}) and
(\ref{eq:maxwell-3}) yield 
(\ref{eq:lemma:alpha-delta-zeta-3}).

We remark that the estimate for $\alpha_i$ is a better by a factor
$\sqrt{\mu_i}$ than required.

\emph{Proof of (\ref{item:lemma:alpha-delta-zeta-ii}):} This case is shown in the
same way as (\ref{item:lemma:alpha-delta-zeta-i}) noting that for the finitely many
propagating modes one has $\mu_i \sim 1$ so that $|\widetilde{\mu}_i | \sim 1$.
\end{proof}
Analogously, we have for the subsystem involving $(\beta_j,\eta_j,\gamma_j)$:

\begin{lemma}
\label{lemma:beta-eta-gamma}
Let $(\beta_j,\eta_j,\gamma_j)$ solve (\ref{eq:maxwell-2}), (\ref{eq:maxwell-5}),
(\ref{eq:maxwell-6}) together with the boundary conditions
(\ref{eq:initial-condition-beta}), 
%(\ref{eq:imdepance-bc-beta}). 
(\ref{eq:dtn-condition-maxwell-b}). 
Assume that
$\inf_j |\widetilde{\lambda}_j| > 0$. Then the following holds:
\begin{enumerate}[(i)]
\item (evanescent modes)
\label{item:lemma:beta-eta-gamma-i}
There is $C > 0$ independent of $j$ such that for all $j$ with $\widetilde{\lambda}_j >0$
\begin{align}
\label{eq:lemma:beta-eta-gamma-1}
\|\beta^\prime_j\| + \sqrt{\lambda_i}\|\beta_j\|
&\leq C \left[ \|f_{j,2}\| + \lambda_j  \|g_{j,3}\| + \sqrt{\lambda_j} \|g_{j,2}\|\right], \\
\label{eq:lemma:beta-eta-gamma-2}
\lambda_j \|\eta_j\| & \leq C \left[  \|f_{j,2}\| +  \lambda_j \|g_{j,3}\| + \sqrt{\lambda_j} \|g_{j,2}\|\right], \\
\label{eq:lemma:beta-eta-gamma-3}
\lambda^{-1/2}_j \|\gamma_j\| &\leq C \left[ \lambda_j^{-1/2} \|f_{j,2}\| + \sqrt{\lambda_j} \|g_{j,3}\| + \|g_{j,2}\|\right].
\end{align}
\item
(propagating modes)
\label{item:lemma:beta-eta-gamma-ii}
For all $j$ with $\widetilde{\lambda}_j \in \bi \doubleIR$ the estimates
(\ref{eq:lemma:beta-eta-gamma-1})--(\ref{eq:lemma:beta-eta-gamma-3}) hold
with $\lambda_j$ replaced by $1$ and an additional factor $L$
on the right-hand side.
\end{enumerate}
\end{lemma}

\begin{proof} The component $\beta_j$ satisfies the following weak form:
For all $v \in H^1_{(0}(0,L)$ 
\begin{align}
\begin{aligned}
(\beta^\prime_j, v^\prime)_{L^2(0,L)} + & \widetilde{\lambda}^2_j (\beta_j,v)_{L^2(0,L)}
+\widetilde{\lambda}_j \beta_j(L) \overline{v}(L) \\
& = -(f_{j,2},v^\prime)_{L^2(0,L)} +
  \frac{\lambda_j}{\bi\ome} (g_{j,3},v^\prime)_{L^2(0,L)} + \frac{\widetilde{\lambda}^2_j}{\bi \ome} (g_{j,2},v)_{L^2(0,L)}.
\end{aligned}
\end{align}
This is obtained by the following steps
(see Appendix~\ref{appendix:beta-eta-gamma}) for details):
first one eliminates the variable $\gamma_j$, which yields the additional equation
\begin{equation}
\label{eq:lemma:beta-eta-gamma-10}
-\bi \ome \beta^\prime_j - \widetilde{\lambda}_j^2 \eta_j = \bi \ome f_{j,2} - \lambda_j g_{j,3}.
\end{equation}
Second, one multiplies
(\ref{eq:lemma:beta-eta-gamma-10}) by
$v^\prime$ and integrates over $(0,L)$; third, one multiplies
(\ref{eq:maxwell-5}) by $v$ and integrates over $(0,L)$; fourth, in the thus obtained
equation the term $(\eta_j^\prime,v)_{L^2(0,L)}$ is integrated by parts and
the condition (\ref{eq:dtn-condition-maxwell-a}) is employed. 
%the condition $i \ome \widetilde{\lambda}_j \beta_j(L) = -\widetilde{\lambda}_j^2 \eta_j (L)$
%from Remark~\ref{rem:alternative-to-impedance-bc} is used.

\emph{Proof of (\ref{item:lemma:alpha-delta-zeta-i}):} We note
$\widetilde{\lambda}_j \sim \sqrt{\lambda_j}$ with implied constant independent of $j$.
Lemma~\ref{lemma:ihl} together with the condition $\beta_j(0) = 0$ from
(\ref{eq:initial-condition-beta}) then readily implies (\ref{eq:lemma:beta-eta-gamma-1}).
Combining (\ref{eq:lemma:beta-eta-gamma-1}) with (\ref{eq:lemma:beta-eta-gamma-10}) provides
(\ref{eq:lemma:beta-eta-gamma-2}). Finally, (\ref{eq:lemma:beta-eta-gamma-1}) and
%(\ref{eq:maxwell-6}) yield (\ref{eq:lemma:alpha-delta-zeta-3}).
(\ref{eq:maxwell-6}) yield (\ref{eq:lemma:beta-eta-gamma-3}).

\emph{Proof of (\ref{item:lemma:beta-eta-gamma-ii}):} This case is shown in the
same way as (\ref{item:lemma:beta-eta-gamma-i}) noting that for the finitely many
propagating modes one has $\lambda_j \sim 1$ so that $|\widetilde{\lambda}_j | \sim 1$.
\end{proof}
\begin{numberedproof}{of Theorem~\ref{thm:stability-maxwell}}
The stability bound (\ref{eq:thm:stability-maxwell-1}) follows from
Lemmas~\ref{lemma:alpha-delta-zeta} and \ref{lemma:beta-eta-gamma} together with the
observation that the estimates (\ref{eq:maxwell-requirement}) imply Theorem~\ref{thm:stability-maxwell}. For the estimate (\ref{eq:thm:stability-maxwell-2}), 
we note that for the pair $(\bfE, \bfH) \in D(A^{\tmaxwell})$, we may
set $(\bff,\bfg):= (\curl \bfE - \bi \ome \bfH, \curl \bfH + \bi \ome \bfE) \in L^2(\Ome)$ and then apply (\ref{eq:thm:stability-maxwell-1}). 
\end{numberedproof}

% Numerical results and conclusions
\section{Numerical results and conclusions}
\label{sec:numerical}

To test the dependence of the ultraweak DPG Maxwell discretization on the waveguide length $L$ and the scaling constant $\beta$ in the test norm, we solve the linear time-harmonic Maxwell equations in a homogeneous rectangular waveguide with transverse domain $D = (0.0, 1.0) \times (0.0, 0.5)$. At $\Gammai$, the waveguide is excited with the lowest-order transverse electric (TE) mode by prescribing the corresponding tangential electric field. At the waveguide exit $\Gammao$, the $\dtn$ operator is approximated with an impedance boundary condition relating the tangential electric field $\bfE$ with the rotated tangential magnetic field $\bfn \times \bfH$ via the impedance constant for the propagating TE mode.
The first-order impedance boundary condition works well for the case of single-mode propagation and can be realized with the existing trace unknowns in the UW DPG method;
for details on implementing impedance boundary conditions in UW DPG, we refer to \cite{FE_Math_book}.
For the more general case of a multi-mode waveguide, we employ a perfectly matched layer (PML) to approximate the DtN operator.
For a discussion of PMLs with the DPG method, we refer to \cite{Vaziri_Keith_Demkowicz_18}; for an analysis of a stretched-coordinate PML for optical waveguide simulations with the full envelope approximation of UW DPG Maxwell, see \cite{Henneking_Grosek_Demkowicz_24}.

The cross-section of the rectangular waveguide is modeled with two hexahedral elements, which is justified by the simple transverse mode profile of the fundamental TE mode. In the longitudinal direction, we use a ``fixed discretization'' of two elements per wavelength; as the waveguide length $L$ increases, the number of elements per wavelength remains the same. The DPG discretization uses uniform polynomial order $p=5$ with enriched test functions of order $p + 1$. All of the computations were done in $hp$3D \cite{Henneking_Demkowicz_hpUserManual, Henneking_Demkowicz_book, Henneking_24_hp3d}.
For additional details about the implementation of the UW DPG Maxwell problem, we refer to \cite{Henneking_phd}.

\begin{figure}[tb]
	\centering
%	\begin{subfigure}{0.55\textwidth}
%		\includegraphics[width=\textwidth,trim={0 0 25pt 0}]
%		{p5error_beta_plot1.pdf}
%	\end{subfigure}
	\begin{subfigure}{0.49\textwidth}
		\includegraphics[width=\textwidth]
		{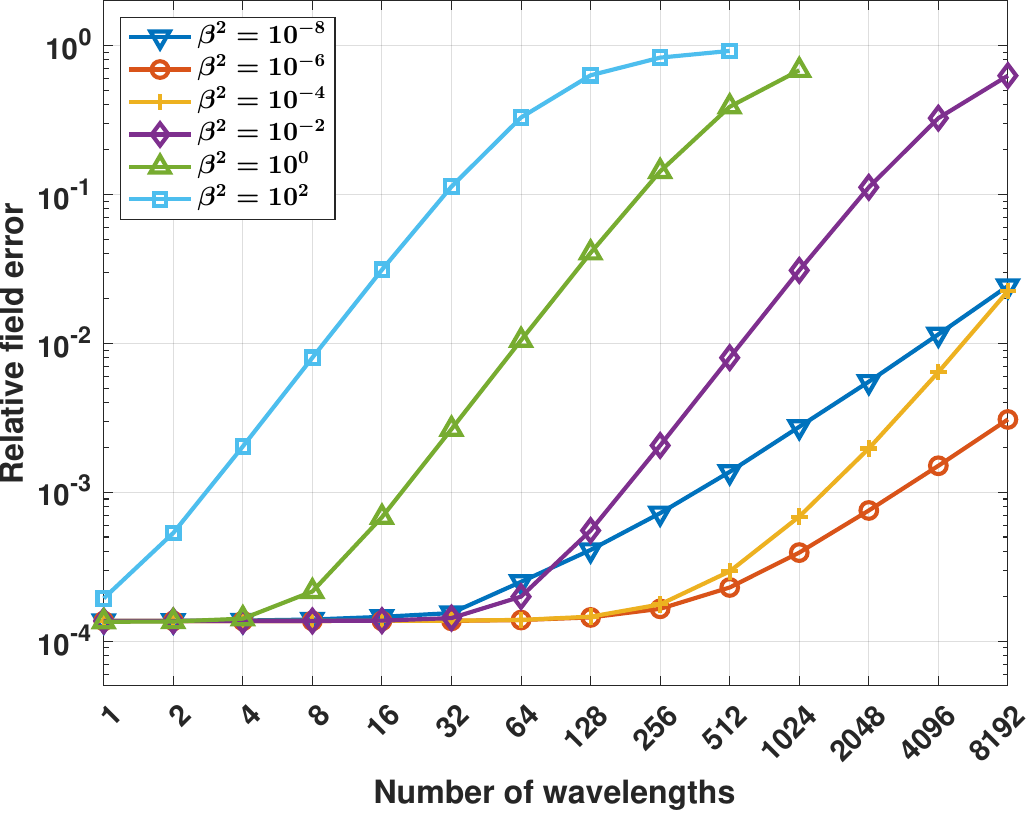}
	\end{subfigure}
	\begin{subfigure}{0.49\textwidth}
		\includegraphics[width=\textwidth]{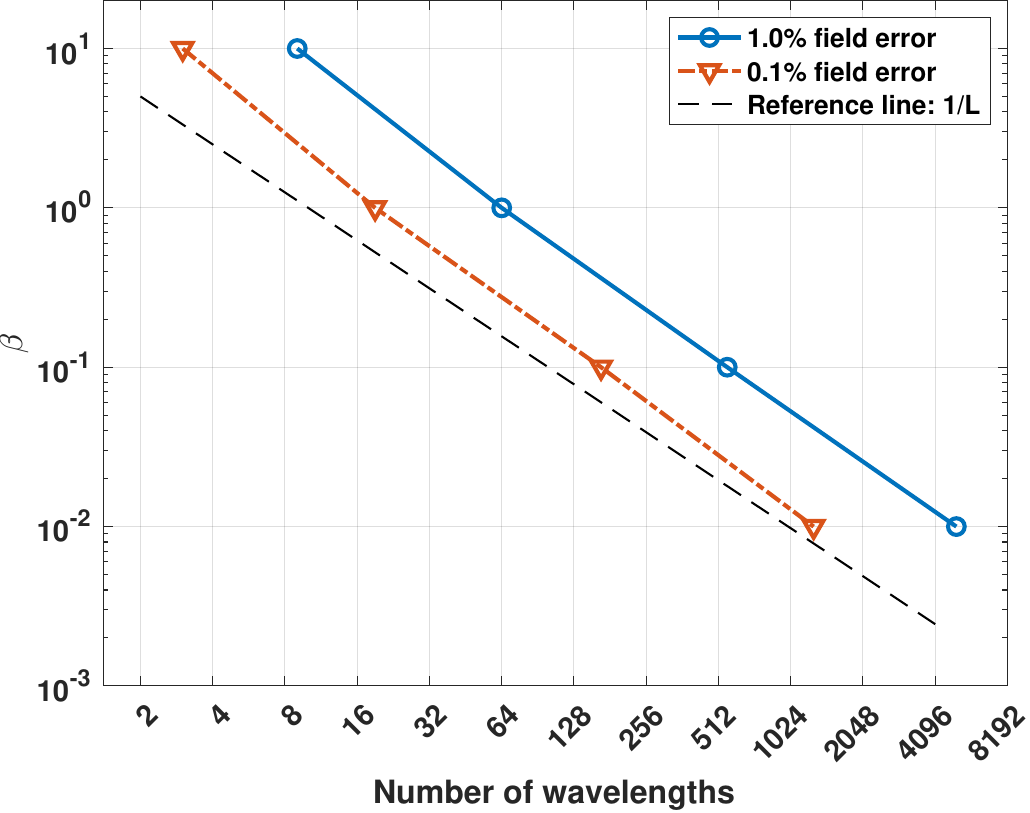}
	\end{subfigure}
	\caption{Propagation of the fundamental TE mode in a homogeneous 3D rectangular waveguide with uniformly fifth-order discretization using a fixed number of elements per wavelength. Choosing a small scaling constant $\beta$ in the DPG test norm improves the stability of the method significantly, but choosing $\beta$ too small results in round-off errors that adversely affect the solution. Left: Relative field error for different choices of $\beta$. Right: Dependence of $\beta$ on the waveguide length $L$ to counter stability loss and maintain a certain relative accuracy of the solution.}
	\label{fig:p5error_beta}
\end{figure}

We recall the scaled adjoint graph norm from Section~\ref{sec:model-problems}:
\[
	\| v \|^2_{H_{A^*}(\Omega_h)} =
	\| A_h^* v \|^2 +
	\beta^2 \| v \|^2 .
\]

The analysis of Lemma~\ref{lem:uw-inf-sup} and  Theorem~\ref{thm:stability-maxwell} suggests scaling $\beta$ with $L^{-1}$ to maintain stability of the method as the waveguide length $L$ increases. Figure~\ref{fig:p5error_beta} (left) shows the relative $L^2$-error of the electric field $\bfE$ for various choices of $\beta$ and increasing waveguide length $L$ (here expressed in terms of the total number of wavelengths). The numerical results confirm that the choice of $\beta$ indeed significantly affects the stability of the discretization.
Figure~\ref{fig:p5error_beta} (right) illustrates the dependence of $\beta$ on the waveguide length $L$ to counter the loss of stability and maintain a certain relative accuracy of the solution.
Specifically, the plot shows the maximum waveguide length $L$ (expressed in terms of number of wavelengths) that can be solved with a given value of $\beta$ while maintaining a relative error of the numerical solution of less than $1.0\%$ or $0.1\%$.
For $\beta \in [10^{-2}, 10]$, these curves follow roughly the expected dependence of $L^{-1}$ plotted as a reference line. 
Note that the stability analysis of the preceding sections, which predicts this $L^{-1}$ dependence, considered the \emph{ideal} DPG method, whereas the numerical results presented here are computed with the \emph{practical} DPG method which only approximates the optimal test functions.
For small values of $\beta$, rounding errors become dominant in the numerical computations, which limits the extent to which stability loss can be countered in practice.
In our experiments, round-off errors became dominant for $\beta < 10^{-3}$, cf.~Figure~\ref{fig:p5error_beta} (left).
Ultimately, the numerical experiments substantiate the theoretical analysis that by choosing $\beta$ sufficiently small, UW DPG can compensate (to a certain extent) the loss of stability due to the $L^{-1}$ dependence of the boundedness-below constant.

\subsection*{Conclusions. Non-homogeneous waveguide}
The convergence analysis for the DPG method for the modified Maxwell model resulting from a full envelope ansatz led us to the stability analysis for EM waveguides. We started with a simpler, acoustic waveguide, and then proceeded to the homogeneous Maxwell model. For both models, we have obtained the same result: the stability constant depends linearly on the length of the waveguide, i.e.,
\[
	\Vert \bfE \Vert + \Vert \bfH \Vert  \leq C L \left(\Vert \bfnab \times \bfE - \bi \omega \bfH \Vert  +
	\Vert \bfnab \times \bfH + \bi \omega \bfE \Vert  \right) \, , \qquad C > 0 \, .
\]
For the non-homogeneous optical waveguide with  $\mu=1$, $\epsilon = 1 + \delta \epsilon$, the triangle inequality implies:
\begin{align*}
	\Vert \bfE \Vert  + \Vert \bfH \Vert
	& \leq C L \left(\Vert \bfnab \times \bfE - \bi \omega \bfH \Vert + 
	\Vert \bfnab \times \bfH + \bi \omega (1 + \delta \epsilon - \delta \epsilon) \bfE \Vert \right) \\
	& \leq C L \left(\Vert \bfnab \times \bfE - \bi \omega \bfH \Vert +
	\Vert \bfnab \times \bfH + \bi \omega (1 + \delta \epsilon ) \bfE \Vert + 
	\omega \Vert \delta \epsilon \Vert_{L^\infty} \Vert \bfE \Vert \right) \, .
\end{align*}
Consequently,
\[
	(1 - CL \omega \Vert \delta \epsilon \Vert_{L^\infty (\Omega)} ) (\Vert \bfE \Vert + \Vert \bfH \Vert )
	\leq C L \left(\Vert \bfnab \times \bfE - \bi \omega \bfH \Vert +
	\Vert \bfnab \times \bfH + \bi \omega (1 + \delta \epsilon ) \bfE \Vert \right) \, ,
\]
which proves that, for a sufficiently small perturbation $\delta \epsilon$, the operator remains bounded below. However, the information about the linear dependence of the stability constant upon $L$ is lost.\footnote{$CL / (1 - CL \omega\,  \|\delta \epsilon\|_{L^\infty}) \approx CL ( 1 + CL \omega \, \|\delta \epsilon\|_{L^\infty})$.} In fact, the larger $L$, the smaller $\delta \epsilon$  has to be. In the second part of this work \cite{Demkowicz_Melenk_Badger_Henneking_24}, we generalize our results to the case of non-homogeneous waveguides using perturbation theory.

%% ===== Bibliography ===== %
%\printbibliography[heading=bibintoc]
%\setcounter{section}{0}
% ===== ========= ===== %

% Appendices
\appendix
%-------------------------------------------------
\section{Proof of Lemma~\ref{lemma:1D-inf-sup}}
\label{appendix:proof-of-lemma-1D-inf-sup}
%-------------------------------------------------
For the convenience of the reader and to clarify some arguments from \cite{MST20}  we provide some details of the proof of Lemma~\ref{lemma:1D-inf-sup}.
To simply the notation, we write in the present section
\begin{align*}
a^{1D}_{\kappa} (u,v) & := (u^\prime,v^\prime)_{L^2(0,1)} + \kappa^2 (u,v)_{L^2(0,1)} + \kappa u(1) \overline{v}(1), \\
\|u\|^2_{1,|\kappa|}& := \|u^\prime\|^2_{L^2(0,1)} + |\kappa|^2 \|u\|^2_{L^2(0,1)},  \\
\|u\|^2_{1,\kappa,\sim}
& := \frac{\Re \kappa}{|\kappa|} \|u\|^2_{1,|\kappa|} + |\kappa| |u(1)|^2 
\!=\! \frac{\Re \kappa}{|\kappa|} \|u^\prime\|^2_{L^2(0,1)} + \Re \kappa |\kappa| \|u\|^2_{L^2(0,1)} + |\kappa| |u(1)|^2. 
\end{align*}
\begin{lemma}[\protect{\cite[Lem.~{4.1}]{MST20}}]
\label{lemma:MST-lemma-4.1}
For any $\kappa \in \JMM{\doubleIC}$ with $\Re \kappa > 0$, the sesquilinear form $a^{1D}_\kappa$ satisfies for any subspace $\{0\} \ne V \subset H^1(0,1)$
\begin{align*}
\inf_{u \in V \setminus \{0\}} \sup_{v \in V \setminus \{0\} } \frac{\Re a^{1D}_\kappa (u,v)}{\|u\|_{1,|\kappa|} \|v\|_{1,|\kappa|}}
\ge \frac{\Re \kappa}{|\kappa|},
\qquad 
\inf_{u \in V \setminus \{0\}} \sup_{v \in V \setminus \{0\} } \frac{\Re a^{1D}_\kappa (u,v)}{\|u\|_{1,\kappa,\sim} \|v\|_{1,\kappa,\sim}}
\ge 1. 
\end{align*}
\end{lemma}
\begin{proof}
Given $u \in V$ take $v:= \frac{\kappa}{|\kappa|} u \in V$ and compute
\begin{align*}
\Re a^{1D}_\kappa(u, v) & =
\Re \left( \frac{\overline{\kappa}}{|\kappa|} \|u^\prime\|^2_{L^2(0,1)} + \kappa |\kappa| \|u\|^2_{L^2(0,1)}  + |\kappa| |u(1)|^2 \right) \\
& \ge \frac{\Re \kappa}{|\kappa|} \|u\|^2_{1,|\kappa|} =
\frac{\Re \kappa}{|\kappa|} \|u\|_{1,|\kappa|} \|v\|_{1,|\kappa|}.
\end{align*}
The same choice $v$ also shows the inf-sup condition for the norm $\|\cdot\|_{1,\kappa,\sim}$. 
\end{proof}
The following Lemma~\ref{lemma:MST-lemma-4.2} provides a corrected proof (restricted to the presently considered 1D situation) of \cite[Lemma~{4.2}]{MST20}. 
%The statement of \cite[Lem.~{4.2}]{MST20} is, however, correct.  
\begin{lemma}[\protect{\cite[Lem.~{4.2}]{MST20}}]
\label{lemma:MST-lemma-4.2}
Let $V  \in \{H^1(0,1), H^1_{(0}(0,1)\}$. 
Let $\kappa_0 > 0$. 
There is $C > 0$ such that for all $\kappa \in \doubleIC_{\ge 0}$ with $|\kappa| \ge \kappa_0$ and all
$f \in L^2(0,1)$, $g \in \doubleIC$ the
solution $u \in V$ of
\begin{equation}
\label{eq:lemma:MST-lemma-4.2-weak-formulation}
a^{1D}_{\kappa}(u,v) = (f,v)_{L^2(0,1)} + g \overline{v}(1) \quad \forall v \in V
\end{equation}
satisfies the stability bound
\begin{equation}
\label{eq:lemma:MST-lemma-4.2-final}
%\|u\|_{1,|\kappa|} \leq C \left[ \frac{1}{1 + \Re \kappa} \|f\|_{L^2(0,1)} + \frac{1}{\sqrt{1+\Re \kappa}} |g|\right].
\|u\|_{1,|\kappa|} \leq C \left[ \frac{1}{1+\Re \kappa} \|f\|_{L^2(0,1)} +  \frac{1}{\sqrt{1+\Re \kappa}} |g|\right].
\end{equation}
\end{lemma}
\begin{proof}
This is taken from \cite[Lem.~{4.2}]{MST20}, where the multidimensional case is covered. We repeat the arguments, which are
based on the multiplier technique worked out in detail in
\cite[Lem.~{4.2}]{MST_arxiv}, \cite[Cor.~{2.11}]{GaGrSp:15}, \cite[Prop.~{8.1.4}]{MelenkDiss}.
We assume $\Re \kappa > 0$ as the case $\Re \kappa = 0$ has been treated in, e.g., \cite[Prop.~{8.1.4}]{MelenkDiss} and could, alternatively, 
be obtained from the arguments below in the limit $\Re \kappa \rightarrow 0$. 
%By the inf-sup condition satisfied by $a^{1D}_\kappa$ (cf.\ Lemma~\ref{lemma:MST-lemma-4.1}) 
%the solution $u$ satisfies 
%\begin{align}
%\label{eq:lemma:MST-lemma-4.2-5}
%\|u\|_{1,|\kappa|} &\leq \frac{|\kappa|}{\Re \kappa} \sup_{v \in V} \frac{\left| (f,v)_{L^2} + g \overline{v}(1)\right|}{\|v\|_{1,|\kappa|} } 
%\leq \frac{|\kappa|}{\Re \kappa} \left[ |\kappa|^{-1} \|f\|_{L^2(0,1)} + \frac{C}{|\kappa|^{1/2}} |g| \right], 
%\end{align}
%where we additionally used the multiplicative trace inequality $|v(1)| \leq C \|v\|_{L^2} \|v\|_{H^1}$.

The second inf-sup condition in Lemma~\ref{lemma:MST-lemma-4.1} implies 
\begin{align*}
\|u\|_{1,\kappa,\sim} \leq \sup_{v \in V} \frac{\left| (f,v)_{L^2} + g \overline{v}(1)\right|}{\|v\|_{1,\kappa,\sim} } 
\leq \frac{\|f\|_{L^2(0,1)}}{\sqrt{\Re \kappa |\kappa|}} + \frac{|g|}{\sqrt{|\kappa|}} 
\end{align*}
so that we get 
\begin{align}
\label{eq:lemma:MST-lemma-4.2-21}
\|u^\prime\|_{L^2(0,1)} + |\kappa| \|u\|_{L^2(0,1)} + \frac{|\kappa|}{\sqrt{\Re \kappa}} |u(1)| 
& \lesssim \frac{\|f\|_{L^2(0,1)}}{\Re \kappa} + \frac{|g|}{\sqrt{\Re \kappa}}. 
\end{align} 
\emph{Step 1 (the case $\Re \kappa \ge 1$):}
For $\Re \kappa \ge 1$, (\ref{eq:lemma:MST-lemma-4.2-21}) immediately implies 
(\ref{eq:lemma:MST-lemma-4.2-final}). 
%\begin{align*}
%\|u\|^2_{1,|\kappa|} &\leq
%C_\beta \left[ \|f\|_{L^2(0,1)} \|u\|_{L^2(0,1)} + |g| |u(1)|\right]
%\leq C_\beta \left[ \frac{1}{|\kappa|} \|f\|_{L^2(0,1)} + \frac{1}{\sqrt{|\kappa|}} |g| \right] \|u\|_{1,|\kappa|}.
%\end{align*}

For the case $\Re \kappa \leq 1$ we fix $\beta > 1$ and 
distinguish between the cases $|\Im \kappa| \leq \beta \Re \kappa$ and $|\Im \kappa| \ge \beta \Re \kappa$. 

\emph{Step 2 (the case $\Re \kappa \leq 1$ with $|\Im \kappa | \leq \beta \Re \kappa$):}
{}From $\kappa_0^2 \leq |\kappa|^2 \leq (1 + \beta^2) (\Re \kappa)^2$ we infer $\kappa_0 (1+\beta^2)^{-1/2} \leq \Re \kappa \leq 1$
so that 
(\ref{eq:lemma:MST-lemma-4.2-21}) again implies (\ref{eq:lemma:MST-lemma-4.2-final}). 

\emph{Step 3 (the case $\Re \kappa \leq 1$ with $|\Im \kappa | \ge \beta \Re \kappa$):} We have 
\begin{align}
\label{eq:lemma:MST-lemma-4.2-25}
|\Im \kappa| &\leq |\kappa| \leq \sqrt{1+\beta^{-2}} |\Im \kappa|, \\
\label{eq:lemma:MST-lemma-4.2-27}
 -\Re \kappa^2 & = -(\Re \kappa)^2 + (\Im \kappa)^2  \ge \frac{1-\beta^{-2}}{1+\beta^{-2}} |\kappa|^2 \sim |\kappa|^2. 
\end{align}
%
%Taking as the test function $v = u$ gives by taking the real part
%\begin{align}
%\label{eq:lemma:MST-lemma-4.2-10}
%\|u^\prime\|^2_{L^2(0,1)} + \Re \kappa^2 \|u\|^2_{L^2(0,1)} + \Re \kappa |u(1)|^2 &\lesssim \|f\|_{L^2(0,1)} \|u\|_{L^2(0,1)} + |g| |u(1)|.
%\end{align}
Taking as the test function $v = (\Im \kappa) u$ in (\ref{eq:lemma:MST-lemma-4.2-weak-formulation}) and taking the imaginary part gives 
\begin{align}
\label{eq:lemma:MST-lemma-4.2-20}
2 \underbrace{|\Im \kappa| \Re \kappa}_{\ge 0} \|u\|^2_{L^2(0,1)} + |\Im \kappa| |u(1)|^2 &\lesssim \|f\|_{L^2(0,1)} \|u\|_{L^2(0,1)} + |g| |u(1)|.
\end{align}
The multiplier technique amounts to taking as a test function $v = x u^\prime$. Using %the relations
\begin{align*}
\Re u^\prime (x \overline{u}^\prime)^\prime & = \frac{1}{2} x \left(|u^\prime|^2\right)^\prime + |u^\prime|^2 ,
&
\Re u x \overline{u}^\prime & = \frac{1}{2} x \left(|u|^2\right)^\prime,
%& \Im u x \overline{u} & = 0, 
\end{align*}
we get  \JMM{with integration by parts} 
\begin{align*}
& \Re a^{1D}_{\kappa}(u,x u^\prime)  = \Re (u^\prime, (x u^\prime)^\prime)_{L^2(0,1)} + \Re \left(\kappa^2 (u, x u^{\prime})_{L^2(0,1)}\right) +
\Re \left(\kappa u(1) \overline{u}^\prime(1)\right)  = \\
& \|u^\prime\|^2_{L^2(0,1)} - \frac{1}{2} \|u^\prime\|^2_{L^2(0,1)} + \frac{1}{2}| u^\prime(1)|^2
+ \frac{1}{2} \left(\Re \kappa^2\right) \left( |u(1)|^2 - \|u\|^2_{L^2(0,1)}\right)  \\
& \qquad -(\Im \kappa^2) \Im (u, x u^\prime)_{L^2(0,1)} + \Re \left(\kappa u(1) \overline{u}^\prime(1)\right).
\end{align*}
Hence, by inserting this in the real part of the weak formulation (\ref{eq:lemma:MST-lemma-4.2-weak-formulation}) for $u$, we get 
\begin{align}
\nonumber 
\|u^\prime\|^2_{L^2(0,1)} &- \Re \kappa^2  \|u\|^2_{L^2(0,1)}  + |u^\prime(1)|^2 
\leq
2 \Bigl( \|f\|_{L^2(0,1)} \|x u^\prime\|_{L^2(0,1)} + |g| |u^\prime(1)| \\
\label{eq:lemma:MST-lemma-4.2-50}
&  +\JMM{|\kappa|}  |u(1)| |u^\prime(1)| 
- \frac{1}{2} \Re \kappa^2 |u(1)|^2 + 2 |\Im \kappa| \Re \kappa \|u\|_{L^2(0,1)}\|x u^\prime\|_{L^2(0,1)} 
\Bigr). 
\end{align} 
Applying 
Young's inequality in
(\ref{eq:lemma:MST-lemma-4.2-20})
and the observation $|\kappa| \sim |\Im \kappa|$ from 
(\ref{eq:lemma:MST-lemma-4.2-25}) give
\begin{align}
\label{eq:lemma:MST-lemma-4.2-1000} 
|\kappa|^2 |u(1)|^2 \lesssim \|f\|_{L^2(0,1)} |\kappa| \|u\|_{L^2(0,1)} + |g|^2.
\end{align}
Inserting (\ref{eq:lemma:MST-lemma-4.2-1000})
in (\ref{eq:lemma:MST-lemma-4.2-50}), 
observing (\ref{eq:lemma:MST-lemma-4.2-27}) and 
$\Im \kappa \sim |\kappa|$, and Young's inequality give 
\begin{align}
\label{eq:lemma:MST-lemma-4.2-100} 
\|u^\prime\|^2_{L^2(0,1)} & + |\kappa|^2  \|u\|^2_{L^2(0,1)}  + |u^\prime(1)|^2  \\
\nonumber 
& \lesssim \|f\|^2_{L^2(0,1)} + |g|^2 + |\Im \kappa| \Re \kappa \|u\|_{L^2(0,1)} \|u^\prime\|_{L^2(0,1)}. 
\end{align}
Next, we distinguish the cases $(f,g) = (f,0)$ and $(f,g) = (0,g)$ and denote the corresponding 
solutions by $u_f$ and $u_g$. For $u_f$, we estimate 
\begin{align}
\label{eq:lemma:MST-lemma-4.2-150} 
|\Im \kappa| \Re \kappa \|u_f\|_{L^2(0,1)} \|u_f^\prime\|_{L^2(0,1)}  
\stackrel{(\ref{eq:lemma:MST-lemma-4.2-21})}{\lesssim} |\Im \kappa| \|u_f\|_{L^2(0,1)}  \|f\|_{L^2(0,1)}. 
\end{align}
Inserting (\ref{eq:lemma:MST-lemma-4.2-150}) in (\ref{eq:lemma:MST-lemma-4.2-100})
and using Young's inequality again to kick the term $|\Im \kappa| \|u_f\|_{L^2(0,1)}$ back to the left-hand side
of (\ref{eq:lemma:MST-lemma-4.2-100}) gives 
\begin{align}
\label{eq:lemma:MST-lemma-4.2-200}
\|u_f^\prime\|^2_{L^2(0,1)} + |\kappa|^2  \|u_f\|^2_{L^2(0,1)}  + |u_f^\prime(1)|^2 
\lesssim \|f\|^2_{L^2(0,1)}.  
\end{align}
Turning to $u_g$, we estimate
\begin{align*}
& |\Im \kappa| \Re \kappa \|u_g\|_{L^2(0,1)} \|u_g^\prime\|_{L^2(0,1)}  = 
\sqrt{|\Im \kappa| \Re \kappa} 
\sqrt{|\Im \kappa| \Re \kappa} \|u_g\|_{L^2(0,1)} \|u^\prime_g\|_{L^2(0,1)}   \\
& \stackrel{ (\ref{eq:lemma:MST-lemma-4.2-20})}{\lesssim} 
\sqrt{|\Im \kappa| \Re \kappa} \sqrt{|g| |u_g(1)|} \|u^\prime_g\|_{L^2(0,1)} 
\stackrel{(\ref{eq:lemma:MST-lemma-4.2-21})}{\lesssim} 
\sqrt{|\Im \kappa| \Re \kappa} \frac{|g|}{\sqrt{|\kappa|}} \|u^\prime_g\|_{L^2(0,1)} 
\stackrel{(\ref{eq:lemma:MST-lemma-4.2-21})}{\lesssim} 
|g|^2. 
\end{align*}
Inserting this in 
(\ref{eq:lemma:MST-lemma-4.2-100}) yields 
\begin{align}
\label{eq:lemma:MST-lemma-4.2-300}
\|u_g^\prime\|^2_{L^2(0,1)} + |\kappa|^2  \|u_g\|^2_{L^2(0,1)}  + |u_g^\prime(1)|^2 
\lesssim |g|^2.  
\end{align}
Since $u = u_f + u_g$, the combination of (\ref{eq:lemma:MST-lemma-4.2-200}) and (\ref{eq:lemma:MST-lemma-4.2-300})
yields the desired estimate. 
\end{proof}
We are now in position to prove Lemma~\ref{lemma:1D-inf-sup}:

\begin{numberedproof}{of Lemma~\ref{lemma:1D-inf-sup}}
We follow \cite[Thm.~{4.3}]{MST20}. For $\Re \kappa \ge 1/\sqrt{2}$ we have by Lemma~\ref{lemma:MST-lemma-4.1}
\begin{align*}
\inf_{u \in V} \sup_{v \in V} \frac{\Re a^{1D}_\kappa(u,v)}{\|u\|_{1,|\kappa|} \|v\|_{1,|\kappa|}} \ge
\frac{\Re \kappa}{|\kappa|}  = \frac{\Re \kappa}{\sqrt{(\Re \kappa)^2 + (\Im \kappa)^2}} \ge
\frac{\Re \kappa }{\Re \kappa + |\Im \kappa|} \ge \frac{1}{1 + c \frac{|\Im \kappa|}{1 + \Re \kappa}}
\end{align*}
with $c = 1 + \sqrt{2}$. 
Let $\Re \kappa \leq 1/\sqrt{2}$. Given $u \in V$, we seek $v$ in the form $v = u + z$ with $z$ solving
\begin{equation*}
a^{1D}_{\overline{\kappa}}(z,w) = \alpha^2 (u,w)_{L^2(0,1)} \quad \forall w \in V,
\qquad \alpha^2 := |\kappa|^2 - {\kappa}^2  = - 2 i (\Im \kappa) \kappa.
\end{equation*}
By Lemma~\ref{lemma:MST-lemma-4.2}, we get
\begin{equation*}
\|z\|_{1,|\kappa|} \leq C |\alpha|^2 \|u\|_{L^2(0,1)} \leq C |\Im \kappa| |\kappa| \|u\|_{L^2(0,1)}
\end{equation*}
and note
\begin{align*}
\Re a^{1D}_{\kappa}(u,u+z) & = \|u^\prime\|^2_{L^2(0,1)} + \Re \kappa^2 \|u\|^2_{L^2(0,1)} + \Re \kappa |u(1)|^2 + \Re a^{1D}_\kappa(u,z)  \\
                      &=  \|u^\prime\|^2_{L^2(0,1)} + \Re \kappa^2 \|u\|^2_{L^2(0,1)} + \Re \kappa |u(1)|^2 + \Re a^{1D}_{\overline{\kappa}}(z,u)   \\
                      &=  \|u^\prime\|^2_{L^2(0,1)} + \Re \kappa^2 \|u\|^2_{L^2(0,1)} + \Re \kappa |u(1)|^2 + \Re \alpha^2 \|u\|^2_{L^2(0,1)}\\
                      &=  \|u\|^2_{1,|\kappa|} +  \underbrace{\Re \kappa}_{\ge 0} |u(1)|^2 \ge \|u\|^2_{1,|\kappa|}
\end{align*}
as well as
\begin{equation*}
\|u + z\|_{1,|\kappa|} \leq \|u\|_{1,|\kappa|} + \|z\|_{1,|\kappa|} %\leq (1 + C |\Im \kappa|) |\kappa| \|u\|_{L^2(0,1)}
\leq (1 + C |\Im \kappa|) \|u\|_{1,|\kappa|}. 
\end{equation*}
Hence, we arrive at
\begin{align*}
\Re a^{1D}_{\kappa}(u,v) =
\Re a^{1D}_{\kappa}(u,u+z)  \ge \|u\|^2_{1,|\kappa|}
\ge  \frac{\|u\|_{1,|\kappa|} \|u + z\|_{1,|\kappa|}}  {1 + C |\Im \kappa|}
\ge  \frac{\|u\|_{1,|\kappa|} \|u + z\|_{1,|\kappa|}}  {1 + C \frac{|\Im \kappa|}{1+\Re \kappa}}.
\end{align*}
This concludes the proof. 
\end{numberedproof}
\section{Reduction of the first-order system to second order equations for Maxwell's equations}
%------------------------------------------
%------------------------------------------
\subsection{The system for \texorpdfstring{$\bs{\alpha_i}$, $\bs{\delta_i}$, $\bs{\zeta_i}$}{alpha-i, delta-i, zeta-i}}
\label{appendix:alpha-delta-zeta}
%------------------------------------------
We consider (\ref{eq:maxwell-1}), (\ref{eq:maxwell-3}), (\ref{eq:maxwell-4}),
viz., (skipping the subscript $i$)
\begin{align}
\label{eq:maxwell-1-hinten}
\alpha^\prime - \bi \ome \delta & = f_1 , \\
\label{eq:maxwell-3-hinten}
\alpha - \bi \ome \mu^{-1} \zeta & = f_3 , \\
\label{eq:maxwell-4-hinten}
- \delta^\prime + \zeta + \bi \ome \alpha & = g_1 .
\end{align}
Multiplying (\ref{eq:maxwell-3-hinten}) with $\mu$ and (\ref{eq:maxwell-4-hinten}) with $\bi\ome$ and adding the equations, 
we eliminate $\zeta$ and obtain 
\begin{align}
\label{eq:maxwell-1-hinten-1}
\alpha^\prime - \bi \ome \delta & = f_1 , \\
\label{eq:maxwell-3-hinten-1}
- \bi \ome \delta^\prime + \underbrace{(\mu - \ome^2)}_{=:\widetilde{\mu}^2} \alpha & = \bi \ome g_1 + \mu f_3 .
\end{align}
On the interval $(L,\infty)$, this system corresponds to the right-hand sides
$f_1 = f_3 = g_1 = 0$. By eliminating $\delta$ we get
$$
\alpha^{\prime\prime} - \widetilde{\mu}^2 \alpha = 0
$$
with fundamental system $e^{\pm \widetilde{\mu} z}$. The requirement
of outgoing waves implies that we seek solutions of the homogeneous
system as multiples of $e^{- \widetilde{\mu} z}$, which provides us
with the boundary condition at $z = L$
\begin{equation}
\label{eq:impedance-bc-alpha-hinten}
\alpha^\prime(L) + \widetilde{\mu} \alpha(L) = 0.
\end{equation}
In view of (\ref{eq:maxwell-1-hinten}) on $(L,\infty)$ with $f_1 = 0$
we obtain
\begin{equation}
\label{eq:alpha-delta-bc}
0 \stackrel{(\ref{eq:maxwell-1-hinten})}{=}
\alpha^\prime(L) - \bi \ome \delta(L)
\stackrel{(\ref{eq:impedance-bc-alpha-hinten})}{=} - \widetilde \mu \alpha(L) - \bi \ome \delta(L).
\end{equation}
We also have the boundary condition
\begin{equation}
\label{eq:initial-condition-alpha-hinten}
\alpha(0) = 0. 
\end{equation}
We now turn to the weak formulation.
For $v$ with $v(0) = 0$ we get by
multiplying (\ref{eq:maxwell-1-hinten-1}) by $v^\prime$ and integrating
over $(0,L)$
and by multiplying (\ref{eq:maxwell-3-hinten-1}) by $v$ and integrating
over $(0,L)$
\begin{align*}
(\alpha^\prime, v^\prime)_{L^2(0,L)} - \bi \ome (\delta,v^\prime)_{L^2(0,L)}
& = (f_1,v^\prime)_{L^2(0,L)}, \\
- \bi \ome (\delta^\prime,v)_{L^2(0,L)} + \widetilde{\mu}^2 (\alpha,v)_{L^2(0,L)} & = \bi\ome (g_1,v)_{L^2(0,L)} + \mu (f_3,v)_{L^2(0,L)}.
\end{align*}
Integrating the term $(\delta^\prime,v)_{L^2(0,L)}$ by parts and
using $v(0) = 0$ as well as $\bi\ome \delta(L) = -\widetilde{\mu} \alpha(L)$
by (\ref{eq:alpha-delta-bc})
gives
\begin{align*}
(\alpha^\prime, v^\prime)_{L^2(0,L)} - \bi \ome (\delta,v^\prime)_{L^2(0,L)}
& = (f_1,v^\prime)_{L^2(0,L)}, \\
\bi \ome (\delta,v^\prime )_{L^2(0,L)} + \widetilde{\mu} \alpha(L) \overline{v}(L) + \widetilde{\mu}^2 (\alpha,v)_{L^2(0,L)} & = \bi\ome (g_1,v)_{L^2(0,L)} + \mu (f_3,v)_{L^2(0,L)}.
\end{align*}
Adding these two equations yields
\begin{multline}
(\alpha^\prime, v^\prime)_{L^2(0,L)} + \widetilde{\mu}^2 (\alpha,v)_{L^2(0,L)}
+\widetilde{\mu} \alpha(L) \overline{v}(L) \\
 = (f_1,v^\prime)_{L^2(0,L)} +
  \bi\ome (g_1,v)_{L^2(0,L)} + \mu (f_3,v)_{L^2(0,L)}.
\end{multline}
%------------------------------------------
\subsection{The system for \texorpdfstring{$\bs{\beta_j}$, $\bs{\gamma_j}$, $\bs{\eta_j}$}{beta-j, gamma-j, eta-j}}
\label{appendix:beta-eta-gamma}
%------------------------------------------
We consider (\ref{eq:maxwell-2}), (\ref{eq:maxwell-5}), (\ref{eq:maxwell-6}),
viz., (skipping the subscript $j$)
\begin{align}
\label{eq:maxwell-2-hinten}
-\beta^\prime + \gamma - \bi \ome \eta& = f_2,  \\
\label{eq:maxwell-5-hinten}
\eta^\prime+\bi \ome \beta& = g_2,  \\
\label{eq:maxwell-6-hinten}
\eta + \bi \ome \lambda^{-1} \gamma & = g_3. 
\end{align}
Multiplying (\ref{eq:maxwell-2-hinten}) by $\bi \ome$, (\ref{eq:maxwell-6-hinten}) by $\lambda$ and subtracting the equations, we get 
%Eliminating $\gamma$ in (\ref{eq:maxwell-2-hinten}) and
%multiplying (\ref{eq:maxwell-5-hinten}) by
%$\widetilde{\lambda}^2:= {\lambda-  \omega^2}$ yields
\begin{align}
\label{eq:maxwell-2-hinten-1}
- \bi\ome \beta^\prime - \widetilde{\lambda}^2 \eta  & = \bi \ome f_2  - \lambda g_3 , \\
\label{eq:maxwell-5-hinten-1}
\widetilde{\lambda}^2 \eta^\prime + \bi \ome \widetilde{\lambda}^2 \beta & = \widetilde{\lambda}^2 g_2 ,
\end{align}
where we also multiplied 
(\ref{eq:maxwell-5-hinten}) by $\widetilde{\lambda}^2:= {\lambda-  \omega^2}$.  
On $(L,\infty)$, this system corresponds to the right-hand sides
$f_2 = g_2 = g_3 = 0$. By eliminating $\eta$ we get
$$
-\beta^{\prime\prime} + \widetilde{\lambda }^2 \beta = 0
$$
with fundamental system $e^{\pm \widetilde{\lambda} z}$. The requirement
of outgoing waves implies that we seek solutions of the homogeneous
system as multiples of $e^{- \widetilde{\lambda} z}$, which provides us
with the boundary condition at $z = L$
\begin{equation}
\label{eq:impedance-bc-beta-hinten}
\beta^\prime(L) + \widetilde{\lambda} \beta(L) = 0.
\end{equation}
In view of (\ref{eq:maxwell-2-hinten-1}) on $(L,\infty)$ with $f_2= g_3= 0$
we obtain
\begin{equation}
\label{eq:beta-eta-bc}
0 \stackrel{(\ref{eq:maxwell-2-hinten-1})}{=}
- \bi\ome \beta^\prime(L) - \widetilde{\lambda}^2  \eta(L)
\stackrel{(\ref{eq:impedance-bc-beta-hinten})}{=} \bi \ome \widetilde \lambda \beta(L) - \widetilde{\lambda}^2 \eta(L).
\end{equation}
We also have the boundary condition
\begin{equation}
\label{eq:initial-condition-beta-hinten}
\beta(0) = 0 .
\end{equation}
We now turn to the weak formulation.
For $v$ with $v(0) = 0$ we get by
multiplying (\ref{eq:maxwell-2-hinten-1}) by $v^\prime$ and integrating
over $(0,L)$
and by multiplying (\ref{eq:maxwell-5-hinten-1}) by $v$ and integrating
over $(0,L)$
\begin{align}
- \bi \ome (\beta^\prime, v^\prime)_{L^2(0,L)} - \widetilde{\lambda}^2 (\eta,v^\prime)_{L^2(0,L)}
& = \bi \ome (f_2,v^\prime)_{L^2(0,L)} - \lambda (g_3,v^\prime), \\
 \widetilde{\lambda}^2 (\eta^\prime,v)_{L^2(0,L)} + \bi \ome \widetilde{\lambda}^2 (\beta,v)_{L^2(0,L)} & = \widetilde{\lambda}^2 (g_2,v)_{L^2(0,L)}.
\end{align}
Integrating the term $(\eta^\prime,v)_{L^2(0,L)}$ by parts and
using $v(0) = 0$ as well as $\widetilde{\lambda}^2 \eta(L) =  \bi \ome \widetilde{\lambda} \beta(L)$
by (\ref{eq:beta-eta-bc})
gives
\begin{align*}
-\bi \ome (\beta^\prime, v^\prime)_{L^2(0,L)} - \widetilde{\lambda}^2 (\eta,v^\prime)_{L^2(0,L)}
& = \bi \ome (f_2,v^\prime)_{L^2(0,L)} - \lambda (g_3,v^\prime), \\
 - \widetilde{\lambda}^2 (\eta,v^\prime)_{L^2(0,L)} + \bi\ome \widetilde{\lambda} \beta(L) \overline{v}(L)
+ \bi \ome \widetilde{\lambda}^2 (\beta,v)_{L^2(0,L)} & = \widetilde{\lambda}^2 (g_2,v)_{L^2(0,L)}.
\end{align*}
Subtracting these two equations  and subsequently dividing by $\bi \omega$ yields
%\begin{align*}
%-\bi \ome (\beta^\prime, v^\prime)_{L^2(0,L)} - \bi \ome \widetilde{\lambda}^2 (\beta,v)_{L^2(0,L)} - \bi \ome \widetilde{\lambda} \beta(L) \overline{v}(L)
%& = \bi \ome (f_2,v^\prime)_{L^2(0,L)} - \lambda (g_3,v^\prime)
%- \widetilde{\lambda}^2 (g_2,v)_{L^2(0,L)},
%\end{align*}
%i.e.,
\begin{align*}
\begin{aligned}
(\beta^\prime, v^\prime)_{L^2(0,L)} + \widetilde{\lambda}^2 (\beta,v)_{L^2(0,L)} + & \widetilde{\lambda} \beta(L) \overline{v}(L) \\
 & = - (f_2,v^\prime)_{L^2(0,L)} + \frac{\lambda}{\bi\ome} (g_3,v^\prime) + \widetilde{\lambda}^2 \frac{1}{\bi\ome}(g_2,v)_{L^2(0,L)} .
\end{aligned}
\end{align*}
%------------------------------------------------------------------

%-------------------------------------------------------------------------
\section{\texorpdfstring{$\dtn$}{DtN} for the Maxwell problem in \texorpdfstring{$\Omega$}{Omega}}
%-------------------------------------------------------------------------
Recall $\Gammam:= \partial\Omega \setminus \overline{\Gammao}$ and 
$\bfH_{0,\Gammam}(\curl, \Ome) := \{\bfE \in \bfH(\curl,\Ome)\,|\, \gamma_t \bfE= 0 \mbox{ on $\Gammam$}\}$. 
We also recall the tangential component operator $\pi_t$.  
%-------------------------------------------------------------------------
\subsection{Derivation of the ODE system (\ref{eq:maxwell-ODE})}
%-------------------------------------------------------------------------
We start with
\begin{lemma}
\label{lemma:curl-of-product}
For $\psi \in H^1(D) $ and $w \in L^2(0,L)$ there holds in the sense
of distributions
\begin{align*}
\curl \left(w \left(\begin{array}{c} \bfnab \times \psi \\ 0 \end{array}\right)\right) 
&= 
\left(\begin{array}{c} w^\prime \bfnab \psi  \\ - w \Delta \psi \end{array}\right), 
& 
\curl \left(w \left(\begin{array}{c} \bfnab \psi \\ 0 \end{array}\right)\right) &= 
- \left(\begin{array}{c} w^\prime \bfnab \times \psi  \\ 0 \end{array}\right),  \\
\curl (\bfe_z \psi w) & = \left(\begin{array}{c} w \bfnab \times \psi  \\ 0 \end{array}\right). 
&&
\end{align*}
Here, the expressions $\Delta \psi$ and $w^\prime$ are understood 
in the sense of distributions (in $D$ and $(0,L)$, respectively). 
\end{lemma}
\begin{proof}
This is a calculation. 
\end{proof}
Next, we give more details about the derivation of the equations (\ref{eq:maxwell-ODE}): 
\begin{lemma}
\label{lemma:maxwell-ODE-details}
Let $\bff$, $\bfg \in L^2(\Ome)$, and let 
the functions $\bfE \in \bfH_{0,\Gammam}(\curl,\Ome)$, $\bfH \in \bfH(\curl,\Ome)$ satisfy (\ref{eq:primal-maxwell-a})--(\ref{eq:primal-maxwell-d}). 
Let the coefficients $\alpha_i$, $\beta_i$, $\gamma_i$, $\delta_i$, $\eta_i$, 
and $\zeta_i$ of the functions $\bfE$, $\bfH$ be given by (\ref{eq:ansatz}). 
Then these coefficients satisfy the ODE system (\ref{eq:maxwell-ODE})
in the sense of distributions. Additionally, the functions $\gamma_i$, $\zeta_i \in L^2(0,L)$ and 
$\alpha_i$, $\beta_i$, $\delta_i$, $\eta_i \in H^1(0,L)$ so that 
the ODE system (\ref{eq:maxwell-ODE}) holds in a weak sense and pointwise 
almost everywhere. 
\end{lemma}
\begin{proof}
Since $\bfE$, $\bfH \in L^2(\Ome)$, we have that $\alpha_i$, $\beta_i$, $\gamma_i$, $\delta_i$, $\eta_i$, $\zeta_i \in L^2(0,L)$.  
To derive the formulas (\ref{eq:maxwell-ODE}), we note that $(\bfE, \bfH)$ satisfies, for $\bfF$, $\bfG \in L^2(\Ome)$, 
\begin{align}
\label{eq:lemma:maxwell-ODE-details-10}
(\bfnab \times  \bfE, \bfF)_{L^2(\Ome)} - \bi \ome (\bfH,\bfF)_{L^2(\Ome)} & = (\bff,\bfF)_{L^2(\Ome)}, \\
\label{eq:lemma:maxwell-ODE-details-20}
(\bfnab \times  \bfH, \bfG)_{L^2(\Ome)} + \bi \ome (\bfH,\bfG)_{L^2(\Ome)} & = (\JMM{\bfg},\bfG)_{L^2(\Ome)}. 
\end{align}
The equations (\ref{eq:maxwell-ODE}) are obtained by suitably choosing the test functions $\bfF$, $\bfG$. 
We illustrate the procedure for (\ref{eq:maxwell-1}) 
%by taking $\bfF = w(z) \left(\begin{array}{c} \bfnab \psi_i \\ 0\end{array}\right)$ 
by taking $\bfF = w(z) \left(\bfnab \psi_i, 0\right)^\top$ 
with $w \in C^\infty_0(0,L)$. We note that $\bfF \in \bfH(\curl,\Ome)$ with 
$$
\bfnab \times \bfF = - \left(\begin{array}{c} w^\prime \bfnab \JMM{\times} \psi_i \\ 0 \end{array}\right) \in L^2(\Ome)
$$
by Lemma~\ref{lemma:curl-of-product}. 
In view of $w \in C^\infty_0(0,L)$ and the 
boundary conditions satisfied by $\bfE$ (cf.\ (\ref{eq:primal-maxwell-c}))
we observe $(\gamma_t \bfE, \pi_t \bfF)_{L^2(\partial\Omega)} = 0$. 
Hence an integration by parts and the fact that the expansion
of $\bfE$ converges in $L^2(\Ome)$ as well as the orthogonalities satisfied 
by the functions $\psi_i$, $\phi_i$ gives 
\begin{align}
\label{eq:lemma:maxwell-ODE-details-30}
(\bfnab \times \bfE, \bfF)_{L^2(\Ome)} & = (\bfE, \bfnab \times \bfF)_{L^2(\Ome)}  = - \int_{z=0}^L \alpha_i(z) w^\prime(z) 
\diff z. 
\end{align}
Hence, we get from 
(\ref{eq:lemma:maxwell-ODE-details-10}) and 
(\ref{eq:lemma:maxwell-ODE-details-30}) 
with the orthogonalities satisfied by the functions $\psi_i$, $\phi_i$ 
\begin{align*}
\int_{z=0}^L -\alpha_i(z) w^\prime(z) - \bi \ome \delta_i(z) w(z)\,\diff z 
&= \int_{z=0}^L w(z) (\bff, \left(\begin{array}{c} \nabla \psi_i \\ 0 \end{array}\right))_{L^2(D)}\,\diff z 
\\ = \int_{z=0}^L w(z)  f_{i,1} w(z)\,\diff z,
\end{align*}
which is the distributional form of (\ref{eq:maxwell-1}). The remaining 5 equations in (\ref{eq:maxwell-ODE}) are obtained similarly with suitable test functions; we note 
that in the choice of $\bfG$, one takes functions such that $\bfG \in \bfH_0(\curl,\Ome)$ so that an integration by parts as in 
(\ref{eq:lemma:maxwell-ODE-details-30}) is again possible. 

Since the right-hand sides $f_{i,1},\ldots,g_{i,3}$ are in $L^2(0,L)$, the system (\ref{eq:maxwell-ODE}) shows that $\alpha_i$, $\beta_i$, $\delta_i \in H^1(0,L)$. 
\end{proof}
%-------------------------------------------------------------------------
\subsection{Mapping properties of the \texorpdfstring{$\dtnm$}{DtNmw}-operator}
\label{sec:mapping-properties-dtnm}
%-------------------------------------------------------------------------
%-------------------------------------------
We note that by \cite[Prop.~{3.6}]{fernandes-gilardi97} the space 
$\{\bfu \in (C^\infty(\overline{\Ome}))^3\,|\, \gamma_t \bfu = 0 \mbox{ in a neighborhood of $\overline{\Gammam}$}\}$
is dense in $\bfH_{0,\Gammam}(\curl,\Ome)$. 

In order to define the $\dtn$-operator, we introduce 
with the eigenpairs $(\mu_i, \psi_i)_i$ and $(\lambda_j, \phi_j)_j$ 
on the space of sufficiently smooth functions the linear functionals (we identify in the canonical way 
$\Gammao \subset \doubleIR^3$ with $D \subset \doubleIR^2$, on which the eigenfunctions $\phi_i$, $\psi_j$ are defined):  
\begin{align}
\label{eq:alphahat}
\alphahat(\bfE)& := (\pi_t \bfE, \bfnab \times \psi_i)_{L^2(\Gammao)}, \\
\label{eq:betahat}
\betahat(\bfE)& := (\pi_t \bfE,  \bfnab \phi_j)_{L^2(\Gammao)}. 
\end{align}
We next show that these linear functionals can be uniquely extended to continuous linear functionals on $\bfH_{0,\Gammam}(\curl,\Ome)$: 
\begin{theorem}
\label{thm:alphai-betaj}
There is $C > 0$ depending only on $\Omega$ and $\Gammam$ such that for all $\bfE \in \bfH_{0,\Gammam}(\curl,\Ome)$ 
\begin{align*}
\sum_{i\ge 1} |\alphahat(\bfE)|^2 \mu_i^{1/2} 
+ 
\sum_{j\ge 1} |\betahat(\bfE)|^2 \lambda_j^{-1/2} &\leq C \|\bfE\|^2_{\bfH(\curl,\Ome)}. 
\end{align*}
\end{theorem}
\begin{proof}
\emph{Step 1:} By the arguments given in \cite[pp.~28/29]{buffa-ciarlet01a}, we can decompose 
$\bfE = \bfpsi + \nabla p$ with $\bfpsi \in \bfH^1(\Ome) \cap \bfH_{0,\Gammam}(\curl,\Ome)$ and 
$p \in H^1_{\Gammam}(\Omega)\JMM{:= \{p \in H^1(\Omega)\,|\, p|_{\Gammam} = 0\}}$ together with the stability estimate
\begin{align}
\label{eq:thm:alphai-betaj} 
\|\bfpsi\|_{\bfH^1(\Ome)} + \|p\|_{H^1(\Ome)} &\leq C \|\bfE\|_{\bfH(\curl,\Ome)}. 
\end{align}
\emph{Step 2:}
By the density result \cite[Prop.~{3.6}]{fernandes-gilardi97}, 
we may assume in Steps~{3}--{6} that $\bfpsi$ is smooth and vanishes 
in a neighborhood of $\Gammam$ as we only need to control
$\|\bfpsi\|_{\bfH(\curl,\Ome)}$. In Step~7, we may assume 
that $\bfpsi$ is smooth and that we control $\|\bfpsi\|_{\bfH^1(\Ome)}$. 
In Step~7, we may also assume $p|_{\Gammao} \in C^\infty_0(\Gammao)$ by the 
the standard density
of $C^\infty_0(\Gammao)$ in $\widetilde{H}^{\JMM{1/2}}(\Gammao)$. 

\emph{Step 3:} With the notation $\bfH^{1/2}_{\parallel}(\Gamma)$ of \cite[p.~{14}]{buffa-ciarlet01a}
and its dual $\bfH^{-1/2}_{\parallel}(\Gamma):= (\bfH^{1/2}_{\parallel}(\Gamma))^\prime$ 
we introduce (cf.\ \cite[p.~{22}]{buffa-ciarlet01a}) 
$\bfH^{-1/2}_{\parallel}(\operatorname{div}_\Gamma,\Gamma):= \{\bfu \in \bfH^{-1/2}_{\parallel}(\Gamma)\,|\, \operatorname{div}_\Gamma  \bfu \in H^{-1/2}(\Gamma)\}$. 
By \cite[Thm.~{3.9}]{buffa-ciarlet01a}, the mapping $\gamma_t: \bfH(\curl,\Ome) \rightarrow \bfH^{-1/2}_{\parallel}(\operatorname{div}_\Gamma,\Gamma)$ 
is continuous. Furthermore, one has 
\begin{equation}
\operatorname{div}_\Gamma \gamma_t \bfu = \curl_\Gamma \pi_t\bfu \quad \mbox{ in $H^{-1/2}(\Gamma)$}, 
\end{equation}
which follows by a straight-forward calculation for smooth $\bfu$ and a density argument. 

\emph{Step 4:} Since $\operatorname{div}_\Gamma \gamma_t \bfpsi \in H^{-1/2}(\Gamma)$ (by the continuity of $\gamma_t$) and 
$\gamma_t \bfpsi = 0$ on $\Gammam$ we get $\operatorname{div}_\Gamma \gamma_t \bfpsi = 0$ in $H^{-1/2}(\Gammam)$. By the density 
of compactly supported functions in $H^{1/2}(\Gammam)$, \cite[Thm.~{11.1}]{lions-magenes72},
we infer the stronger result 
$\operatorname{div}_\Gamma \gamma_t \bfpsi = 0$ in $(H^{1/2}(\Gammam))^\prime$. 
From \cite[Prop.~{3.3}]{fernandes-gilardi97}, we get the bound  
%the stronger result $\operatorname{div}_\Gamma \gamma_t \bfpsi \in \left(H^{1/2}(\Gammao)\right)^\prime$
%together with 
\begin{equation}
\label{eq:thm:alphai-betaj-20}
\|\operatorname{div}_\Gamma \gamma_t \bfpsi \|_{(H^{1/2}(\Gammao))^\prime} \leq C 
\|\operatorname{div}_\Gamma \gamma_t \bfpsi \|_{H^{-1/2}(\Gamma)} \leq C \|\bfpsi \|_{\bfH(\curl,\Ome)}. 
\end{equation}
\emph{Step 5:} The operators $\bfnab \times $ is a continuous mapping $H^1(\Gammao) \rightarrow L^2(\Gammao)$ as well as 
$L^2(\Gammao) \rightarrow \bfH^{-1}(\Gammao)$ so that, by interpolation, 
\begin{align}
\|\bfnab \times \psi_i\|_{\bfH^{-1/2}(\Gammao)} &\leq C \|\psi_i\|_{H^{1/2}(\Gammao)} \quad \forall i,   \\
\nonumber 
\|\bfnab \phi_j\|_{\widetilde \bfH^{-1/2}(\Gammao)} &\leq C \|\phi_j\|_{\widetilde H^{1/2}(\Gammao)} \quad \forall j.  
\end{align}
To see this, we note that the operator $\operatorname{div}$ is a continuous operator $\bfH^1(\Gammao) \rightarrow L^2(\Gammao)$ and 
$L^2(\Gammao) \rightarrow H^{-1}(\Gammao)$ so that by interpolation it is a continuous map $\bfH^{1/2}(\Gammao) \rightarrow H^{-1/2}(\Gammao)$. 
Hence, using the boundary conditions satisfied by $\phi_j$, 
\begin{align*}
\|\bfnab \phi_j\|_{\widetilde{\bfH}^{-1/2}(\Gammao)} & = 
\sup_{\bfv \in \bfH^1(\Gammao)} \frac{(\bfnab \phi_j, \bfv)_{L^2(\Gammao)}}{\|\bfv\|_{\bfH^{1/2}(\Gammao)} } 
= 
\sup_{\bfv \in \bfH^1(\Gammao)} \frac{(\phi_j, \operatorname{div} \bfv)_{L^2(\Gammao)}}{\|\bfv\|_{\bfH^{1/2}(\Gammao)} }  \\
& \leq \sup_{\bfv \in \bfH^1(\Gammao)} \|\phi_j\|_{\widetilde{H}^{1/2}(\Gammao)}\frac{\|\operatorname{div} \bfv \|_{H^{-1/2}(\Gammao)}}{\|\bfv\|_{\bfH^{1/2}(\Gammao)}}
\leq C \|\phi_j\|_{\widetilde{H}^{1/2}(\Gammao)}. 
\end{align*}
\emph{Step 6:} (estimating $\alphahat$)
{}From the smoothness of $\bfpsi$ and since it vanishes in a neighborhood of $\partial\Gammam$, we get 
\begin{align*}
 \alphahat(\bfpsi)  & =  (\pi_t \bfpsi,\bfnab \times \psi_i)_{L^2(\Gammao)}
= (\curl_\Gamma \pi_t \bfpsi, \psi_i)_{L^2(\Gammao)}
= (\operatorname{div}_\Gamma \gamma_t \bfpsi, \psi_i)_{L^2(\Gammao)}.   
%\leq \|\operatorname{div}_\Gamma \gamma_t \bfpsi\|_{\widetilde{H}^{-1/2}(\Gammao)} \|\psi_i\|_{H^{1/2}(\Gammao)}.  
%&\stackrel{(\ref{eq:thm:alphai-betaj-20})}{\lesssim} \|\bfpsi\|_{\bfH_\Gammam}(\curl,\Ome}  
\end{align*} 
Hence, the coefficients $\alphahat(\bfpsi)$ are the coefficients obtained by expanding $\operatorname{div}_\Gamma \gamma_t \bfpsi$ in 
terms of the functions $(\psi_i)_i$.  Since $\|\psi_i\|_{L^2(\Gammao)}  = \mu_i^{-1/2}$ and using (\ref{eq:thm:alphai-betaj-20}), we get 
\begin{align*}
\sum_{i\ge 1} |\alphahat(\bfpsi)|^2 \mu_i^{1/2} & \sim \|\operatorname{div}_\Gamma \gamma_t \bfpsi\|^2_{\widetilde{H}^{-1/2}(\Gammao)} 
\stackrel{(\ref{eq:thm:alphai-betaj-20})}{\lesssim} \|\bfpsi\|^2_{\bfH(\curl,\Ome)}. 
\end{align*}
Since $p \in \widetilde{H}^{1/2}(\Gammao)$, we compute with an integration by parts 
\begin{align*} 
\alphahat(\nabla p) & = (\nabla p, \bfnab \times \psi_i)_{L^2(\Gammao)} = 0. 
\end{align*}
\emph{Step 7:} (estimating $\betahat$) 
%Using that $\phi_j \in \widetilde{H}^{1/2}(\Gammao)$, an integration by parts gives 
Using that $\phi_j \in \widetilde{H}^{1}(\Gammao)$, an integration by parts gives 
\begin{align*}
\betahat(\bfpsi) & = (\pi_t \bfpsi, \nabla \phi_j)_{L^2(\Gammao)}  = - (\operatorname{div}_\Gamma \pi_t \bfpsi, \phi_j)_{L^2(\Gammao)}. 
\end{align*}
Hence, the coefficients $\betahat(\bfpsi)$ are the coefficients obtained by expanding $-\operatorname{div}_\Gamma \pi_t \bfpsi$ in 
terms of the $(\phi_j)_j$.  Since $\|\phi_j\|_{L^2(\Gammao)}  = \lambda_j^{-1/2}$, we get  
\begin{align*}
\sum_{j\ge 1} |\betahat(\bfpsi)|^2 \lambda_j^{\JMM{1/2}} &\sim  \|\operatorname{div}_\Gamma \pi_t \bfpsi \|^2_{H^{-1/2}(\Gammao)} 
\!\!\stackrel{\operatorname{div}:H^{1/2} \rightarrow H^{-1/2}}{\lesssim} \!\!\! \|\pi_t \bfpsi\|^2_{H^{1/2}(\Gammao)} 
\lesssim \|\bfpsi\|^2_{H^1(\Ome)},  
\end{align*}
which is a stronger estimate than needed. 
Finally, we note that $(\nabla p, \nabla \phi_j)_{L^2(\Gammao)} = \lambda_j (p,\phi_j)_{L^2(\Gammao)}$ so that 
\begin{align*}
\sum_{j\ge 1} |\betahat(\nabla p)|^2 \lambda_j^{-1/2} & = 
\sum_{j\ge 1}  \lambda_j^{3/2} |(p, \phi_j)_{L^2(\Gammao)}|^2  \sim \|p\|^2_{\widetilde{H}^{1/2}(\Gammao)} \lesssim \|p\|^2_{H^1(\Ome)}. 
\end{align*}
\end{proof}
The $\dtnm$-operator (\ref{eq:dtnm}) takes sequences $(\alpha_i)_i$, and $(\beta_j)_j$ 
and forms a new function on $\Gammao$. To see that this function is the 
trace of an $\bfH(\curl,\Ome)$ function, we need the following result: 
\begin{lemma}
\label{lemma:alphai-betaj-reconstruction}
Let $(\alpha_i)_i$ and $(\beta_j)_j$ be two sequences with 
$\sum_{i\ge 1} |\alpha_i|^2 \mu_i^{1/2} < \infty$, 
$\sum_{j\ge 1} |\beta_j|^2 \lambda_j^{-1/2} < \infty$. Then there exists $\bfE \in \bfH_{0,\Gammam}(\curl,\Omega)$ and $\bfH \in \bfH(\curl,\Omega)$ 
such that 
\begin{align*}
\pi_t \bfE & = \sum_{i\ge 1} \alpha_i \bfnab \times \psi_i + \sum_{j\ge 1} \beta_j \bfnab  \phi_j \quad \mbox{ on $D \times \{L\}$},  \\
\pi_t \bfH & = \sum_{i\ge 1} \alpha_i \mu_i^{1/2} \bfnab \psi_i + \sum_{j\ge 1} \beta_j \lambda_j^{-1/2} \bfnab \times \phi_j \quad \mbox{ on $D \times \{L\}$},  \\
%\end{align*}
%and 
%\begin{align*}
\|\bfE \|^2_{\bfH(\curl,\Ome)} + \|\bfH\|^2_{\bfH(\curl,\Ome)} &\lesssim \sum_{i\ge 1} |\alpha_i|^2 \mu_i^{1/2} + \sum_{j\ge 1} |\beta_j|^2 \lambda_j^{-1/2}. 
%\\ \|\bfH \|^2_{\bfH(\curl,\Ome)} &\lesssim \sum_i |\alpha_i|^2 \mu_i^{1/2} + \sum_j |\beta_j|^2 \lambda_j^{-1/2}. 
\end{align*}
\end{lemma}
\begin{proof}
We construct the liftings explicitly. Define the functions 
\begin{align*}
\Phi_j(x,y,z)& := \phi_j(x,y) e^{-\sqrt{\lambda_j} (L-z)}, 
&
\Psi_i(x,y,z)& := \psi_i(x,y) e^{-\sqrt{\mu_i} (L-z)}
\end{align*}
and 
\begin{align}
\label{eq:lemma:alphai-betaj-reconstruction-10}
\bfE(x,y,z) &:= \sum_{i\ge 1} \alpha_i \curl (\bfe_z \Psi_i) + \sum_{j \ge 1} \beta_j \bfnab \Phi_j(x,y,z), \\
\label{eq:lemma:alphai-betaj-reconstruction-20}
\bfH(x,y,z) &:= \sum_{i \ge 1} \alpha_i \mu_i^{1/2}\bfnab \Psi_i + \sum_{j\ge 1} \beta_j \lambda_j^{-1/2} \curl (\bfe_z  \Phi_j(x,y,z)). 
\end{align}
Define $N_\alpha:= \sum_{i\ge 1} |\alpha_i|^2 \mu_i^{1/2}$ and $N_\beta:= \sum_{j\ge 1} |\beta_j|^2 \lambda_j^{-1/2}$. 
By orthogonality properties of the functions $\phi_j$, we estimate 
\begin{align*}
\|\sum_{j\ge 1} \beta_j \bfnab \Phi_j\|^2_{L^2(\Ome)} &\lesssim \sum_{j\ge 1} |\beta_j|^2 \left( \|\phi_j\|^2_{L^2(D)} \lambda_j^{+1/2}  + \|\bfnab \phi_j\|^2_{L^2(D)} \lambda_j^{-1/2}\right) 
\sim N_\beta. 
\end{align*}
With $\curl (\bfe_z \Psi_i) = \bfnab \Psi_i \times \bfe_z$ 
we compute (cf.\ Lemma~\ref{lemma:curl-of-product})
\begin{align*}
\curl (\bfe_z \Psi_i ) & = \left(\begin{array}{c} \bfnab \times \psi_i \\ 0 \end{array}\right) e^{-\sqrt{\mu_i} (L-z) }, 
& 
\curl \curl (\bfe_z \Psi_i ) & = \left(\begin{array}{c} \sqrt{\mu_i}\bfnab \psi_i \\ -\Delta \psi_i  \end{array}\right) e^{-\sqrt{\mu_i}( L-z) }, 
\end{align*}
so that by orthogonality properties of the functions $\psi_i$ we estimate 
\begin{align*} 
\|\sum_{i\ge 1} \alpha_i \curl (\bfe_z \Psi_i)\|^2_{L^2(\Ome)} & \lesssim \sum_{i\ge 1} |\alpha_i|^2 \mu_i^{-1/2}, \\ 
\|\curl \sum_{i\ge 1} \alpha_i \curl (\bfe_z \Psi_i)\|^2_{L^2(\Ome)} & \lesssim \sum_{i\ge 1} |\alpha_i|^2 \mu_i^{-1/2} \left(  \mu_i \|\bfnab \psi_i\|^2_{L^2(D)} + \mu_i^2 \|\psi_i\|^2_{L^2(D)} \right) \\
 & \lesssim \sum_{i\ge 1} |\alpha_i|^2 \mu_i^{1/2} \sim N_\alpha. 
\end{align*}
That is, the series defining the function $\bfE \in \bfH(\curl,\Ome)$ in (\ref{eq:lemma:alphai-betaj-reconstruction-10})
converges in $\bfH(\curl,\Ome)$ so that $\bfE \in \bfH(\curl,\Ome)$. 
Since $\Psi_i|_{D \times \{L\}} = \psi_i$ and $\Phi_j|_{D \times \{L\}} = \phi_j$ the tangential component of $\bfE$ on $D \times \{L\}$ coincides with the given sum.
Since the functions $\phi_j \in H^1_0(D)$, the term $\gamma_t \sum_{j\ge 1} \beta_j \bfnab \Phi_j$ vanishes on $\ptl D \times (0,L)$. For the term 
$\sum_{i\ge 1} \alpha_i \curl (\bfe_z \Psi_i)$, we observe that the functions $\psi_i$ are smooth near $\partial D\setminus {\mathcal V}$, where ${\mathcal V}$ is the set of vertices of $D$
so that $\partial_n \psi_i = 0$ pointwise on $\partial D \setminus \{{\mathcal V}\}$. Therefore, $\gamma_t \curl (\bfe_z \Psi_i)$ has a classical trace on 
$(D \setminus{\mathcal V}) \times (0,L)$ and vanishes there. Hence, $\gamma_t \bfE = 0$ on $\partial D \times (0,L)$. 

The arguments for the function $\bfH$ are similar. By the orthogonality properties of the functions $\psi_i$ we estimate
\begin{align*}
\|\sum_{i \ge 1}\alpha_i \mu_i^{1/2} \bfnab \Psi_i\|^2_{L^2(\Ome)} &\lesssim \sum_{i\ge 1} |\alpha_i|^2 \mu_i^{1-1/2} \left(\|\bfnab \psi_i\|^2_{L^2(D)} + \mu_i^1 \|\psi_i\|^2_{L^2(D)}\right) \\
& \sim N_\alpha , \\
\|\sum_{j\ge 1} \beta_j \lambda_j^{-1/2} \curl (\bfe_z \Phi_i)\|^2_{L^2(\Ome)} &\lesssim \sum_{j\ge 1} |\beta_j|^2 \lambda_j^{-1-1/2}  \|\bfnab \phi_j\|^2_{L^2(D)}, \\
\|\curl \sum_{j\ge 1} \beta_j \lambda_j^{-1/2} \curl (\bfe_z \Phi_i)\|^2_{L^2(\Ome)} &\lesssim \sum_{j \ge 1} |\beta_j|^2 \lambda_j^{-1-1/2}  \left( \lambda_j \|\bfnab \phi_j\|^2_{L^2(D)} + \lambda_j^2 \|\phi_j\|^2_{L^2(D)} \right) \\
& \sim N_\beta. 
\end{align*}
That is, the series defining the function $\bfH \in \bfH(\curl,\Ome)$ 
in (\ref{eq:lemma:alphai-betaj-reconstruction-20})
converges in $\bfH(\curl,\Ome)$ so that $\bfH \in \bfH(\curl,\Ome)$. 
Since $\Psi_i|_{D \times \{L\}} = \psi_i$ and $\Phi_j|_{D \times \{L\}} = \phi_j$ the tangential component of $\bfH$ on $D \times \{L\}$ coincides with the given sum.
\end{proof}
%On $\bfH_{0,\Gammam}(\curl,\Ome)$ we define the operator $\dtnm$ by 
%\begin{align}
%\dtnm \bfE &:= \sum_i -\frac{\widetilde{\mu}_i}{i \ome} \alphahat(\bfE) \bfnab \psi_i + \sum_j \frac{i \ome}{\widetilde{\lambda}_j} \betahat(\bfE) \bfnab \phi_j. 
%\end{align}
%with $\widetilde\mu_i$, $\widetilde\lambda_j$ given by (\ref{eq:widetilde-mu_i}), (\ref{eq:widetilde-lambda_j}). 
%This operator actually coincides with the operator $\dtnm$ of (\ref{eq:dtnm}) if the coefficients functions $\alpha_i$, $\beta_j$ are in $H^1(0,L)$ so that  
%$\alpha_i(L) = \alphahat(\bfE)$ and $\beta_j(L) = \betahat(\bfE)$.  
%Motivated by Remark~\ref{rem:alternative-to-impedance-bc}, we define the 
%operator $\dtnm:\bfH_{0,\Gammam}(\curl,\Ome) \rightarrow \bfH^{-1/2}(\Gammao)$ as 
%\begin{align}
%\dtnm \bfE &:= \sum_i -\frac{\widetilde{\mu}_i}{i \ome} \alpha_i \bfnab \psi_i + \sum_j \frac{i \ome}{\widetilde{\lambda}_j} \beta_j \bfnab \phi_j
%\end{align}
%where the coefficients $\alpha_i = \alphahat(\bfE)$ and $\beta_j = \betahat(\bfE)$ are given by (\ref{eq:alphahat}), (\ref{eq:betahat}). 
\begin{theorem}
\label{thm:maxwell-dtn} 
The operator $\dtnm$ of (\ref{eq:dtnm}) is a continuous linear operator 
$\pi_t (\bfH_{0,\Gammam}(\curl,\Ome)) \rightarrow (\gamma_t (\bfH_{0,\Gammam}(\curl,\Ome)))^\prime$ and 
\begin{align}
\label{eq:thm:maxwell-dtn-1} 
\left| \langle \dtnm \bfE,(\gamma_t \bfG)\rangle \right| 
\leq C \|\bfE\|_{\bfH(\curl,\Ome)} \|\bfG\|_{\bfH(\curl,\Ome)}
\  \forall \bfE, \bfG \in \bfH_{0,\Gammam}(\curl,\Ome). 
\end{align}
Furthermore, upon expanding $\bfE$, $\bfG \in \bfH_{0,\Gammam}(\curl,\Ome)$ on $\Gammao$ as 
$\pi_t \bfE = \sum_{i\ge 1} \alphahat(\bfE) \bfnab\times \psi_i + \sum_{j\ge 1} \betahat(\bfE) \bfnab \phi_j$ 
and $\pi_t \bfG = \sum_{i\ge 1} \alphahat(\bfG) \bfnab\times \psi_i + \sum_{j\ge 1} \betahat(\bfG) \bfnab \phi_j$, we have 
\begin{align}
\label{eq:thm:maxwell-dtn-2} 
\langle \dtnm \bfE, \gamma_t \bfG\rangle & 
 = \sum_{i\ge 1} \frac{\widetilde\mu_i}{\bi \ome} \alphahat(\bfE) \alphahat(\bfG) + \sum_{j \ge 1} \frac{\bi \ome}{\widetilde\lambda_j} \betahat(\bfE) \betahat(\bfG). 
\end{align}
\end{theorem}
\begin{proof}
\emph{Proof of (\ref{eq:thm:maxwell-dtn-1}):}
The map $(\bfE,\bfG) \mapsto \langle \dtnm \bfE, \gamma_t \bfG \rangle$ is a continuous sequilinear form in view of Theorem~\ref{thm:alphai-betaj} and Lemma~\ref{lemma:alphai-betaj-reconstruction}. Indeed, 
by Theorem~\ref{thm:alphai-betaj} we have 
$\sum_{i\ge 1} |\alpha_i|^2 \mu_i^{1/2} + \sum_{j\ge 1} |\beta_j|^2 \lambda_j^{-1/2} \lesssim \|\bfE \|^2_{\bfH(\curl,\Ome)}$. 
%By Lemma~\ref{lemma:alphai-betaj-reconstruction}, 
On $\Gammao$ we have $\pi_t \bfE = \sum_{i\ge 1} \alpha_i \bfnab \times \psi_i + \sum_{j\ge 1} \beta_j \bfnab \phi_j$ and, 
since $|\widetilde{\mu}_i| \sim \mu_i^{1/2}$ as well as 
$|\widetilde{\lambda}_j| \sim \lambda_j^{1/2}$, we have 
by Lemma~\ref{lemma:alphai-betaj-reconstruction} 
that the sum defining $\dtnm \bfE$ is the tangential component on $\Gammao$ of a 
function $\bfH \in \bfH(\curl,\Ome)$. 
For smooth $\bfG$ vanishing in a neighborhood of 
$\Gammam$ we therefore get with this function $\bfH$  
\begin{align*}
\langle \dtnm \bfE, \gamma_t \bfG\rangle & = 
\langle \pi_t \bfH, \gamma_t \bfG\rangle  = 
(\bfH,\curl \bfG)_{L^2(\Ome)} - (\curl \bfH, \bfG)_{L^2(\Ome)}, 
\end{align*}
which shows the desired boundedness assertion 
for $(\bfE, \bfG) \mapsto \langle \dtnm \bfE, \gamma_t \bfG \rangle$.  

\emph{Proof of (\ref{eq:thm:maxwell-dtn-2}):} We may assume that $\bfE$, $\bfG$ are smooth and vanish in a neighborhood of $\overline{\Gammam}$. 
Then, the expansions 
\begin{align*}
(\pi_t \bfE)|_{\Gammao} & = \sum_{i\ge 1} \alphahat(\bfE) \bfnab \times \psi_i + \sum_{j\ge 1} \betahat(\bfE) \bfnab \phi_j, \\
(\pi_t \bfG)|_{\Gammao} & = \sum_{i\ge 1} \alphahat(\bfG) \bfnab \times \psi_i + \sum_{j\ge 1} \betahat(\bfG) \bfnab \phi_j
\end{align*}
converge in $L^2(\Gammao)$. In fact, the sequences 
$(\alphahat(\bfE))_i$, 
$(\betahat(\bfE))_j$, 
$(\alphahat(\bfG))_i$, 
$(\betahat(\bfG))_j$, 
are also controlled in weighted $\ell^2$-spaces by Theorem~\ref{thm:alphai-betaj}. 
From $\gamma_t \bfG = (\pi_t \bfG) \times \bfn$ we get with the identities (\ref{eq:identities})
\begin{align*}
(\gamma_t \bfG)|_{\Gammao} & = - \sum_{i\ge 1} \alphahat(\bfG) \bfnab \psi_i + \sum_{j\ge 1} \betahat(\bfG) \bfnab \times \phi_j. 
\end{align*}
In view of orthogonalities, we have 
\begin{align*}
(\dtnm \bfE, \gamma_t \bfG)_{L^2(\Gammao)} & = \sum_{i\ge 1} \frac{\widetilde\mu_i}{\bi \ome} \alphahat(\bfE) \alphahat(\bfG) + \sum_{j\ge 1} \frac{\bi \ome}{\widetilde\lambda_j} \betahat(\bfE) \betahat(\bfG). 
\end{align*}
\end{proof}
%-------------------------------------------------------------------------
\subsection{Reconstruction}
%-------------------------------------------------------------------------
The solution of the ODE-system (\ref{eq:maxwell-ODE}) yields a solution
of problem (\ref{eq:primal-maxwell}). To see that, we have to show 
the functions $\bfE$, $\bfH$ defined by (\ref{eq:ansatz}) are actually
in $\bfH(\curl,\Ome)$: 
\begin{lemma}
\label{lemma:reconstruction}
Let $\bff$, $\bfg \in L^2(\Ome)$, and let the functions $\alpha_i$, $\beta_i$, $\gamma_i$, $\delta_i$, $\eta_i$, $\zeta_i$
be defined by the ODE system (\ref{eq:maxwell-ODE}) together with the ``initial'' conditions (\ref{eq:initial-conditions}) and the 
``endpoint'' conditions 
(\ref{eq:dtn-condition-maxwell}). 
%(\ref{eq:impedance-bc}). 
Then the functions 
\begin{align*}
\bfE& :=
\sum_{i\ge 1} \alpha_i(z) \left(\begin{array}{c} \bfnab \times \psi_i \\ 0 \end{array}\right) 
+ \sum_{j\ge 1}  \beta_j(z) \left(\begin{array}{c} \bfnab \phi_j \\ 0\end{array}\right) + \sum_{j\ge 1} \gamma_j(z) \left(\begin{array}{c} 0 \\ \phi_j\end{array}\right), \\
\bfH& :=\sum_{i\ge 1} \delta_i(z) \left(\begin{array}{c} \bfnab \psi_i \\ 0 \end{array}\right) + \sum_{j\ge 1}  \eta_j(z) \left(\begin{array}{c} \bfnab \times \phi_j \\ 0\end{array}\right)+ \sum_{i\ge 1} \zeta_i(z) \left(\begin{array}{c} 0 \\ \psi_i \end{array}\right)
\end{align*}
are in $\bfH(\curl,\Ome)$ and $\|\bfE\|_{\bfH(\curl,\Ome)} + \|\bfH\|_{\bfH(\curl,\Ome)} \leq C \left( \|\bff\|_{L^2(\Ome)} + \|\bfg\|_{L^2(\Ome)}\right)$. 
Additionally, $\bfE \in \bfH_{0,\Gammam}(\Ome)$. 
\end{lemma}
\begin{proof}
\emph{Step 1:} We use the notation of (\ref{eq:maxwell-ODE}) and observe that $\bff$, $\bfg \in L^2(\Ome)$ implies that 
\begin{align*}
\sum_{i\ge 1} \|f_{i,1}\|^2_{L^2(0,L)} + \|f_{i,2}\|^2_{L^2(0,L)} + \mu_i \|f_{i,3}\|^2_{L^2(0,L)} & = \|\bff\|^2_{L^2(\Ome)},  \\
\sum_{i\ge 1} \|g_{i,1}\|^2_{L^2(0,L)} + \|g_{i,2}\|^2_{L^2(0,L)} + \lambda_i \|g_{i,3}\|^2_{L^2(0,L)} & = \|\bfg\|^2_{L^2(\Ome)}. 
\end{align*} 
We abbreviate 
$\displaystyle 
M_{\bff, \bfg}:= \|\bff\|^2_{L^2(\Omega)} + \|\bfg \|^2_{L^2(\Omega)}. 
$
From Lemmas~\ref{lemma:alpha-delta-zeta} and \ref{lemma:beta-eta-gamma}
we have 
\begin{align}
\label{eq:lemma:reconstruction-10}
\sum_{i\ge 1} \|\alpha^\prime_i\|^2_{L^2(0,L)} + \mu_i \|\alpha_i\|^2_{L^2(0,L)} + \|\delta_i\|^2_{L^2(0,L)} + \mu_i^{-1} \|\zeta_i\|^2_{L^2(0,L)} 
%&\lesssim \|\bff\|^2_{L^2(\Ome)} + \|\bfg\|^2_{L^2(\Ome)}, \\
&\lesssim M_{\bff,\bfg}, \\
\label{eq:lemma:reconstruction-20}
\sum_{j\ge 1} \|\beta_j\|^2_{L^2(0,L)} + \lambda_j \|\eta_j\|^2_{L^2(0,L)} + \lambda_j^{-1} \|\gamma_j\|^2_{L^2(0,L)} 
&\lesssim M_{\bff,\bfg}.
%&\lesssim \|\bff\|^2_{L^2(\Ome)} + \|\bfg\|^2_{L^2(\Ome)}. 
\end{align}
These bounds together with (\ref{eq:maxwell-2}), (\ref{eq:maxwell-4}), (\ref{eq:maxwell-5}) provide
\begin{align}
\label{eq:lemma:reconstruction-30}
 \sum_{j\ge 1} \|-\beta^\prime_j + \gamma_j\|^2_{L^2(0,L)} &\lesssim \sum_{j\ge 1} \|f_{j,2}\|^2_{L^2(0,L)} + \|\eta_j\|^2_{L^2(0,L)} 
\lesssim M_{\bff,\bfg}, \\
%\lesssim \|\bff\|^2_{L^2(\Ome)} + \|\bfg\|^2_{L^2(\Ome)}, \\
\label{eq:lemma:reconstruction-40}
 \sum_{i\ge 1} \|-\delta^\prime_i + \zeta_i\|^2_{L^2(0,L)} &\lesssim \sum_{i\ge 1} \|g_{i,1}\|^2_{L^2(0,L)} + \|\alpha_i\|^2_{L^2(0,L)} 
\lesssim M_{\bff,\bfg}, \\
%\lesssim \|\bff\|^2_{L^2(\Ome)} + \|\bfg\|^2_{L^2(\Ome)},  \\
\label{eq:lemma:reconstruction-45}
 \sum_{i\ge 1} \|\eta^\prime_i \|^2_{L^2(0,L)} &\lesssim \sum_{i\ge 1} \|g_{i,2}\|^2_{L^2(0,L)} + \|\beta_i\|^2_{L^2(0,L)} 
\lesssim M_{\bff,\bfg}. 
%\lesssim \|\bff\|^2_{L^2(\Ome)} + \|\bfg\|^2_{L^2(\Ome)}. 
\end{align}

\emph{Step 2:} The orthogonality properties and (\ref{eq:lemma:reconstruction-10}), (\ref{eq:lemma:reconstruction-20}) lead to 
\begin{align*} 
\|\bfE\|^2_{L^2(\Ome)} & = \sum_{i\ge 1} \|\alpha_i\|^2_{L^2(0,L)} + \sum_{j\ge 1} \|\beta_j\|^2_{L^2(0,L)} + \sum_{i\ge 1} \lambda_i^{-1} \|\gamma_i\|^2_{L^2(0,L)} 
\lesssim M_{\bff,\bfg},  \\
%\lesssim \|\bff\|^2_{L^2(\Ome)} + \|\bfg\|^2_{L^2(\Ome)}, \\
\|\bfH\|^2_{L^2(\Ome)} & = \sum_{i\ge 1} \|\delta_i\|^2_{L^2(0,L)} + \sum_{j\ge 1} \|\eta_j\|^2_{L^2(0,L)} + \sum_{j\ge 1} \mu_j^{-1} \|\zeta_i\|^2_{L^2(0,L)} 
\lesssim M_{\bff,\bfg}. 
%\lesssim \|\bff\|^2_{L^2(\Ome)} + \|\bfg\|^2_{L^2(\Ome)}. 
\end{align*} 

\emph{Step 3:} From Lemma~\ref{lemma:curl-of-product} we conclude in the sense of distributions
\begin{subequations}
\label{eq:control-curlE-curlH}
\begin{align}
\curl \bfE & = \sum_{i\ge 1} \alpha^\prime_i \left( \begin{array}{c} \bfnab \psi_i \\ 0 \end{array}\right)  + 
\sum_{j\ge 1} (-\beta^\prime_j + \gamma_j) \left(\begin{array}{c} \bfnab \times \phi_j \\ 0 \end{array}\right) + 
\sum_{i\ge 1} \alpha_i \bfe_z \mu_i \psi_i, \\
\curl \bfH & = \sum_{i\ge 1} (-\delta^\prime_i + \zeta_i) \left( \begin{array}{c} \bfnab \times \psi_i \\ 0 \end{array}\right)  + 
\sum_{j\ge 1} \eta^\prime_j  \left(\begin{array}{c} \bfnab \phi_j \\ 0 \end{array}\right) + 
\sum_{i\ge 1} \eta_i \bfe_z \lambda_i \phi_i. 
\end{align}
\end{subequations}
To see $\curl \bfE$, $\curl \bfH \in L^2(\Ome)$, we have to control 
\begin{align*}
\sum_{i\ge 1} \|\alpha_i^\prime\|^2_{L^2(0,L)} + 
\sum_{j\ge 1} \|-\beta_j^\prime+\gamma_j\|^2_{L^2(0,L)} + 
\sum_{i\ge 1} \mu_i \|\alpha_i\|^2_{L^2(0,L)}, \\
\sum_{i\ge 1} \|\eta_i^\prime\|^2_{L^2(0,L)} + 
\sum_{j\ge 1} \|-\delta^\prime_j+ \zeta_j\|^2_{L^2(0,L)} + 
\sum_{i\ge 1} \lambda_i \|\eta_i\|^2_{L^2(0,L)}. 
\end{align*}
Both terms are controlled by $M_{\bff,\bfg} = \|\bff\|^2_{L^2(\Ome)} + \|\bfg\|^2_{L^2(\Ome)}$: for the first one, one uses  
(\ref{eq:lemma:reconstruction-10}), (\ref{eq:lemma:reconstruction-30}) and for the second one one appeals to 
(\ref{eq:lemma:reconstruction-20}), (\ref{eq:lemma:reconstruction-40}), (\ref{eq:lemma:reconstruction-45}). This shows that the curl of partial sums 
of $\bfE$ and $\bfH$ can be controlled in $L^2$ and therefore 
$\bfE$, $\bfH \in \bfH(\curl,\Ome)$. 

To see that $\bfE \in \bfH_{0,\Gammam}(\Ome)$, it suffices to assert, in view of \cite[Prop.~{3.3}]{fernandes-gilardi97}
that $\pi_t \bfE = 0$ on $D \times \{0\}$ and $\partial D \times (0,L)$. For $z = 0$, we note that  $\alpha_i(0) = \beta_j(0) = 0$
implies $\pi_t \bfE = 0$ on $z = 0$. To see that $\bfE$ vanishes on the lateral side $\Gammal$, we use that for the partial
sums $\bfE_N:= \sum_{i=1}^N \alpha_i \bfe_{1,i} + \beta_i \bfe_{2,i} + \gamma_i \bfe_{3,i}$ have that $\gamma_t \bfE_N = 0$ 
on the pieces of $D \times (0,L)$ and therefore, by \cite[Prop.~{3.3}]{fernandes-gilardi97}, also in $\left(H^{1/2}(\Gammal)\right)^\prime$.  
The convergence of the series in $\bfH(\curl,\Ome)$ then implies $\gamma_t \bfE = 0$ on $\Gammal$. 
\end{proof}
\section{Stability analysis of the adjoint operator for the Helmholtz problem (\ref{eq:primal-helmholtz})}
Our goal is the proof of Theorem~\ref{thm:stability-helmholtz-adjoint}, which parallels Theorem~\ref{thm:stability-helmholtz}.

%--------------------------------------------
\subsection{The ultraweak formulation of (\ref{eq:primal-helmholtz})}
%--------------------------------------------
The operator $A$ is given by
\begin{align*}
A \left(\begin{array}{c} p \\ \bfu \end{array}\right)
 = \left(
\begin{array}{c}
\bi \ome \bfu + \fraka \nabla p \\ \bi \ome p + \div \bfu
\end{array}
\right)
\end{align*}
where $D(A)$ includes the boundary conditions%\footnote{note: need to check wether we want $H^{-1/2}(\Gammao)$ instead of $\widetilde{H}^{-1/2}(\Gammao)$.
%Also: need to deal with $\fraka$ in the definition of $\Hdiv$}
\begin{align*}
D(A) & = \{(p,\bfu) \in H^1(\Omega) \times \bfHdiv \colon \\& \qquad  p|_{\Gammai} = 0, \quad \bfu \cdot \bfn = 0 \mbox{ in $H^{-1/2}(\Gammal)$},\quad
\bi\ome \bfu \cdot \bfn + \dtn p = 0 \quad \mbox{ in $\widetilde{H}^{-1/2}(\Gammao)$}\}.
\end{align*}
%\begin{remark}[(sanity check for the boundary conditions)]
%$\bfu \cdot \bfn \in H^{-1/2}(\Gammal)$ makes sense for $\bfu \in \bfHdiv$ since the test functions are in $\widetilde{H}^{1/2}(\Gammal) \subset H^{1/2}(\partial\Omega)$.
%We also have $\bfu \cdot \bfn \in \widetilde{H}^{-1/2}(\Gammao)$: we have $\bfu \cdot \bfn \in H^{-1/2}(\partial\Omega)$. For an arbitrary $\phi \in H^{1/2}(\Gammao)$,
%we let $\widetilde\phi \in H^{1/2}(\partial\Omega)$ by an extension to $\partial\Omega$
%and set $\langle \bfu \cdot \bfn ,\phi\rangle_{\widetilde{H}^{-1/2}(\Gammao) \times H^{1/2}(\Gammao)}:= \langle \bfu \cdot \bfn, \widetilde \phi   \rangle$.
%This is well-defined since the difference of two extensions $\widetilde\phi_1$, $\widetilde\phi_2$ satisfies $\widetilde\phi_1 - \widetilde\phi_2 \in
%\widetilde{H}^{1/2}(\partial\Omega\setminus \overline{\Gammao})$ (we assume, as we may, that the extensions vanish on $\Gammai$.) From $\bfu \cdot \bfn = 0$
%in $H^{-1/2}(\Gammal)$, we conclude $\langle \bfu  \cdot \bfn, \widetilde\phi_1 - \widetilde\phi_2\rangle = 0$.
%
%\todo{insert also the references to \cite[Thm.~{3.29}]{mclean00}. }
%\end{remark}
The adjoint operator is\footnote{we use that $\fraka$ is real-valued}
\begin{align}
A^\ast \left(\begin{array}{c} q \\ \bfv \end{array}\right) =
\left(
\begin{array}{c}
-\bi \ome \bfv - \nabla q  \\ -\div (\fraka \bfv) - \bi \ome q
\end{array}
\right)
\end{align}
and the domain is
\begin{align*}
D(A^\ast)  = \{(q,\bfv) \,|\, & q \in H^1(\Omega), \fraka \bfv \in \bfHdiv, \\ & q|_{\Gammai} = 0, \quad (\fraka \bfv) \cdot \bfn = 0 \quad \mbox{ in $H^{-1/2}(\Gammal)$},
\quad \bi \ome \fraka \bfv \cdot \bfn + \dtn^\ast q =0 \mbox{ in $\widetilde{H}^{-1/2}(\Gammao)$}\}.
\end{align*}

We recall that the operator $\dtn$ is defined in Section~\ref{sec:dtn} in terms of the
eigenpair $(\phi_n,\lambda_n)_{n \in \doubleIN}$ of (\ref{eq:helmholtz-eigenpairs}) and that
the values $\kappa_n$ are defined in (\ref{eq:helmholtz-kappa_n}).

We also record that the computation shows that the adjoint of the operator $\dtn$ is given by
\begin{align*}
\langle p, \dtn^\ast q\rangle &\stackrel{\text{def}}{=}
\langle \dtn p, q\rangle
= -\sum_{n} p_n \overline{\overline{\kappa_n} q_n}
\end{align*}
so that
\begin{align}
\label{eq:dtnstar}
\dtn^\ast q &= -\sum_{n} \overline{\kappa}_n q_n
\end{align}

%-------------------------------------
\subsection{Stability estimates for \texorpdfstring{$A^\ast$}{A*}}
%-------------------------------------
Introduce $H^1_{\Gammai}:= \{v \in H^1(\Omega)\,|\, v|_{\Gammai} = 0\}$. In the following, we will need the following observation
about the sequences $\{\lambda_n\}_n$  and $\{\kappa_n\}_n$:
Noting that there are only finitely many propagating modes and that we assumed (\ref{eq:non-degeneracy}), we have
\begin{subequations}
\label{eq:lambda-tildelambda-adjoint}
\begin{align}
\label{eq:lambda-tildelambda-a-adjoint}
\max_{n \in \IP} \left( |\kappa_n|^{-1} + |\kappa_n|\right) & \leq C  ,\\
\label{eq:lambda-tildelambda-b-adjoint}
\max_{n \in \IP} |\sqrt{\lambda_n}| & \leq C  ,\\
\label{eq:lambda-tildelambda-c-adjoint}
\max_{n \in \doubleIN} \frac{|\sqrt{\lambda_n}|}{|\kappa_n|} & \leq C
\end{align}
\end{subequations}
for a constant that depends on $\ome$.

\begin{lemma}
\label{lemma:L2-estimate-adjoint}
The solution $q \in H^1_{\Gammai}$ of
\begin{align}
\label{eq:lemma:L2-estimate-weak-form-adjoint}
(\nabla v, \fraka \nabla q)_{L^2(\Ome)} - \ome^2 (v,q)_{L^2(\Ome)} - \langle  v, \dtn^\ast q\rangle = (v,f)_{L^2(\Ome)}
\quad \forall v\in H^1_{\Gammai}
\end{align}
satisfies
\begin{align*}
\|q\|_{H^1(\Ome)} \leq C L \|f\|_{L^2(\Ome)}.
\end{align*}
\end{lemma}
\begin{proof}
We make the ansatz
$$
q(x,z) = \sum_n q_n(z) \phi_n(x)
$$
and set
$$
f_n(z):= (\phi_n,f(\cdot,z))_{L^2(D)}.
$$
By orthogonality properties of the functions $\{\phi_n\}_n$ we have
\begin{align*}
\|q\|^2_{L^2(\Ome)} & = \sum_{n} \|q_n\|^2_{L^2(0,L)}, \\
\|\sqrt{a} \nabla_x q\|^2_{L^2(\Ome)} & = \sum_{n} \lambda^2_n \|q_n\|^2_{L^2(0,L)}, \\
\| \partial_z q\|^2_{L^2(\Ome)} & = \sum_{n}  \|q^\prime_n\|^2_{L^2(0,L)}, \\
\|f\|^2_{L^2(\Ome)} & = \sum_{n} \|f_n\|^2_{L^2(0,L)}.
\end{align*}
Testing (\ref{eq:lemma:L2-estimate-weak-form-adjoint}) with $v(x,z) = v_n(z) \phi_n(x)$ with arbitrary $v_n \in H^1_{(0}(0,L)$
gives due to the orthogonalities satisfied by the functions $\{\phi_n\}_n$
$$
(v^\prime_n,q^\prime_n)_{L^2(0,L)} + {{\kappa^2_n}} (v_n,q_n)_{L^2(0,L)} + v_n(L) {\kappa_n} \overline{q}_n(L) = (v_n,f_n)_{L^2(0,L)}
$$
From Lemma~\ref{lemma:ihl}, we conclude
\begin{align*}
\|q_n\|_{1,|\kappa_n|} \leq C
\begin{cases}
L \|f_n\|_{L^2(0,L)} & \mbox{ if $n \in \IP$}, \\
\kappa_n^{-1} \|f_n\|_{L^2(0,L)} & \mbox{ if $n \in \IE$}.
\end{cases}
\end{align*}
We arrive at
\begin{align*}
\|\sqrt{a}\nabla_x q\|^2_{L^2(\Ome)} &= \sum_{n} \lambda^2_n \|q_n\|^2_{L^2(0,L)}
\stackrel{(\ref{eq:lambda-tildelambda-adjoint})}{\lesssim}  L^2 \sum_{n\in\IP} \frac{\lambda^2_n}{|{\kappa}_n|^2} \|f_n\|^2_{L^2(0,L)}
+ \sum_{n \in \IE} \frac{\lambda^2_n}{|{\kappa}_n|^4} \|f_n\|^2_{L^2(0,L)}
\stackrel{(\ref{eq:lambda-tildelambda-adjoint})}{\lesssim} L^2 \|f\|^2_{L^2(\Ome)}, \\
\|\partial_z q\|^2_{L^2(\Ome)} & = \sum_{n} \|q^\prime_n\|^2_{L^2(0,L)} \lesssim L^2 \sum_{n\in\IP}  \|f_n\|^2_{L^2(0,L)}
+ \sum_{n \in \IE} |{\kappa}_n|^{-2} \|f_n\|^2_{L^2(0,L)}
\lesssim L^2 \|f\|^2_{L^2(\Ome)}, \\
\|q\|^2_{L^2(\Ome)} & = \sum_{n} \|q_n\|^2_{L^2(0,L)} \lesssim
L^2 \sum_{n\in\IP} |{\kappa_n}|^{-2} \|f_n\|^2_{L^2(0,L)} + 
\sum_{n\in\IE} |{\kappa_n}|^{-4} \|f_n\|^2_{L^2(0,L)}  
\lesssim  L^2\|f\|^2_{L^2(\Ome)}. 
\end{align*}
\end{proof}
\begin{lemma}
\label{lemma:L2div-estimate-adjoint}
The solution $q \in H^1_{\Gammai}$ of
\begin{align}
\label{eq:lemma:L2div-estimate-weak-form-adjoint}
(\nabla v, \fraka \nabla q)_{L^2(\Ome)} - \ome^2 (v,q)_{L^2(\Ome)} - \langle  v, \dtn^\ast q\rangle = (\nabla v,\bff)_{L^2(\Ome)}
\quad \forall v\in H^1_{\Gammai}
\end{align}
satisfies
\begin{align*}
\|q\|_{H^1(\Ome)} \leq C L \|\bff\|_{L^2(\Ome)}.
\end{align*}
\end{lemma}
\begin{proof}
We write the vector $\bff$ as $\bff = (\bff_x,f_z)^\top$ with a vector-valued function $\bff_x$ and a scalar function $f_z$. By linearity of the problem,
we may consider the cases $(\bff_x,0)^\top$ and $(0,f_z)^\top$ as right-hand sides separately.
For $\bff_x = 0$, we proceed as in Lemma~\ref{lemma:L2-estimate-adjoint} by writing
$f_z = \sum_{n} f_n(z) \phi_n(x)$ and get with Lemma~\ref{lemma:ihl} for the corresponding functions $q_n$
\begin{align*}
\|q_n\|_{1,|\kappa_n|} & \leq C
\begin{cases}
L |\kappa_n| \|f_n\|_{L^2(0,L)}  & \mbox{ if $n \in \IP$}, \\
 \|f_n\|_{L^2(0,L)}  & \mbox{ if $n \in \IE$}.
\end{cases} \\
&
\leq C L \|f_n\|_{L^2(0,L)}
\end{align*}
since $\max_{n \in \IP} |\kappa_n| \leq C$.
We may repeat the calculations performed in Lemma~\ref{lemma:L2-estimate-adjoint} to establish
\begin{align*}
\|\sqrt{a} \nabla_x q\|^2_{L^2(\Ome)} &=  \sum_{n} \lambda_n \|q_n\|^2_{L^2(0,L)}
\leq C L^2 \sum_{n} \frac{\lambda_n}{|\kappa_n|^2}  \|f_n\|^2_{L^2(0,L)}  \leq C \|f_z\|^2_{L^2(\Ome)}, \\
\|\partial_z q\|^2_{L^2(\Ome)} &=  \sum_{n} \|q^\prime_n\|^2_{L^2(0,L)}
\leq C L^2 \sum_{n}  \|f_n\|^2_{L^2(0,L)}  \leq C L^2 \|f_z\|^2_{L^2(\Ome)}, \\
\|q\|^2_{L^2(\Ome)} &=  \sum_{n} \|q_n\|^2_{L^2(0,L)}
\leq C L^2 \sum_{n}|\kappa_n|^{-2}   \|f_n\|^2_{L^2(0,L)}  \leq C L^2 \|f_z\|^2_{L^2(\Ome)}.
\end{align*}
For the case of the right-hand side $\bff = (\bff_x,0)^\top$, we define
$
f_n(z):= (\nabla \phi_n, \bff_x(\cdot,z))_{L^2(D)}  = (a \nabla \phi_n, a^{-1} \bff_x(\cdot,z))_{L^2(D)}
$
and note by the fact that the functions $\{\|\sqrt{a} \nabla \phi_n\|_{L^2(D)}^{-1} \nabla \phi_n\}_n$ are an orthonormal (with respect to
$(a \cdot, \cdot)_{L^2(D)}$) basis of its span that
$$
\sum_{n} \lambda^{-1}_n \|f_n\|^2_{L^2(0,L)} =
\int_{0}^L \sum_{n} \frac{1}{\|\sqrt{a} \nabla \phi_n\|^2_{L^2(D)}} |(a \nabla \phi_n, a^{-1} \bff_x(\cdot,z))_{L^2(D)}|^2
\leq \|a^{-1} \bff_x\|^2_{L^2(\Ome)}.
$$
We expand the solution $q$ as $q = \sum_{n} q_n(z) \phi_n(x)$. Testing the equation with functions of the form $v_n(z) \phi_n(x)$ yields again
an equation for the coefficients $q_n$:
\begin{equation*}
 \kappa^2_n (v_n,q_n)_{L^2(0,L)} + (v^\prime_n,q^\prime_n)_{L^2(0,L)} + {\kappa}_n v_n(L) \overline{q}_n(L)  = (v_n,f_n)_{L^2(0,L)} \qquad \forall v_n \in H^1_{(0}(0,L)
\end{equation*}
By Lemma~\ref{lemma:ihl} we get
\begin{align*}
\|q_n\|_{1,|{\kappa}_n|} & \leq C
\begin{cases}
L \|f_n\|_{L^2(0,L)} & \mbox{ if $n \in \IP$ }, \\
|\kappa_n|^{-1} \|f_n\|_{L^2(0,L)} & \mbox{ if $n \in \IE$ }.
\end{cases}
\end{align*}
Hence,
\begin{align*}
\|\sqrt{a} \nabla_x q\|^2_{L^2(\Ome)} & = \sum_{n} \lambda_n \|q_n\|^2_{L^2(0,L)}
\leq C L^2 \sum_{n \in \IP} \lambda_n \|f_n\|^2_{L^2(0,L)} + C \sum_{n \in \IE} \frac{\lambda_n}{|\kappa_n|^4} \|f_n\|^2_{L^2(0,L)}
\leq C L^2 \|\bff_x\|^2_{L^2(\Ome)}, \\
\|\partial_z q\|^2_{L^2(\Ome)} & = \sum_{n} \|q^\prime_n\|^2_{L^2(0,L)}
\leq C L^2 \sum_{n \in \IP} \|f_n\|^2_{L^2(0,L)} + C \sum_{n \in \IE} |\kappa_n|^{-2} \|f_n\|^2_{L^2(0,L)}
\leq C \|\bff_x\|^2_{L^2(\Ome)}  \\
\|q\|^2_{L^2(\Ome)} & \sum_{n} \|q_n\|^2_{L^2(0,L)} \leq C L^2 \sum_{n \in \IP} |\kappa_n|^{-2} \|f_n\|^2_{L^2(0,L)}
+ \sum_{n \in \IE} |\kappa_n|^{-4} \|f_n\|^2_{L^2(0,L)} \leq C \|\bff_x\|^2_{L^2(\Ome)}.
\end{align*}
Putting together the above results proves the claim.
\end{proof}
\begin{theorem}
\label{thm:stability-helmholtz-adjoint}
There is a constant $C > 0$ (depending on $\fraka$ and $\ome$) such that for all $(\bfv,q) \in D(A^\ast)$
\begin{equation*}
\|A^\ast (q,\bfv)^\top\|_{L^2(\Ome)} \ge C L^{-1} \|(\bfv,q)\|_{L^2(\Ome)}.
\end{equation*}
\end{theorem}
\begin{proof}
%First, we note that $A^\ast:D(A^\ast) \rightarrow L^2(\Ome)$ is injective. 
%Indeed, $A^\ast(q,\bfv)^\top = 0$ implies
%$\bfv = \bi\ome \nabla q$ and therefore $q \in H^1(\Ome)$ satisfies a homogeneous second-order equation. Together with the
%boundary condition, one checks that $q = 0$ so that also $\bfv = 0$.
%
Abbreviate for the two components of $A^\ast (q,\bfv)^\top$
$$
\bff := -\bi \ome \bfv - \nabla q \in L^2(\Ome),
\qquad f := -\div (\fraka \bfv) -\bi \ome q \in L^2(\Ome).
$$
Hence, $(q,\bfv)$ satisfy for smooth $(\tilde p,\tilde \bfu)$
\begin{align*}
(\tilde \bfu, -\bi \ome \bfv )_{L^2(\Ome)} - (\tilde \bfu, \nabla q)_{L^2(\Ome)} & = (\tilde \bfu, \bff)_{L^2(\Ome)}, \\
(\tilde p, -\div (\fraka \bfv))_{L^2(\Ome)} - (\tilde p, \bi \ome q)_{L^2(\Ome)} & = (\tilde p, f)_{L^2(\ome) }
\end{align*}
Considering $\tilde p$ with $\tilde p|_{\Gammai} = 0$ and using the boundary conditions satisfied by $(q,\bfv)$ (i.e., $(q,\bfu) \in D(A^\ast)$) yields
after an integration by parts
\begin{align*}
(\tilde \bfu, -\bi \ome \bfv )_{L^2(\Ome)} - (\tilde \bfu, \nabla q)_{L^2(\Ome)} & = (\tilde \bfu, \bff)_{L^2(\Ome)}, \\
(\nabla \tilde p, \fraka \bfv)_{L^2(\Ome)} - (\tilde p, \bi \ome q)_{L^2(\Ome)} + \langle\tilde p, \frac{1}{\bi\ome} \dtn^\ast q\rangle_{\Gammao}  & = (\tilde p, f)_{L^2(\Ome) }
\end{align*}
Selecting $\tilde \bfu = \frac{1}{i\ome}\fraka \nabla \tilde p$ yields
\begin{align*}
(\nabla \tilde p, \fraka \bfv )_{L^2(\Ome)} - \frac{1}{\bi\ome} (\fraka \nabla \tilde p, \nabla q)_{L^2(\Ome)} & = \frac{1}{\bi\ome}(\nabla \tilde p, \bff)_{L^2(\Ome)}, \\
(\nabla \tilde p, \fraka \bfv)_{L^2(\Ome)} - (\tilde p, \bi \ome q) + \langle\tilde p, \frac{1}{\bi\ome}\dtn^\ast q\rangle_{\Gammao}  & = (\tilde p, f)_{L^2(\Ome) }
\end{align*}
so that, by subtracting these two equations, we arrive at
\begin{align*}
- \frac{1}{\bi\ome} (\nabla \tilde p, \nabla q)_{L^2(\Ome)} + (\tilde p, \bi \ome q)_{L^2(\Ome)} - \langle \tilde p, \frac{1}{\bi\ome}\dtn^\ast q\rangle_{\Gammao}
& = \frac{1}{\bi\ome} (\nabla \tilde p, \bff)_{L^2(\Ome)} - (\tilde p, f)_{L^2(\Ome)}
\end{align*}
Rearranging terms yields
\begin{align*}
(\nabla \tilde p, \nabla q)_{L^2(\Ome)} -\ome^2  (\tilde p, q)_{L^2(\Ome)} - \langle \tilde p, \dtn^\ast q\rangle_{\Gammao}
& = - (\nabla \tilde p, \bff)_{L^2(\Ome)} +\bi\ome (\tilde p, f)_{L^2(\Ome)}
\end{align*}
From Lemmas~\ref{lemma:L2-estimate-adjoint}, \ref{lemma:L2div-estimate-adjoint} we infer
\begin{align*}
\|q\|_{H^1(\Ome)} &\leq C L \left[ \|\bff\|_{L^2(\Ome)} + \|f\|_{L^2(\Ome)}\right],
\end{align*}
which in turn yields
$$
\|(q,\bfv)\|^2_{L^2(\Ome)} \leq \|q\|^2_{L^2(\Ome)} + \|\bfv\|^2_{L^2(\Ome)}
\leq \|q\|^2_{L^2(\Ome)} + 2 \ome^{-2} \|\nabla q\|^2_{L^2(\Ome)} + 2 \ome^{-2} \|\bff \|^2_{L^2(\Ome)}.
$$
In total, we arrive at $\|(q,\bfv)\|_{L^2(\Ome)} \leq C L \left[\|\bff \|_{L^2(\Ome)} + \|f\|_{L^2(\Ome)} \right]$, i.e.,
$\|(q,\bfv)\|_{L^2(\Ome)} \leq C L \|A^\ast(q,\bfv)^\top\|_{L^2(\Ome)}$, which is the claim.
\end{proof}

%---------------------------------------------
\section*{Acknowledgments}
J.M.~Melenk was supported by a JTO fellowship of the Oden Institute \JMM{of T}he University of Texas at Austin and by the Austrian Science Fund (FWF) under grant F65 ``taming complexity in partial differential systems'' (DOI: \href{https://doi.org/10.55776/F65}{10.55776/F65}). 
For open access purposes, the authors have applied a CC BY public copyright licence. 
%This research was funded in whole or in part by the Austrian Science Fund
%(FWF) [grant DOI]. For open access purposes, the author has applied a CC BY
%public copyright license to any author accepted manuscript version arising
%from this submission.\224
L.~Demkowicz and S.~Henneking were supported with AFOSR grant FA9550-19-1-0237 and NSF award 2103524.
Supercomputing resources were provided by the Texas Advanced Computing Center (TACC) at UT Austin on its \emph{Frontera} system under TACC award DMS22025.
%------------------------------------------------------
\bibliographystyle{plain}
\bibliography{bib/journals-iso4.bib, bib/ref_Demkowicz.bib, bib/ref_DPG.bib, bib/ref_laser.bib}
\end{document}